\documentclass[10pt]{amsart}
\usepackage{}
\usepackage{amssymb}
\usepackage{xcolor}
\usepackage{mathabx}

\usepackage{amsmath}

\usepackage[top=0.8in, bottom=0.8in, left=1.25in, right=1.25in]{geometry}

\input xy
\xyoption{all}
\usepackage{amsfonts}
\usepackage{mathrsfs}
\usepackage{tikz-cd}

\usepackage{graphicx}
\usepackage{amscd,amsbsy,amsthm}
\usepackage[all]{xy}
\usepackage[colorlinks,plainpages,urlcolor=blue]{hyperref}%\usepackage{subfigure},backref
\usepackage{verbatim}

\usepackage{enumitem}

\usepackage[normalem]{ulem}

\usepackage{tikz}
\usetikzlibrary{arrows,calc}
\tikzset{
>=stealth',
help lines/.style={dashed, thick},
axis/.style={<->},
important line/.style={thick},
connection/.style={thick, dotted},
}

\newcommand {\Omit}[1]{}

\usepackage{stackengine,scalerel}

\newcommand{\nc}{\newcommand}
\nc{\rnc}{\renewcommand}
\nc{\bb}[1]{{\mathbb #1}}
\nc{\bbA}{\bb{A}}\nc{\bbB}{\bb{B}}\nc{\bbC}{\bb{C}}\nc{\bbD}{\bb{D}}
\nc{\bbE}{\bb{E}}\nc{\bbF}{\bb{F}}\nc{\bbG}{\bb{G}}\nc{\bbH}{\bb{H}}
\nc{\bbI}{\bb{I}}\nc{\bbJ}{\bb{J}}\nc{\bbK}{\bb{K}}\nc{\bbL}{\bb{L}}
\nc{\bbM}{\bb{M}}\nc{\bbN}{\bb{N}}\nc{\bbO}{\bb{O}}\nc{\bbP}{\bb{P}}
\nc{\bbQ}{\bb{Q}}\nc{\bbR}{\bb{R}}\nc{\bbS}{\bb{S}}\nc{\bbT}{\bb{T}}
\nc{\bbU}{\bb{U}}\nc{\bbV}{\bb{V}}\nc{\bbW}{\bb{W}}\nc{\bbX}{\bb{X}}
\nc{\bbY}{\bb{Y}}\nc{\bbZ}{\bb{Z}}
\nc{\mbf}[1]{{\mathbf #1}}
\nc{\bfA}{\mbf{A}}\nc{\bfB}{\mbf{B}}\nc{\bfC}{\mbf{C}}\nc{\bfD}{\mbf{D}}
\nc{\bfE}{\mbf{E}}\nc{\bfF}{\mbf{F}}\nc{\bfG}{\mbf{G}}\nc{\bfH}{\mbf{H}}
\nc{\bfI}{\mbf{I}}\nc{\bfJ}{\mbf{J}}\nc{\bfK}{\mbf{K}}\nc{\bfL}{\mbf{L}}
\nc{\bfM}{\mbf{M}}\nc{\bfN}{\mbf{N}}\nc{\bfO}{\mbf{O}}\nc{\bfP}{\mbf{P}}
\nc{\bfQ}{\mbf{Q}}\nc{\bfR}{\mbf{R}}\nc{\bfS}{\mbf{S}}\nc{\bfT}{\mbf{T}}
\nc{\bfU}{\mbf{U}}\nc{\bfV}{\mbf{V}}\nc{\bfW}{\mbf{W}}\nc{\bfX}{\mbf{X}}
\nc{\bfY}{\mbf{Y}}\nc{\bfZ}{\mbf{Z}}
\nc{\bfa}{\mbf{a}}\nc{\bfb}{\mbf{b}}\nc{\bfc}{\mbf{c}}\nc{\bfd}{\mbf{d}}
\nc{\bfe}{\mbf{e}}\nc{\bff}{\mbf{f}}\nc{\bfg}{\mbf{g}}\nc{\bfh}{\mbf{h}}
\nc{\bfi}{\mbf{i}}\nc{\bfj}{\mbf{j}}\nc{\bfk}{\mbf{k}}\nc{\bfl}{\mbf{l}}
\nc{\bfm}{\mbf{m}}\nc{\bfn}{\mbf{n}}\nc{\bfo}{\mbf{o}}\nc{\bfp}{\mbf{p}}
\nc{\bfq}{\mbf{q}}\nc{\bfr}{\mbf{r}}\nc{\bfs}{\mbf{s}}\nc{\bft}{\mbf{t}}
\nc{\bfu}{\mbf{u}}\nc{\bfv}{\mbf{v}}\nc{\bfw}{\mbf{w}}\nc{\bfx}{\mbf{x}}
\nc{\bfy}{\mbf{y}}\nc{\bfz}{\mbf{z}}

\nc{\mcal}[1]{{\mathcal #1}}
\nc{\calA}{\mcal{A}}\nc{\calB}{\mcal{B}}\nc{\calC}{\mcal{C}}\nc{\calD}{\mcal{D}}
\nc{\calE}{\mcal{E}} \nc{\calF}{\mcal{F}}\nc{\calG}{\mcal{G}}\nc{\calH}{\mcal{H}}
\nc{\calI}{\mcal{I}}\nc{\calJ}{\mcal{J}}\nc{\calK}{\mcal{K}}\nc{\calL}{\mcal{L}}
\nc{\calM}{\mcal{M}}\nc{\calN}{\mcal{N}}\nc{\calO}{\mcal{O}}\nc{\calP}{\mcal{P}}
\nc{\calQ}{\mcal{Q}}\nc{\calR}{\mcal{R}}\nc{\calS}{\mcal{S}}\nc{\calT}{\mcal{T}}
\nc{\calU}{\mcal{U}}\nc{\calV}{\mcal{V}}\nc{\calW}{\mcal{W}}\nc{\calX}{\mcal{X}}
\nc{\calY}{\mcal{Y}}\nc{\calZ}{\mcal{Z}}
\nc{\fA}{\frak{A}}\nc{\fB}{\frak{B}}\nc{\fC}{\frak{C}} \nc{\fD}{\frak{D}}
\nc{\fE}{\frak{E}}\nc{\fF}{\frak{F}}\nc{\fG}{\frak{G}}\nc{\fH}{\frak{H}}
\nc{\fI}{\frak{I}}\nc{\fJ}{\frak{J}}\nc{\fK}{\frak{K}}\nc{\fL}{\frak{L}}
\nc{\fM}{\frak{M}}\nc{\fN}{\frak{N}}\nc{\fO}{\frak{O}}\nc{\fP}{\frak{P}}
\nc{\fQ}{\frak{Q}}\nc{\fR}{\frak{R}}\nc{\fS}{\frak{S}}\nc{\fT}{\frak{T}}
\nc{\fU}{\frak{U}}\nc{\fV}{\frak{V}}\nc{\fW}{\frak{W}}\nc{\fX}{\frak{X}}
\nc{\fY}{\frak{Y}}\nc{\fZ}{\frak{Z}}
\nc{\fa}{\frak{a}}\nc{\fb}{\frak{b}}\nc{\fc}{\frak{c}} \nc{\fd}{\frak{d}}
\nc{\fe}{\frak{e}}\nc{\fFf}{\frak{f}}\nc{\fg}{\frak{g}}\nc{\fh}{\frak{h}}
\nc{\fri}{\frak{i}}\nc{\fj}{\frak{j}}\nc{\fk}{\frak{k}}\nc{\fl}{\frak{l}}
\nc{\fm}{\frak{m}}\nc{\fn}{\frak{n}}\nc{\fo}{\frak{o}}\nc{\fp}{\frak{p}}
\nc{\fq}{\frak{q}}\nc{\fr}{\frak{r}}\nc{\fs}{\frak{s}}\nc{\ft}{\frak{t}}
\nc{\fu}{\frak{u}}\nc{\fv}{\frak{v}}\nc{\fw}{\frak{w}}\nc{\fx}{\frak{x}}
\nc{\fy}{\frak{y}}\nc{\fz}{\frak{z}}

\newcommand{\bM}{\textbf{M}}

\newtheorem{theorem}{Theorem}[section]

\newtheorem{lemma}[theorem]{Lemma}
\newtheorem{corollary}[theorem]{Corollary}
\newtheorem{prop}[theorem]{Proposition}

\newtheorem{setting}[theorem]{Setting}

\theoremstyle{definition}
\newtheorem{definition}[theorem]{Definition}
\newtheorem{ass}[theorem]{Assumption}
\newtheorem{data}[theorem]{Data}
\newtheorem{example}[theorem]{Example}
\newtheorem{folklore example}[theorem]{Folklore Example}
\newtheorem{remark}[theorem]{Remark}

\DeclareMathOperator{\im}{im} 
\DeclareMathOperator{\codim}{codim} \DeclareMathOperator{\id}{id}
 \DeclareMathOperator{\Sym}{Sym}
\DeclareMathOperator{\ch}{ch} 
 \DeclareMathOperator{\GL}{GL}
\DeclareMathOperator{\Hom}{{Hom}} \DeclareMathOperator{\Tor}{{Tor}}
\DeclareMathOperator{\Ext}{{Ext}}

\DeclareMathOperator{\Hilb}{{Hilb}}

\DeclareMathOperator{\proj}{pr} 
\DeclareMathOperator{\Spec}{{Spec}} \DeclareMathOperator{\tr}{tr}
\DeclareMathOperator{\Aut}{Aut}
 \DeclareMathOperator{\End}{End}

\DeclareMathOperator{\Coh}{Coh}

\DeclareMathOperator{\Bl}{Bl}

\DeclareMathOperator{\Tot}{Tot}

\DeclareMathOperator{\Frac}{Frac}

\DeclareMathOperator{\Stab}{Stab}

\newcommand{\inj}{\hookrightarrow}

\newcommand{\pt}{\text{pt}}

\newcommand{\Z}{\bbZ}
\newcommand{\C}{\bbC}

\newcommand{\Q}{\bbQ}

\DeclareMathOperator{\Crit}{Crit}

\DeclareMathOperator{\bCrit}{\textbf{Crit}}

\newcount\cols
{\catcode`,=\active\catcode`|=\active
 \gdef\Young(#1){\hbox{$\vcenter
 {\mathcode`,="8000\mathcode`|="8000
  \def,{\global\advance\cols by 1 &}%\fBun
  \def|{\cr
        \multispan{\the\cols}\hrulefill\cr
        &\global\cols=2 }%
  \offinterlineskip\everycr{}\tabskip=0pt
  \dimen0=\ht\strutbox \advance\dimen0 by \dp\strutbox
  \halign
   {\vrule height \ht\strutbox depth \dp\strutbox##
    &&\hbox to \dimen0{\hss$##$\hss}\vrule\cr
    \noalign{\hrule}&\global\cols=2 #1\crcr
    \multispan{\the\cols}\hrulefill\cr%
   }}$}}}

\newcommand{\yl}[1]{\textcolor{blue}{$[$ Yalong: #1 $]$}}

\newcommand{\zz}[1]{\textcolor{orange}{$[$ Zijun: #1 $]$}}

\DeclareFontFamily{U}{rsfs}{%
\skewchar\font127}
\DeclareFontShape{U}{rsfs}{m}{n}{%
<-6>rsfs5<6-8.5>rsfs7<8.5->rsfs10}{}
\DeclareSymbolFont{rsfs}{U}{rsfs}{m}{n}
\DeclareSymbolFontAlphabet
{\mathrsfs}{rsfs}
\DeclareRobustCommand*\rsfs{%
\@fontswitch\relax\mathrsfs}

\newdimen\argwidth
\def\db[#1\db]{
 \setbox0=\hbox{$#1$}\argwidth=\wd0
 \setbox0=\hbox{$\left[\box0\right]$}
  \advance\argwidth by -\wd0
 \left[\kern.3\argwidth\box0 \kern.3\argwidth\right]}

\newcommand{\cC}{\mathcal{C}}

\newcommand{\oO}{\mathcal{O}}

\newcommand{\Supp}{\mathop{\rm Supp}\nolimits}

\newcommand{\dR}{\mathbf{R}}

\newcommand{\Pic}{\mathop{\rm Pic}\nolimits}

\newcommand{\rk}{\mathop{\rm rk}\nolimits}
\newcommand{\td}{\mathop{\rm td}\nolimits}
\newcommand{\QCoh}{\mathop{\rm QCoh}\nolimits}

\newcommand{\ev}{\mathop{\rm ev}\nolimits}

\newcommand{\cneq}{\mathrel{\raise.095ex\hbox{:}\mkern-4.2mu=}}
\newcommand{\eqcn}{\mathrel{=\mkern-4.5mu\raise.095ex\hbox{:}}}

\newcommand{\DT}{\mathop{\rm DT}\nolimits}

\newcommand{\vir}{\mathrm{vir}}
\newcommand{\num}{\mathrm{num}}

\makeatletter
 
  \@addtoreset{equation}{section}
\makeatother

\title[A degeneration formula of DT theory on CY 4-folds]
{A degeneration formula of Donaldson-Thomas theory \\ on local Calabi-Yau 4-folds}

\author{Yalong Cao}
\address{Morningside Center of Mathematics, Institute of Mathematics \& State Key Laboratory of Mathematical Sciences, Academy of Mathematics and Systems Sciences, Chinese Academy of Sciences, 55 Zhongguancun East Road, 100190, Beijing, China}
\email{yalongcao@amss.ac.cn}
\author{Gufang Zhao}
\address{School of Mathematics and Statistics, University of Melbourne, Parkville VIC 3010, Australia}
\email{gufangz@unimelb.edu.au}
\author{Zijun Zhou}
\address{School of Mathematical Sciences, Shanghai Jiao Tong University, 800 Dongchuan Road, 200240, Shanghai, hina}
\email{zijun.zhou@sjtu.edu.cn}

\date{\today}

\begin{document}

\maketitle

\begin{abstract}

We define relative Donaldson-Thomas invariants of local log Calabi-Yau 4-folds, under the assumption that the derived Hilbert scheme on the anti-canonical divisor is a derived critical locus. We prove a degeneration formula and apply it to compute 
zero dimensional invariants on $\mathbb{C}^4$ and on any local curve. This verifies a conjecture proposed by the first author and Kool. 

%Given a family of Calabi-Yau 4-folds, where generic fibers are smooth and singular fiber is the union of two log Calabi-Yau 4-folds, we construct a $(-2)$-shifted 
%symplectic structure on the derived Hilbert stack on the family. We check the isotropic condition of the underlying symmetric obstruction theory. 

\end{abstract}

\setcounter{tocdepth}{1}
\hypersetup{linkcolor=black}
\tableofcontents

\section{Introduction}

\subsection{Background/Motivation}

Enumerating curves on compact Calabi-Yau 4-folds has been drawing increasing interest in recent years. A key feature distinguishing Calabi-Yau manifolds of dimension 4 and above is that Gromov-Witten invariants vanish for genus greater than one.\footnote{On holomorphic symplectic 4-folds, there could have nontrivial genus 2 (reduced) Gromov-Witten invariants \cite{COT1, COT2}.} Inspired by this observation and drawing from heuristic arguments in ideal geometries, researchers have proposed Gromov-Witten/Donaldson-Thomas 
correspondence for compact Calabi-Yau 4-folds \cite{KP, CMT1, CMT2, CT2, CK2}.

However, the picture becomes considerably more complex for non-compact (e.g.,~local) Calabi-Yau 4-folds. Their invariants, defined via torus localization, can encompass all genera and are rational functions of equivariant parameters, making the aforementioned correspondences highly elusive in this setting. 
Attempting on a better understanding on (non-compact) Calabi-Yau 4-folds, we draw inspirations from 
 the realm of enumerative geometry for 3-folds. On 3-folds, the degeneration formula for Donaldson-Thomas and Pandharipande-Thomas theories stands as a formidable tool \cite{LW} (see also \cite{Zhou1} for the case of orbifolds). Its elegance has been instrumental in numerous groundbreaking computations, for example \cite{MNOP2, MOOP, OP, Oko, PP1, PP2, PP3}. 

Parallel to the relationship between DT theory of 3-folds and the quantum cohomology of Nakajima quiver varieties \cite{OP1, OP}, curve counting theory on non-compact Calabi-Yau 4-folds is connected to the quantum critical cohomology of quivers with potentials and also the theory of gauged linear sigma models \cite{CZ}. The latter provides representation theoretical, and physical motivations and contexts in studying the degeneration formula, where integrable systems related to cohomological Hall algebras \cite{KS2} and quantum groups \cite{LY, RSYZ, VV} emerge.
 
The present paper % is the first attempt to 
extends the framework of degeneration formula to Donaldson-Thomas theory of Calabi-Yau 4-folds. 
%\yc{We remark that such a generalization is nontrivial and is not a straightforward extension of the 3-fold case, where difficulties include an essential use of shifted 
%symplectic geometry in the context of derived algebraic geometry and the fact that invariants depend on the choice of orientation \cite{CGJ, Boj}. 
%Although most of the technologies developed in this paper extend directly to any Calabi-Yau 4-fold, our presentation restricts to the case of local Calabi-Yau 4-folds,~i.e.~the total spaces of %canonical bundles of 3-folds, in order to construct canonical orientations on moduli spaces for computational purposes (see Remark \ref{rmk on local cy4 ori}).
%\footnote{Most of the technologies developed in this paper extend directly to any Calabi-Yau 4-fold. The condition that the Calabi-Yau 4-fold being local is only used to construct canonical %orientations on moduli spaces for computational purposes (see Remark \ref{rmk on local cy4 ori}).}. 
%{\color{orange}
%\begin{remark}We remark that this is \textbf{\textit{not}} a straightforward extension of the formula on 3-folds, due to difficulties including an essential use of shifted 
%symplectic geometry in the context of derived algebraic geometry, and the fact that invariants depend on the choice of orientation \cite{CGJ, Boj}. 
%\end{remark}}
%\zz{Shall we mention here briefly which technologies developped here can extend to an arbitrary CY 4-fold?}
We remark that such an extension is a difficult albeit potentially fruitful venture. The difficulty comes from two sources: (1) Due to the nature of moduli spaces involved in the 4-fold case, the use of derived algebraic geometry, and shifted symplectic geometry is crucial. 
A detailed explanation is given in Remark~\ref{rmk:Joyce}. (2) The invariants of 4-folds depend on the choice of orientation \cite{CGJ, CGJ2}, which behave more complicated than those of 3-folds.

In below, we give a summary of results obtained in this paper. 
\iffalse
Although many technologies developed in this paper extend directly to arbitrary Calabi-Yau 4-folds, for exposition purpose, we formulate degeneration formula on local Calabi-Yau 4-folds,~i.e.~the total spaces of canonical bundles of 3-folds, where canonical orientations on moduli spaces can be constructed  and hence explicit computations can  be done (see Remark \ref{rmk on local cy4 ori}). As an application, we apply degeneration formula to compute zero-dimensional invariants on $\bbC^4$ and all local curves.
\fi
 %The degeneration formula presented here not only expands the applicability of this powerful tool to 4-folds but also holds promise for establishing Gromov-Witten/Donaldson-Thomas correspondences (as in \cite{MNOP1, MNOP2}) on any Calabi-Yau 4-fold. 

%remain mysterious as, making Gromov-Witten/Donaldson-Thomas style correspondence as \cite{MNOP1, MNOP2} much more difficult to formulate (if there is any). 

%\subsection{Main results}

\subsection{Expanded degenerations and family invariants}
Let $\pi : X \to \bbA^1$ be a Calabi-Yau simple degeneration of relative dimension $n$ (Definitions \ref{defi of dege}, \ref{cy simple dege}).
There is an associated stack $\calC$ of expanded degeneration (Definition \ref{defi of exp dege}) and a universal family  
$$ \calX \to \calC, $$
%over a smooth Artin stack $\calC$,
which is a Calabi-Yau family (Proposition \ref{prop on canon bdl}). 
 Let $\Hilb^{}(X/\bbA^1)$ be the Hilbert stack of stable subschemes on the family $\calX \to \calC$ (Definition \ref{defin on rel hilb stack}).
Fixing compactly supported numerical $K$-theory classes $P_t \in K_{c, \leq n - 2}^\num (X_t)$ for all $t\in \bbA^1$, we have an open substack (Lemma \ref{lem on open}):
$$\Hilb^{P_t}(X/\bbA^1)\subset \Hilb(X/\bbA^1)$$ 
of stable subschemes 
on the decorated expanded degeneration $\calX^{P_t} \to \calC^{P_t}$ with class $P_t$ for $t\in \bbA^1$ (see  \S \ref{sect on rel hilb stack} for the details). 
A version of AKSZ-type argument, similar to  \cite{PTVV, Pre}, yields:
\begin{theorem}
[Theorem \ref{thm on sft symp str}]
\label{intro thm on sft symp str}
$\Hilb^{P_t }(X/\bbA^1)$ has a derived enhancement $\textbf{\emph{Hilb}}^{P_t }(X/\bbA^1)$ with a canonical $(2-n)$-shifted symplectic structure  over $\calC^{P_t}$. 
%Moreover it is orientable in the sense of Definition \ref{ori on even cy}. 
\end{theorem}
In the present paper, we focus on the case when $n=4$, where $\textbf{{Hilb}}^{P_t }(X/\bbA^1)\to \calC^{P_t}$  has $(-2)$-shifted symplectic structure. Subject to the isotropic condition (verified in Proposition~\ref{iso cond 1}), 
%\begin{remark}
%Here the existence of shifted symplectic structures is proven for an \emph{arbitrary} dimensional Calabi--Yau simple degeneration 
%and follows from.The assumption as in \eqref{intro X to A1} is only required for the existence of orientation.\end{remark}
there is a well-defined square root virtual pullback \cite{Park1}. 
%which is a family version of the construction of \cite{OT}. 
%The isotropic condition is proved in  by extending the deformation invariance argument of \cite{Park2}. 
%Having such a (relative) $(-2)$-shifted symplectic structure, one can define square root virtual pullback \cite{Park1} after checking the isotropic condition, which is done in %Proposition~\ref{iso cond 1} by extending the deformation invariance argument of \cite{Park2}.  
 \begin{theorem}
[Family invariants, Proposition \ref{iso cond 1},~Theorem \ref{thm pb on family}]\label{intro thm pb on family}
Fix $K$-theory classes $P_t \in K_{c, \leq 2}^\mathrm{{num}} (X_t)$ for all $t\in \mathbb{A}^1$. \begin{enumerate}\item The map 
\begin{equation}\label{intro equa on pi}\pi: \Hilb^{P_t}(X/\bbA^1)\to \calC^{P_t} \end{equation}
has a canonical symmetric obstruction theory in sense of Definition \ref{def on sym ob}, which is isotropic 
if $P_t\in K_{c, \leq 1}^{\mathrm{num}} (X_t)$ for all $t\in \mathbb{A}^1$,~i.e.~their support has dimension not bigger than one.
\item In the presence of an orientation on $\pi$  in the sense of Definition \ref{ori on even cy}, then  there is a square root virtual pullback
\begin{equation}\label{intro sqr pb of hilb family} \sqrt{\pi^!} : A_*(\calC^{P_t})\to A_*(\Hilb^{P_t}(X/\bbA^1)), \end{equation}
where $A_*(-)$ denotes the Chow group with rational coefficients.  
%Here the choice of orientation is given as Remark \ref{rmk on ori}. 
%Moreover when there is a torus $T$ acting 
\item In the further presence of a torus $T$ action on the family $X\to \bbA^1$, preserving Calabi-Yau volume forms on fibers, 
the above pullback map lifts to a map between $T$-equivariant Chow groups. 
\end{enumerate}
\end{theorem}
The above pullback map captures the virtual counting for the family $\pi$ whose base change to a general element in $\calC^{P_t}$ recovers the virtual 
counting of the general fiber of $X/\mathbb{A}^1$. The base change to the special fiber can be linked with relative invariants which we now introduce. 

\subsection{Expanded pairs and relative invariants}

Let $(Y,D)$ be a smooth log Calabi-Yau pair (Definition \ref{defi of log cy pair}) with $\dim_{\C}Y=n$, 
and $$(\calY,\,\calD\cong D\times \calA)$$ be the associated universal expanded pairs over the stack $\calA$ of expanded pairs (Definition~\ref{def of exp pair}),
which is a log Calabi-Yau pair (relative to $\calA$) (Proposition \ref{prop on anti can div}). 

There exists a Hilbert stack $\Hilb (Y, D)$ parameterizing families of stable subschemes in the universal family 
$\calY / \calA$ (Definition \ref{defin on rel hilb stack}) and a restriction map 
$$r_{\calD\to \calY}: \Hilb (Y, D)\to \Hilb (\calD)=\Hilb (D)\times \calA. $$
%$$I_Z\mapsto I_{Z\cap D}.$$
Fix $P\in K_{c, \leq n-2}^{\mathrm{num}} (Y)$ with $P|_D\in K_c^{\mathrm{num}} (D)$ to be the $K$-theoretic Gysin pullback \eqref{equ on k pullback} to $D$. We have open 
substacks 
$$\Hilb^P(Y, D)\subset \Hilb (Y, D), \quad \Hilb^{P|_D}(D)\subset \Hilb (D), $$
and a similar restriction map  
\begin{equation}\label{intro rest map cl}r_{\calD\to \calY}: \Hilb^P(Y, D)\to \Hilb^{P|_D}(D)\times \calA^P.  \end{equation}
As in Theorem \ref{intro thm on sft symp str}, the target of \eqref{intro rest map cl} has a derived enhancement with a canonical 
$(3-n)$-shifted symplectic structure relative to $\calA^P$.
By a family argument of \cite{Cal}, we obtain:
\begin{theorem}
[Theorem \ref{thm on lag}]
The restriction map $r_{\calD\to \calY}$ has a derived enhancement 
\begin{equation}\label{intro rest map}r_{\calD\to \calY}:\textbf{\emph{Hilb}}^P(Y,D)\to \textbf{\emph{Hilb}}^{P|_D} (D)\times \calA^P,  \end{equation}
which has a $(3-n)$-shifted Lagrangian structure relative to $\calA^P$. 
\end{theorem}
We restrict to the case when $n=4$ and hence $D$ is a Calabi-Yau 3-fold. 
\begin{remark}\label{rmk:Joyce}
The question of extracting ``correct" relative invariants from the $(-1)$-shifted Lagrangian structure on \eqref{intro rest map} is illusive. 
Indeed, there is a conjecture by 
Joyce \cite{Joy}, asserting the existence of a map:
\begin{equation}\label{equ on jy map}\Q_{\Hilb^P (Y, D)}[\mathrm{vd}_{\Hilb^P (Y, D)}] \to r_{\calD\to \calY}^!\varphi_{\Hilb^{P|_D} (D)\times \calA^P}, \end{equation}
with prescribed properties. Here $\varphi_{\Hilb^{P|_D} (D)\times \calA^P}$ is certain perverse sheaf coming from gluing local sheaves of vanishing cycles. 
%complex of the pullback function $\phi|_{W\times \calA^P}: W\times \calA^P\to \C$of $\phi: W\to \C$. 
However, there are layers of difficulties in proving this conjecture, partially involving gluing in infinity categories
(see also \cite{KK} which addresses other difficulties).
\end{remark}
%\gufang{Do we say something about ``in what follows, we establish such relative invariants in some cases and hope this method can be generalized"}
%\yl{no we did not do it.}
%\begin{remark}In this theorem, the existence of Lagrangian structures holds for any log 
%Calabi-Yau pair of arbitrary dimension and the proof follows from a family argument of \cite{Cal}.
%The construction of the orientation in this theorem is specific to the case \eqref{intro YD}. \end{remark}
Note that by Darboux theorem \cite{BG, BBJ}, derived critical loci of regular functions on smooth schemes 
provide local charts of the $(-1)$-shifted symplectic derived scheme $\textbf{{Hilb}}^{P|_D} (D)$.
In this paper, we focus on such local charts and make the following assumption.
%Notice that Hilbert schemes $\Hilb^{P|_D}(D)$ on local CY 3-fold $D$ in \eqref{intro YD} are often written as global critical loci. 
%We need the following derived enhancement for our construction.
\begin{ass}\label{intro ass on lag fib0}
%[Assumption \ref{ass on lag fib}]
Fix $P_D\in K_c^\mathrm{{num}}(D)$.
We assume there is an equivalence 
\begin{equation}\label{intro equ on hilb equ dcri}\textbf{Hilb}^{P_D} (D)\cong \textbf{Crit}(\phi:W\to \mathbb{C}) \end{equation} 
of $(-1)$-shifted symplectic derived scheme, where 
$W$ is a quasi-projective smooth scheme and $\phi$ is a regular function.
\end{ass}
%we define a version of relative invariants based on the 
%above theorem (when $n=4$, which corresponds to the case when $D$ is a Calabi-Yau 3-fold), under the following assumption. 
Let $P\in K_{c,\leq 2}^{\mathrm{num}} (Y)$ and $P_D:=P|_D\in K_c^{\mathrm{num}} (D)$. Denote the inclusion to the ambient space by 
$$p: \textbf{Hilb}^{P_D}(D)\to W.$$
By an abuse of notation, 
we also denote its product with $\id_{\calA^P}$ as $p$.  
The way taken in this paper to define \textit{relative invariants} is then motivated by the study on quasimaps to quivers with potentials \cite{CZ}, 
where we observe that the \textit{composition map}:
\begin{equation}\label{equ on pc}p\circ r_{\calD\to \calY}: \textbf{Hilb}^P(Y,D) \to W\times \calA^P \end{equation}
has a (relative) $\textit{(-2)}$-\textit{shifted symplectic structure} (Proposition \ref{prop on iso condition2}).

As in Theorem \ref{intro thm pb on family}, we check the isotropic condition (Proposition~\ref{iso cond 2}), and obtain a virtual pullback map 
(by abuse of notation, $p$ and $r_{\calD\to \calY}$ also denote their classical truncations):
%Before listing examples where Assumption \ref{intro ass on lag fib} is satisfied, we first give the consequence of it.  
%We denote the inclusion of the Hilbert scheme to the ambient space by $p: \Hilb (D)\to W$. 
\begin{theorem} 
[Relative invariants, Proposition \ref{iso cond 2}, Theorem \ref{thm on relative invs}]\label{intro thm on relative invs}
Let $P\in K_{c, \leq 2}^{\mathrm{num}} (Y)$.
\begin{enumerate}\item Under Assumption \ref{intro ass on lag fib0}, the map 
\begin{equation}\label{equ on pc cl}p\circ r_{\calD\to \calY}: \Hilb^P(Y,D) \to W\times \calA^P \end{equation}
has a canonical symmetric obstruction theory in sense of Definition \ref{def on sym ob}. If $P\in K_{c,\leq 1}^{\mathrm{num}} (Y)$, the obstruction theory  is isotropic after base change to the zero locus $Z(\phi)\hookrightarrow W$ of $\phi: W\to \C$.
%where $Z(\phi)\hookrightarrow W$ denotes . 

\item In the presense of an orientation on \eqref{equ on pc cl}   in the sense of Definition \ref{ori on even cy},  there is a square root virtual pullback
\begin{equation}\label{intro sqr pb of hilb}\sqrt{(p\circ r_{\calD\to \calY})^!} : A_*(Z(\phi)\times \calA^P)\to A_*(\Hilb^P(Y,D)). \end{equation}
\item In the further presence of a torus $T$ action on the pair $(Y,D)$ and the equivalence \eqref{intro equ on hilb equ dcri}, the above pullback lifts to a map between $T$-equivariant Chow groups. 
\end{enumerate}
\end{theorem}
\begin{remark}
%\guf{ It is natural to attempt removing  Assumption \ref{intro ass on lag fib}. In general, there is a $(-1)$-shifted Lagrangian structure on \eqref{intro rest map}, hence  a certain perverse sheaf $\varphi_{\Hilb(D)\times \calA}$  on $\Hilb (D)\times \calA$ constructed by gluing sheaves of vanishing cycles of (local) regular functions \cite{BBDJS, KL}. It is conjectured by  Joyce \cite{Joy}, that there is map:
%it is often very difficult to do explicit computations outside the algebraic geometric setting 
%it is often preferable to remain in algebro-geometric category to be able to do explicit computations, while $\varphi_{-}$ here lies in the constructible category. 
%\yl{should not mention natural to remove Assumption \ref{intro ass on lag fib}. Here is my version}
It is expected that our relative invariant \eqref{intro sqr pb of hilb} encodes the effect of \eqref{equ on jy map} on \textit{algebraic cycles},~i.e.~we should have 
the following commutative diagram (see also \cite[diagram~(1.6)]{CZ}): 
\begin{equation}\label{diag on rel to jy map}
\xymatrix{
A_*(Z(\phi))\otimes A_*(\calA^P)   \ar[r]^{ } \ar[d]_{can\,\circ\, cl}   & A_*(\Hilb^P(Y,D)) \ar[d]^{cl} \\
H_*(W,\varphi_\phi)\otimes H^{BM}_*(\calA^P)  \ar[r]^{\,\,\,\, \eqref{equ on jy map}} &  H^{BM}_*(\Hilb^P(Y,D)),
}
\end{equation}
where the left vertical map is given by the composition of the cycle map with the canonical map from Borel-Moore homology of $Z(\phi)$ to the 
vanishing cycle cohomology (ref.~\cite[Eqn.~(A.8)]{CZ}).  

There is a version of Theorem~\ref{intro thm on relative invs} where Chow groups are replaced by Grothendieck groups $K_0(-)$ of coherent sheaves, with similar proof. Moreover, 
we expect an analogue of diagram~\eqref{diag on rel to jy map}, where the cohomology of vanishing cycle is replaced by the  Grothendieck group of the matrix factorization category $MF(W,\phi)$.
In this setting, the canonical map $$K_0(Z(\phi))\to K_0(MF(W,\phi))$$ is  \textit{surjective} 
(e.g.~\cite[Lem.~1.10]{PS},~\cite{Sch}).
Therefore, the $K$-theoretic version of \eqref{intro sqr pb of hilb} contains all information about the $K$-theoretic version of the provisional map \eqref{equ on jy map}. 
The $K$-theoretic analog of diagram~\eqref{diag on rel to jy map} is recently addressed in \cite{CTZ}.
\end{remark}

\subsection{Local (log) Calabi-Yau 4-folds}\label{sect on intro log cy4}
In general, we do not know whether the maps \eqref{intro equa on pi}, \eqref{equ on pc} are orientable.
It is an even more difficult question to find canonical orientation suitable for calculations (see Remark \ref{rmk on local cy4 ori} for details). 
\textit{Local} (\textit{log}) \textit{Calabi-Yau 4-folds} provide a large class of examples where we can obtain canonical orientation for the maps in \textit{loc}.~\textit{cit}.

Let $U\to \bbA^1$ be a simple degeneration of 3-folds (Definition \ref{defi of dege}), whose fiber over $0\in \bbA^1$ is 
a union $U_+\cup_S U_-$ of two smooth 3-folds $U_\pm$ along a common smooth divisor $S$. Let 
\begin{equation}\label{intro X to A1}X=\Tot(\omega_{U/\bbA^1})\to \bbA^1 \end{equation} 
be the associated Calabi-Yau simple degeneration, where generic fibers are total spaces of canonical bundles of smooth 3-folds
and central fiber is a union of log Calabi-Yau 4-folds (Definition \ref{defi of log cy pair}),~i.e.~the fiber $X_c$ over $c\in \bbA^1$ satisfies 
\begin{itemize}
\item $X_c=\Tot(\omega_{U_c})$ if $c\neq 0$, 
\item $X_0=Y_+\cup_D Y_-$, where $Y_{\pm}=\Tot(\omega_{U_\pm}(S))$ and $D=\Tot(\omega_S)$. 
\end{itemize}
\begin{theorem}
[Theorem \ref{ori of cy4 family},~Remark \ref{rmk on ori}]
\label{intro thm can ori1}
For \eqref{intro X to A1}, 
the map $\pi: \Hilb^{P_t }(X/\bbA^1)\to \calC^{P_t}$ \eqref{intro equa on pi} has a canonical orientation.  
\end{theorem}

Let $(U,S)$ be a smooth pair with $\dim_{\mathbb{C}} U=3$. Consider the log Calabi-Yau pair (Definition \ref{defi of log cy pair}):
\begin{equation}\label{intro YD}Y=\Tot(\omega_{U}(S)), \quad D=\Tot(\omega_S). \end{equation}

In this case, there is a line bundle $K^{-1/2}_{\Hilb(D)}$ \eqref{square root1.2} constructed using the fact that a derived moduli stack 
of compactly supported sheaves on $D=\Tot(\omega_S)$ is the $(-1)$-shifted cotangent bundle of a derived moduli stack of sheaves on $S$. 
\begin{ass}\label{intro ass on lag fib}
Let $\textbf{Hilb}^{P_D} (D)\cong \textbf{Crit}(\phi:W\to \mathbb{C})$ be as in Assumption \ref{intro ass on lag fib0}. Assume
\begin{equation}\label{intro equ coptble}K^{-1/2}_{\Hilb(D)}\cong \det\left(\bbT_W|_{\Hilb(D)}\right). \end{equation}
\end{ass}
Eqn.~\eqref{intro equ coptble} is a compatibility condition between the square root provided using $S$ and the one provided using the ambient space $W$ of the Hilbert scheme.
There are many cases where \eqref{intro equ coptble} is satisfied (Example~\ref{example on Scp}). 
Examples where Assumption \ref{intro ass on lag fib0} is satisfied are discussed in \S \ref{sect on appli}. 

\begin{theorem}
[Theorem \ref{relative ori of cy4 family}, Remark \ref{rmk on ori2}]
\label{intro thm can ori2}
Under Assumption \ref{intro ass on lag fib}, the map \eqref{equ on pc cl}:
$$p\circ r_{\calD\to \calY}: \Hilb^P(Y,D) \to W\times \calA^P $$ 
has a canonical orientation.
\end{theorem}
 
\subsection{Degeneration formulae}
Given $X/\bbA^1$ as in \eqref{intro X to A1} and $(Y,D)$ as in \eqref{intro YD}, by combining Theorems \ref{intro thm pb on family}, \ref{intro thm can ori1}, 
and Theorems \ref{intro thm on relative invs}, \ref{intro thm can ori2}, we define the following invariants using canonical orientations of \S \ref{sect on intro log cy4}.  
%To state the degeneration formula, we need the following:
\begin{itemize}
\item
Fix $P_t \in K_{c,\leqslant 1}^\mathrm{{num}} (X/\mathbb{A}^1)$ for all $t\in \bbA^1$. The \textit{family invariant} (Definition~\ref{defi of fam inv}) is 
given by the pullback map \eqref{intro sqr pb of hilb family}:
\begin{equation*}\Phi^{P_t}_{X/\bbA^1}:=\sqrt{\pi^!}: A_*(\calC^{P_t}) \to A_*(\Hilb^{P_t}(X/\bbA^1)).  \end{equation*}
%where $\sqrt{\pi^!}$ is given by \eqref{intro sqr pb of hilb family}, 
\item
Fix $P \in K_{c,\leq 1}^\mathrm{{num}} (Y)$. The \textit{relative invariant} (Definition \ref{defi of rel inv}) is a map 
\begin{equation*} \Phi^{P}_{Y,D}:=\sqrt{(p\circ r_{\calD\to \calY})^!}\circ \boxtimes : A_*(Z(\phi))\otimes A_{*}(\calA^{P}) \to A_*(\Hilb^P(Y,D)),  \end{equation*}
where $\sqrt{(p\circ r_{\calD\to \calY})^!}$ is given by \eqref{intro sqr pb of hilb} and $\boxtimes$ is the exterior product for Chow cycles, 
\item
Fix $P_{\pm} \in K_{c,\leq 1}^\mathrm{{num}} (Y_\pm)$. The \textit{product version} \eqref{tensor of rel inv} of the above relative invariant  is
\begin{equation*}\Phi^{P_-}_{Y_-,D}\otimes\Phi^{P_+}_{Y_+,D}:
A_*(Z(\phi\boxplus \phi))\otimes A_{*}(\calA^{P_-}\times \calA^{P_+}) \to A_*\left(\Hilb^{P_-}(Y_-,D)\times \Hilb^{P_+}(Y_+,D) \right). \end{equation*}
\end{itemize}
%Here the decoration $(-)^{P_{(\pm)}}$ denotes the open and closed substack of objects with total continuous weight $P_{(\pm)}$ (see \S \ref{sect on rel hilb stack} for details). 
%Definitions of above invariants use orientations, which we choose in a canonical way as specified in Remarks \ref{rmk on ori} \& \ref{rmk on ori2}. 
\begin{theorem}
[Degeneration formula, Theorem \ref{thm on dege formula}]\label{intro thm on dege formula}
For invariants defined above using canonical orientations of \S \ref{sect on intro log cy4}, we have    
$$i_r^!\left(\Phi^{P_t}_{X/\bbA^1}[\calC^{P_t}]\right)=[\Hilb^{P_r} (X_r)]^{\vir}, \quad \forall\,\,r\in \mathbb{A}^1\backslash \{0\},$$
$$i_{0}^!\left(\Phi^{P_t}_{X/\bbA^1}[\calC^{P_t}]\right)= n_*\,\sum_{(P_-, P_+)\in \Lambda^{P_0}_{spl}}i_{(P_-, P_+)}^!\left(\Phi^{P_t}_{X/\bbA^1}[\calC^{P_t}]\right). $$
Moreover, for each splitting data $(P_-, P_+)\in \Lambda^{P_0}_{spl}$ \eqref{equ on split data}, we have 
%$$(\bar{\sigma})^*i_{(P_-, P_+)}^!\left(\Phi^{P_t}_{X/\bbA^1}[\calC^{P_t}]\right)=
%\sum_{(P_-, P_+)\in \Lambda^{P_0}_{spl}}\sqrt{r_{\bar{\Delta}}^!}\,\left([W]\times [\calC_0^{\dagger, (P_-, P_+)}] \right), $$
%which satisfies 
$$i_{\bar{\Delta}*}(\bar{\sigma})^*i_{(P_-, P_+)}^!\left(\Phi^{P_t}_{X/\bbA^1}[\calC^{P_t}]\right)=
 \Phi^{P_-}_{Y_-,D}\otimes\Phi^{P_+}_{Y_+,D} \,\left(\bar{\Delta}_*[W]\times [\calC_0^{\dagger, (P_-, P_+)}] \right). $$ 
In the presence of a torus action,  the above equalities hold in equivariant Chow groups. 
%The above equalities hold in equivariant Chow groups when there are torus actions. 
\end{theorem}
Here  $i_{(-)}^!$ is the Gysin pullback of $i_{(-)}$ in \eqref{big diag}, $i_{\bar{\Delta}}$ is given in \eqref{eqn:chasing},  $\bar{\sigma}$ is the involution \eqref{equ on auto sig}, 
$\bar{\Delta}$ is the anti-diagonal embedding \eqref{equ on bar del}, 
%and the summation is taken over all splitting data \eqref{equ on split data}.  
and $n$ is the normalization map in \eqref{diag on cpr two maps0}.
%In particular, the degeneration formula is compatible with canonical orientations specified in Theorems \ref{intro thm can ori1}, \ref{intro thm can ori2}.}
\iffalse
\begin{remark}
%\yc{We remark that although our $\DT_4$ degeneration formula looks similar to the formula in the 3-fold case \cite{LW}. The technologies needed in the 4-fold case are much harder, 
%which are based on shifted symplectic geometry \cite{PTVV, Park2} and the recent development of new style $\DT_4$ virtual classes \cite{OT, Park1}. } 
%The 4-fold degeneration formula presented here takes similar form to that in dimension 3 \cite{LW}, the proof requires new ideas and  technologies, 
%and are based on shifted symplectic geometry \cite{PTVV, Park2} and the recent development of new style $\DT_4$ virtual classes \cite{OT, Park1}. }
When $Y=\mathrm{Tot}_C(L_1\oplus L_2\oplus L_3)$ is a local curve and $D=\C^3$, the Hilbert stack\footnote{More precisely, a variant of it using Pandharipande-Thomas stability has such a property.} $\Hilb (Y, D)$ has component(s) isomorphic to the moduli stack of stable quasimaps from $C$ to $\Hilb^*(\C^3)$. 
In this case, our relative invariants  define Gromov-Witten type invariants for the generally singular scheme $\Hilb^*(\C^3)$, and degeneration 
formula provides corresponding gluing formula \cite{CZ}. 
\end{remark}
\fi

\begin{remark}
Our degeneration formula is stated on the total space $X=K_U$ of the canonical bundle of a 3-fold $U$. It \textbf{\textit{by}} \textbf{\textit{no}} \textbf{\textit{means}} can be reduced to the degeneration formula of the 3-fold $U$ \cite{LW} for the following two reasons. (1) For general curve classes, the stable subschemes (or stable pairs) on $X$ \textbf{\textit{do}} \textbf{\textit{not}} necessarily scheme theoretically lie on the 3-fold $U$, so in general there is even no map between Hilbert stacks on $X$ and $U$.  (2) Even for an \textit{irreducible curve} class $\beta$, where a moduli $P(K_U,\beta)$ of PT stable pairs on the total space $K_U$
can be identified with a moduli $P(U,\beta)$ of PT stable pairs on $U$, our degeneration formula \textbf{\textit{does}} \textbf{\textit{not}}  follow from the 3-fold one. 
This is because under the above identification, the 4-fold virtual class reduces to a virtual class on $P(U,\beta)$, defined 
by pair-deformation obstruction theory $\Ext^*(I,F)$ (e.g.~\cite[Prop.~4.3]{CMT2}), different from the sheaf obstruction theory $\Ext^*(I,I)_0$ used in \cite{LW}. 

%The derived structure/obstruction theory  for a simple degeneration of 4-folds does not simply reduce to that of the corresponding 3-folds.
\end{remark}

\subsection{Applications}\label{sect on appli}

Next we mention examples where Assumption \ref{intro ass on lag fib} is satisfied and discuss applications. The first one is given by relative Hilbert schemes of one dimensional subschemes on local $\mathbb{P}^1$ (which are total spaces of rank three bundles on $\mathbb{P}^1$) with their restriction maps to Hilbert schemes of points on $\C^3$ (Examples \ref{fl ex}, \ref{example on Scp}). We expect this will be helpful\footnote{
As suggested by the recent work of Pardon \cite{Pard} on curve counting on 3-folds, one might expect that if Gromov-Witten/Donaldson-Thomas style correspondence on any Calabi-Yau 4-fold  exists, the proof can be reduced to the case of 4-folds which are local curves.}  to understand
the structure of Donaldson-Thomas/Pandharipande-Thomas type curve counting invariants on such non-compact Calabi-Yau 4-folds. 
%which will be addressed in a future work.
%\gufang{As suggested by the recent work of Pardon \cite{Pard} on curve counting on 3-folds, one might expect that the existence as well as a formulation of a Gromov-Witten/Donaldson-%Thomas style correspondence on any Calabi-Yau 4-fold to be determined, using degeneration formula, by calculations on 4-folds which are local curves. } 
% to be strong enough to determine Gromov-Witten/Donaldson-Thomas style correspondence on any Calabi-Yau 4-fold (if there is any)

The other example which satisfies Assumption \ref{intro ass on lag fib} is the main focus of this paper, given by relative Hilbert schemes of points on the pair \eqref{intro YD} with their restriction maps to a point.  By using a variant of the theory of 
expanded pairs, called rubber theory \cite{OP}, we are able to prove the conjectural tautological Donaldson-Thomas formula on $\C^4$ \cite[Conj.~1.6]{CK1} by following the \textit{magical} pole analysis
as the 3-fold case \cite{MNOP2}.
\begin{theorem}
[Theorem \ref{thm on ck conj}]
\label{intro thm on C4}
Let $T\subset (\C^*)^4$ be the Calabi-Yau subtorus \eqref{equ on cy subtor} and $L\in \Pic^{(\C^*)^4}(\C^4)$ be an equivariant line bundle on $\C^4$.
With the canonical choice of orientation (in Definition \ref{defi of dim zero dt4}), we have 
\begin{equation}\label{intro mac}1+\sum_{n=1}^\infty q^n \int_{[\Hilb^n (\bbC^4)]_T^\vir}e_T(L^{[n]})=M(q)^{\int_{\C^4} c^T_1(L)\cdot c^T_3(\C^4) }, \end{equation}
where 
$$M(q):=\prod_{n\geqslant 1}\frac{1}{(1-q^n)^n}$$
is the MacMahon function and $\int_{\C^4}$ denotes equivariant push-forward to point. 
\end{theorem}
As a special case of the above result, we prove a conjecture of Nekrasov: 
\begin{corollary}
[Corollary \ref{cor on nek}]
We have 
$$\displaystyle 1+\sum_{n=1}^\infty q^n \int_{[\Hilb^n (\bbC^4)]_T^\vir} 1  = \exp \left[-\dfrac{(s_1 + s_2 )(s_1 + s_3) (s_2 + s_3)  }{s_1 s_2 s_3 (s_1 + s_2 + s_3) } q
 \right], $$
where $s_1,s_2,s_3$ are equivariant parameters of $T$. 
\end{corollary}
We recall the basics of rubber theory in Appendix \ref{sec on rub}, and prove the above theorem in \S \ref{sect on hilb}. 
We remark that compared with the 3-fold case, $\DT_4$ invariants are much more difficult to compute due to the difficulty on the choice of orientation. 
The canonical orientations of \S \ref{sect on intro log cy4}, used in Theorem \ref{intro thm on dege formula}, are compatible with degeneration formula. 
We also prove the compatibility of  this canonical orientation with the orientation used in the work of Kool and Rennemo \cite{KR}, see Proposition \ref{prop on cpt ori}. Based on these observations, the proof of Theorem \ref{intro thm on C4} follows from Theorem \ref{intro thm on dege formula}, the pole analysis, as well as Kool-Rennemo's explicit sign formula for those zero dimensional subschemes in $\C^4$, which are scheme theoretically supported on a $\C^3$. See the proof of Theorem \ref{thm on ck conj} for more details.
%\yl{Delete the following: To the authors' knowledge, the above theorem is the first case where a closed formula of a whole $\DT_4$ generating series is given (without relying on any conjecture).}

By using the degeneration formula, we are also able to compute zero dimensional $\DT_4$-invariants for some \textit{non-toric} Calabi-Yau 4-folds.
Consider the following smooth quasi-projective 4-fold
\begin{equation}\label{intro equ of local curve}X := \Tot_C (L_1 \oplus L_2 \oplus L_3), \end{equation}
where $C$ is a smooth connected projective curve of genus $g$, and $L_1$, $L_2$, $L_3$ are line bundles on $C$. 
Let $p_1, \ldots, p_r \in C$ be $r$ distinct points and $D_i := \{p_i\} \times \bbC^3\subset X$ be the fibers over $p_i$.  
We impose the (log) CY condition:
$$
L_1 \otimes L_2 \otimes L_3 \cong \omega_C (p_1 + \cdots + p_r), 
$$
which implies that 
\begin{equation}\label{intro eqn antican div}
D:=\sqcup_{i=1}^r D_i \end{equation}
is an anti-canonical divisor of $X$. Let $(\bbC^*)^3$ be the torus which acts on the fibers of the projection $X\to C$ and consider the subtorus
$$T:=\{(t_2,t_3,t_4) \,|\, t_2t_3t_4=1\} \subset (\bbC^*)^3. $$
As in the $\C^4$ case, one can define the tautological relative invariants and their generating series
$Z (X, \sqcup_{i=1}^r D_i\,; L_m)$
(ref.~Definition \ref{defi of rel inv of local curve}, where orientations are canonically chosen). 
%The following result gives a closed formula for them. 
\begin{theorem}
[Theorem \ref{thm on local curve}]
\label{intro thm on local curve}
Let $X$ be a local curve \eqref{intro equ of local curve} and $\sqcup_{i=1}^r D_i$ \eqref{intro eqn antican div} be its anti-canonical divisor.
Then
$$
Z (X, \sqcup_{i=1}^r D_i\,; L_m) = M(q)^{\int_X c_1^{T \times \bbC^*_m} (L_m)\, \cdot \,c_3^T (T_X[-D])}, 
$$
where $T_X[-D]$ is the sheaf of tangent vectors with logarithmic zeros along $D$ (Definition \ref{defi of log tan bdl}).
\end{theorem}

\subsection{Related works and some future direction}
Nekrasov conjectured a $K$-theoretic analogue of the formula \eqref{intro mac} \cite{Nek}, which is proved in a recent work of Kool and Rennemo \cite{KR}. 
While our proof of \eqref{intro mac} uses degeneration formula and pole analysis, their proof uses totally different idea and is based on factorizable sequence of sheaves developed by Okounkov \cite{Oko}. Note that the formula of {\it loc.\,cit.} is $K$-theoretical and reproduces the cohomological formula of the present paper under the cohomological limit studied in \cite{CKM1}.
And the advantage of using degeneration method in this paper is that one can also deal with non-toric Calabi-Yau 4-folds,~see e.g.~\eqref{intro equ of local curve}. 

In future work, we hope to apply the degeneration formula developed in this paper to formulate and prove Gromov-Witten/Donaldson-Thomas correspondence (as in \cite{MNOP1, MNOP2}) on general Calabi-Yau 4-folds.

We also remark that the study of degeneration formula of $\DT_4$ theory (in its most general form) should be put in the context of logarithmic-geometry (see \cite{MR} for a study in this direction on 3-folds). It is interesting to explore this perspective more in the future.

\subsection*{Acknowledgments}
%This work benefits from discussions and communications with many people, 
We are grateful to Tasuki Kinjo, Martijn Kool, Xiaolong Liu, Sergej Monavari, Hiraku Nakajima, Andrei Okounkov, Rahul Pandharipande, 
Hyeonjun Park, Pavel Safronov, Yukinobu Toda for many helpful discussions and communications during the preparation of this work.
We particularly thank Martijn Kool for explaining the sign rule and the orientation on Hilbert schemes of points on $\C^4$ used in his upcoming work with Rennemo, and Hyeonjun Park for several very helpful discussions on virtual pullbacks and shifted symplectic structures. 
%We are grateful to Taro Kimura, Tasuki Kinjo, Yongbin Ruan, Pavel Safronov,  and Yan Soibelman for insightful feedback on a preliminary draft of this paper. 
We thank the referee for several helpful suggestions which improve the exposition of the paper. 

The work of Y.~C.~was partially supported by RIKEN Interdisciplinary Theoretical and Mathematical Sciences
Program (iTHEMS), JSPS KAKENHI Grant Number JP19K23397 and Royal Society Newton International Fellowships Alumni 2022 and 2023. 
G.~Z.~is partially supported by the Australian Research Council via DE190101222 and DP210103081. 
Z.~Z.~is supported by NSFC grant 12401077 and the startup grant from Shanghai Jiao Tong University.

%\subsection*{Statements and Declarations}We have no conflicts of interest to disclose.

\section{Recollection of expanded pairs and degenerations  }

%\section{Recollection of stable quotients on stacks of expanded degenerations and pairs}

In this section, we review foundations on moduli stacks of expanded pairs and degenerations as developed by Jun Li \cite{Li1}. 
%In literature, these moduli spaces are used to establish a degeneration formula for Donaldson-Thomas invariants of 3-folds \cite{LW, Zhou1}, and beautifully used in many computations
%(e.g.~\cite{MNOP2, MOOP, OP, PP1, PP2, PP3}). 
In the present paper, we are mainly interested in those moduli stacks associated to (log) Calabi-Yau 4-folds. Nevertheless, this section works in arbitrary dimension. 

%For the first reading, the reader can skip this section and directly go to later sections.

\subsection{Expanded pairs} \label{Sec-exp-pair}
\begin{definition}\label{defi of smooth pair}
A \emph{smooth pair} $(Y, D)$ consists of a smooth quasi-projective variety $Y$ and a smooth connected\footnote{
The assumption that $D$ is connected can be removed, while the theory can be generalized without much effort. 
It suffices to work with each connected component of the relative divisor separately. In literature, 
this observation has been used in multiple instances, explicitly or implicitly, in Donaldson-Thomas invariants. In the present paper, we focus on the connected case in most of the sections. In Section \ref{sect on hilb}, we drop the connectedness assumption using this observation.} divisor $D\subset Y$. 
More generally, a \emph{locally smooth pair} $(Y', D')$ consists of a quasi-projective variety $Y'$ and a connected divisor $D'$, so that \'etale locally around $D'$ the pair is a smooth pair. 
\end{definition}

Let  $(Y, D)$ be a smooth pair. The associated \emph{bubble component} is the following $\bbP^1$-bundle over $D$
$$
\Delta := \bbP_D (\calO_D \oplus N_{D/Y}) . 
$$
It comes with two canonical sections $D_0 = \bbP_D ( \calO_D )$ and $D_\infty = \bbP_D (N_{D/Y})$. 

For an integer $k\geq 0$, let
$$
Y[k]_0 := Y \cup_D \underbrace{\Delta \cup_D \cdots \cup_D \Delta}_{k \text{-times}},
$$
where the divisor $D_\infty$ of the $i$-th $\Delta$ is glued to the $D_0$ of the $(i+1)$-th $\Delta$. 
The normal bundle of $D_0$ and that of $D_\infty$ in $\Delta$  are inverse to each other, making the singularity of $Y[k]_0$ smoothable in the first order. 

%\guf{There is a $(\bbC^*)^k$-action on $Y[k]_0$, where the $i$th $\bbC^*$-factor scales the $i$th bubble component.}\gufang{Can we be more precise?}
%\yl{maybe keep it short}

Let $D[k]_0$ be the $D_\infty$ of the last bubble component $\Delta$.
Then $$(Y[k]_0, D[k]_0)$$ is a locally smooth pair, which we call the \emph{expanded pair of length} $k$.

There is a projective map 
\begin{equation}\label{equ on contra map}p: Y[k]_0 \to Y,  \end{equation}
which contracts all bubble components, and maps the divisor $D[k]_0$ isomorphically to $D$.
% projectiveness of the map comes from relative to D, it is a P1 bundle.

In \cite{Li1},  for each $k$, a family of expanded pairs is explicitly constructed whose properties are given below. In particular, the generic fiber is $Y$ and the most special fiber is $Y[k]_0$. This family is called the {\it standard family}. 
%Standard families for different $k$'s fit together. 

\begin{prop}$($\cite{Li1, LW}, \cite[Prop.~2.16]{Zhou1}$)$\label{prop1} 
There is a flat family $$\pi_k : Y[k] \to \bbA^k, $$ 
together with a smooth connected divisor $D[k] \subset Y[k]$, and a projective map $p_k: Y[k] \to Y \times \bbA^k$ which satisfy the following properties.
\begin{enumerate}

\item The pair $(Y[k], D[k])$ is a smooth pair, with $Y[k]$ a smooth quasi-projective variety of dimension $\dim Y + k$. There is an isomorphism  $D[k] \cong D \times \bbA^k$ under which  $\pi_k |_{D[k]}$ coincides with projection to $\bbA^k$.

\item Let $t = (t_1, \ldots, t_k) \in \bbA^k$ be a closed point. 
Let $l = \sharp \{ i \mid t_i =0 \}$. The fiber of $\pi_k$ over $t$ is isomorphic to $Y[l]_0$. 
The restriction of $p_k$ to this fiber is the contraction map $p$ \eqref{equ on contra map}.

\item There is a $(\bbC^*)^k$-action on $Y[k]$, such that $\pi_k$ and $p_k$ are equivariant, where $(\bbC^*)^k$ acts on $\bbA^k$ in the standard way. 
For a closed point $t\in \bbA^k$ as above, the stabilizer group $\Stab_t (\bbA^k) \cong (\bbC^*)^l$ acts on the fiber $Y[l]_0$ by scaling $l$ bubble components $\Delta$'s. 

\end{enumerate}
\end{prop}
There is a smooth equivalence relation $\sim$ on both $Y[k]$ and $\bbA^k$ \cite[\S 2.5]{Zhou1}, induced  by the $(\bbC^*)^k$-action above and 
permutations of the coordinates of  $\bbA^k$ off diagonal. Moreover, for $l\leq k$, the family $Y[l] \to \bbA^l$ has standard embeddings into $Y[k] \to \bbA^k$.  The direct limit of the families, up to the equivalence relations, define stacks of interest.

\begin{definition}$($\cite{Li1},~\cite[\S 2.3]{LW},~\cite[Def.~2.20]{Zhou1}$)$\label{def of exp pair}
The \emph{stack of expanded pairs} is defined as $\calA := \underrightarrow\lim \ [\bbA^k / \sim~]$, with \emph{universal expanded pair} $\calY := \underrightarrow\lim \ [Y[k] / \sim ~]$, and  \emph{universal relative divisor} $\calD := \underrightarrow\lim \ [D[k] / \sim ~]$. 
\end{definition}

%Proposition \ref{prop1} above implies that 

$\calA$ is a smooth Artin stack of dimension $0$, and $\calY$ is a smooth Artin stack of dimension $\dim Y$. 
There is a representable, projective map $$p: \calY \to Y \times \calA, $$ 
whose composition with the projection to $\calA$ gives the family map 
$$\pi : \calY \to \calA, $$ 
which is flat and also representable. 
The universal relative divisor $\calD$ satisfies  
$$\calD\cong D \times \calA, $$ 
such that the restriction of $p$ to which coincides with the base change of $D \hookrightarrow Y$ over $\calA$. 

For a scheme $S$, a \emph{family of expanded pairs over $S$} is 
 a map $S \to \calA$.  Pulling back the universal family $\pi$ gives a family $$\pi_S : \calY_S \to S, $$ 
endowed with a divisor $\calD_S\subset \calY_S$ obtained by 
 pulling back the universal divisor $\calD$. We also refer $(\calY_S,\calD_S)$ as a family of expanded pairs over $S$.
%In particular, any expanded pair $(Y[k]_0, D)$ is a family of expanded pairs over point.

\begin{definition}\label{defi of log cy pair}
Let $(Y, D)$ be a locally smooth pair, such that $Y$ is Gorenstein with canonical line bundle $\omega_Y$.  
Let $S$ be a scheme,  and $(\calY_S, \calD_S)$  a family of expanded pair over $S$. 

\begin{itemize}

\item $(\calY_S, \calD_S)$ is \emph{log Calabi--Yau (log~CY) over $S$} if $\omega_{\calY_S / S} (\calD_S) \cong \calO_{\calY_S}$.

\item In particular, if $S = \Spec \C$, where $(\calY_S, \calD_S) = (Y[k]_0, D[k]_0)$ for some $k\geq 0$, we say that $(Y[k]_0, D[k]_0)$ is \emph{log CY} if it is log CY over $\Spec \C$.

\end{itemize}
 
\end{definition}

If $(Y, D)$ is log CY, adjunction formula implies that $\omega_D \cong \calO_D$.

\begin{prop}\label{prop on anti can div}
If $(Y, D)$ is a smooth log CY pair,
then any family $\calY_S \to S$ of expanded pairs  with pullback divisor $\calD_S$ is log CY over $S$.  
In particular, the standard families $(Y[k], D[k])$ are log CY over $\bbA^k$; the universal pair $(\calY, \calD)$ is log CY over the base $\calA$; the expanded pairs $(Y[k]_0, D[k]_0)$ are log CY. 
\end{prop}

\begin{proof}
Since $\calY_S \to S$ is flat, the log CY property is preserved under arbitrary base change. 
It suffices to prove the log CY property for the standard families $(Y[k], D[k])$ over $\bbA^k$.
Obviously, $$\omega_{Y[k] / \bbA^k} \cong \omega_{Y[k]}. $$ Hence it suffices to prove $(Y[k], D[k])$ is log CY. 
We prove this by induction on $k$.
Recall \cite[pp.~521]{Li1} that $Y[k]$ is constructed as $$Y[k]=\Bl_{D[k-1] \times \{0\} } (Y[k-1] \times \bbA^1), \quad k\geq 1. $$
When $k=1$, $Y[1] = \Bl_{D \times \{0\} } (Y \times \bbA^1)$.  
The map $p: Y[1] \to Y \times \bbA^1$ is the blow-up, and $D[1]$ is the proper transform of $D\times \bbA^1$. 
Let $E$ be the exceptional divisor. 
We have 
\begin{eqnarray*}
\omega_{Y[1]} &\cong& p^* \omega_{Y\times \bbA^1} \otimes \calO (E) \\
&\cong& p^* \calO (-D \times \bbA^1 ) \otimes \calO (E) \qquad\qquad \text{(since $(Y, D)$ is log CY)} \\
&\cong& \calO (-D[1]) \otimes \calO (-E) \otimes \calO (E) \\
&\cong& \calO (-D[1]) . 
\end{eqnarray*}
Therefore, $(Y[1], D[1])$ is log CY and the proposition follows. 
Inductive process is obtained by replacing $Y$ by $Y[k-1]$ in the $k=1$ case above. 
\end{proof}

\subsection{Expanded degenerations}\label{sect on exp dege}    
%Consider the following notion of degenerations. 
\begin{definition}\label{defi of dege}
Let $X$ be a smooth quasi-projective variety over $\bbC$. A map  
$$\pi : X \to \bbA^1$$ 
is a \textit{simple degeneration} if it is flat, with smooth fiber $X_c=\pi^{-1}(c)$ for $c\neq 0$, and central fiber $$X_0 = Y_- \cup_D Y_+,$$ 
where $(Y_\pm, D)$ are smooth pairs (Definition \ref{defi of smooth pair}), and $Y_\pm$ intersect transversally at the singular\footnote{Note that this divisor itself is smooth.} divisor $D=Y_- \cap Y_+ $
in the sense that \'etale locally around $0\in \bbA^1$, the map $\pi$ is of the form $$\Spec \bbC [x, y, t, \vec z] / (xy -t) \to \Spec \bbC [t], $$ 
where $Y_\pm$ are cut out by the equations $(x=0)$ and $(y=0)$ respectively. 
\end{definition}
%$(Y_\pm, D)$ are smooth pairs in the sense of . 

By definition, we automatically have $$N_{D / Y_+} \otimes N_{D / Y_-} \cong \calO_D. $$
For $k\geq 0$, consider the following \emph{expanded degeneration of length} $k$:
$$
X[k]_0 := Y_- \cup_D \underbrace{\Delta \cup_D \cdots \cup_D \Delta}_{k \text{-times}} \cup_D Y_+. 
$$
Similarly, these arise as central fibers of standard families.
There is a projective map 
\begin{equation}\label{equ on contra map2}p: X[k]_0 \to X_0, \end{equation} 
which contracts all $\Delta$'s.

Define the multiplication map 
\begin{equation}\label{multi map}m: \bbA^{k+1} \to \bbA^1, \quad (t_0, \ldots, t_k) \mapsto t_0 \cdots t_k. \end{equation}
Let $(\bbC^*)^k$ acts on $\bbA^1$ trivially and on $\bbA^{k+1}$ such that $m$ is equivariant.

\begin{prop}\label{prop:st_family}$($\cite{Li1, LW},~\cite[Prop.~2.13]{Zhou1}$)$
There is a standard family $$\pi_k: X[k]\to \bbA^{k+1}, $$ which satisfies the following properties. 

\begin{enumerate}

\item $X[k]$ is a smooth quasi-projective variety of dimension $(\dim X + k)$. 
There is a projective map $p_k: X[k] \to X \times_{\bbA^1} \bbA^{k+1}$. 

\item Let $t = (t_0, \ldots, t_k) \in \bbA^{k+1}$ be a closed point. 
If $c = t_0 \cdots t_k\neq 0$, then the fiber of $\pi_k$ over $t$ is isomorphic to $X_c$, to which the restriction of $p_k$ is the identity map; 
if $t_0 \cdots t_k = 0$, and hence $l := \sharp \{i \mid t_i = 0 \} \geq 1$, then the fiber of $\pi_k$ over $t$ is isomorphic to $X[l-1]_0$, to which the restriction of $p_k$ is the contraction map $p$ \eqref{equ on contra map2}. 

\item There is a $(\bbC^*)^k$-action on $X[k]$, such that $\pi_k$ and $p_k$ are equivariant. 
For a closed point $t\in \bbA^{k+1}$ as above with $t_0 \cdots t_k = 0$, the stabilizer group $\Stab_t (\bbA^{k+1} ) \cong (\bbC^*)^{l-1}$ acts on the fiber $X[l-1]_0$ by scaling 
$(l-1)$ bubble components. 
\end{enumerate}
\end{prop}
An expanded degeneration can be obtained by gluing two expanded pairs:
$$
X[k]_0 = Y_-[i] \cup_D Y_+[k-i],
$$
for any $0\leq i\leq k$. 
The following proposition is a description of this fact in terms of standard families.

\begin{prop}$($\cite{Li1, LW},~\cite[Prop.~2.18]{Zhou1}$)$ \label{prop-split}
The composition  $\proj_i \circ\, \pi_k: X[k] \to \bbA^1$ is a simple degeneration, with singular divisor $D_i[k]$ and central fiber $X[k] \times_{\bbA^1} 0 = X[k]_-^i \cup_{D_i[k]} X[k]_+^i$, where\,\footnote{Here the superscript of $Y_+[k-i]^\circ$ means the standard family $Y_+[k-i]$ with \emph{reversed} ordering of the coordinates in $\bbA^{k-i}$.}
$$
D_i[k] \cong \bbA^i \times D \times \bbA^{k-i} , \quad 
X[k]_-^i \cong Y_-[i] \times \bbA^{k-i}, \quad 
X[k]_+^i \cong \bbA^i \times Y_+ [k-i]^\circ. 
$$
\end{prop}
As before, one can define the stack of expanded degenerations by quotienting out the $(\bbC^*)^k$ together with the discrete symmetries, and then taking the direct limit.
\begin{definition}$($\cite{Li1}, \cite[\S 2.1,~2.2]{LW}, \cite[Def.~2.22]{Zhou1}$)$\label{defi of exp dege}
The \emph{stack of expanded degenerations} is defined as $\calC:= \underrightarrow\lim \ [\bbA^{k+1} / \sim ~]$, with \emph{universal expanded degeneration} $\calX := \underrightarrow\lim \ [X[k] / \sim ~]$.
\end{definition}

$\calC$ is a smooth Artin stack of dimension $1$, and $\calX$ is a smooth Artin stack of dimension $\dim X$. 
They both admit a map to $\bbA^1$.
There is a representable and projective map $$p: \calX \to X \times_{\bbA^1} \calC, $$ whose composition with the projection to $\calC$ is the family map 
$$\pi: \calX \to \calC, $$
which is flat and also representable. 

In general, let $S \to \bbA^1$ be a scheme over $\bbA^1$. 
Given a $\bbA^1$-morphism $S \to \calC$, the pullback of the universal family $\pi$ gives a family 
$$\pi_S : \calX_S \to S, $$ 
which we call a \emph{family of expanded degenerations over $S$}.
%In particular, any expanded degeneration $X[k]_0$ is a family of expanded degenerations over point.

\begin{definition}\label{cy simple dege}
Let $\pi: X\to \bbA^1$ be a simple degeneration and let $\pi_S: \calX_S \to S$ be a family of expanded degenerations.
\begin{itemize}

\item $\pi_S: \calX_S \to S$ is called \emph{Calabi-Yau~(CY)}, if the relative canonical bundle is trivial: $\omega_{\calX_S / S} \cong \calO_{\calX_S}$. 

\item In particular $\pi: X \to \bbA^1$ is called \emph{Calabi-Yau~(CY)}, if it is CY over $\bbA^1$; any expanded degeneration $X[k]_0$ is CY if it is CY in the usual sense,~i.e.~$\omega_{X[k]_0} \cong \calO_{X[k]_0}$. 
\end{itemize}
\end{definition}
\begin{prop}\label{prop on canon bdl}
Let $\pi: X\to \bbA^1$ be a CY simple degeneration.
Then
\begin{enumerate}

\item the smooth pairs $(Y_\pm, D)$ in Definition \ref{defi of dege} are log CY; 

\item any family $\calX_S \to S$ of expanded degenerations is CY. 
In particular, the standard famililes $\pi_k: X[k] \to \bbA^{k+1}$ are CY; the universal family $\pi: \calX \to \calC$ is CY; any expanded degeneration $X[k]_0$ is CY.

\end{enumerate}
\end{prop}

\begin{proof}
For (1), it is straightforward to see that 
$$\omega_X \cong \omega_{X / \bbA^1} = \calO_X, \quad \calO_X (Y_-) \otimes \calO_X (Y_+) \cong \calO_X. $$ 
By the definition of simple degeneration, we have 
$$\calO_X (Y_-) |_{Y_+} \cong \calO_{Y_+} (D). $$ 
Adjunction formula then implies that
$$
\omega_{Y_+} \cong \omega_X |_{Y_+} \otimes \calO_X (Y_+) |_{Y_+} 
\cong  \calO_X (-Y_-) |_{Y_+} 
\cong \calO_{Y_+} (-D).
$$
The case for $Y_-$ is similar.

For (2), as in the proof of Proposition \ref{prop on anti can div},  it suffices to show that the standard family $X[k]$ is CY.
Recall \cite{Li1} that $X[k]$ is constructed inductively from $X[k-1]$. 
We prove this by induction on $k$. 
When $k=1$, $X[1]$ is constructed in two steps. 
Firstly, one takes the blow-up $\Bl_{D \times \{0\} } (X \times_{\bbA^1} \bbA^2)$, where the map $\bbA^2 \to \bbA^1$ is the multiplication. 
The exceptional divisor is $\Delta \times_D \Delta$, with normal bundle $\calO_{\Delta \times_D \Delta} (-1, -1)$. 
One then blows down a factor of $\bbP^1$ and obtain $X[1]$. 
The resulting map $$p: X[1] \to X \times_{\bbA^1} \bbA^2$$ is a resolution of singularities, with exceptional divisor $\Delta$, whose normal bundle is $\calO_\Delta (-1)^{\oplus 2}$. 
Therefore, it is a crepant resolution, and $X[1]$ is CY.
Replacing $X$ by $X[k-1]$ in the $k=1$ case gives the inductive process.
\end{proof}

Next we describe the central fiber of the universal family $\calX$. 
The multiplication map \eqref{multi map} gives a map $$\calC\to \bbA^1, $$
whose fiber over $0$ is denoted by 
\begin{equation}\label{equ on c0}\calC_0 := \calC \times_{\bbA^1} 0. \end{equation}
By construction, it is a limit of the standard smooth charts $[\bbA^{k+1} \times_{\bbA^1} 0 / \sim] = [\,\bigcup_{i=1}^{k+1} H_i / \sim]$, where $H_i = (t_i=0) \subset \bbA^{k+1}$ is the $i$-th standard hyperplane. Let 
\begin{equation}\label{equ on c0dagg}\calC_0^\dagger \to \calC_0 \end{equation} 
be the normalization, which means that $\calC_0^\dagger := \underrightarrow\lim\, [\,\coprod_{i=1}^{k+1} H_i / \sim]$. 
Let 
\begin{equation}\label{equ on x0dagg}\calX_0^\dagger := \calX \times_\calC \calC_0^\dagger \end{equation} be the fiber product, which by definition is $\underrightarrow\lim\, [\,\coprod_{i=1}^{k+1} X[k] |_{H_i} / \sim]$. 

A geometric interpretation of $\calX^\dagger_0$ and $\calC^\dagger_0$ can be obtained from Proposition \ref{prop-split}, which implies that $\calC_0^\dagger$ classifies expanded degenerations \emph{with a marked nodal divisor}, and $\calX^\dagger_0$ is the universal divisor.
This marking then makes it possible to split the expanded degeneration at the marked nodal divisor, resulting in two expanded pairs. 

\begin{corollary}
\label{cor-decomp}

There is a canonical isomorphism $\calC_0^\dagger \cong \calA \times \calA$, such that \footnote{$\calY_+^\circ$ is the direct limit of $[Y_+[i]^\circ / \sim]$, where $Y_+ [i]^\circ$ is as in Proposition \ref{prop-split}.}
$$
\calX_0^\dagger \cong (\calY_- \times \calA ) \cup_{\calA \times D \times \calA } (\calA \times \calY_+^\circ). 
$$
\end{corollary}

\subsection{Topological data}

Let $B$ be a scheme, and $\pi: V \to B$ be a quasi-projective scheme over $B$. 
Let $K_c (V / B)$ be the \emph{fiberwise compactly supported} $K$-group,~i.e.~the Grothendieck group of coherent sheaves $F$ on $V$, whose supports $\Supp F$ are proper over $B$.
In particular, there is a well-defined $K$-\textit{theoretic pushforward map} 
\begin{equation}\label{k push}
\pi_*^K: K_c (V / B) \to K_0 (B)
\end{equation}
to the Grothendieck group $K_0 (B)$ of coherent sheaves on $B$. 

\begin{definition}\label{num K-theory}
%As in \cite[\S 2.2]{BCY}, 
$F_1$, $F_2 \in K_c (V / B)$ are called \emph{fiberwise numerically equivalent} and denoted by 
$F_1 \sim_{\mathrm{num}} F_2$, 
if 
$$\pi_*^K (F_1\otimes E) = \pi_*^K (F_2\otimes E)\in K_0 (B), $$ 
for any locally free sheaf $E$ on $V$.
Define the \emph{fiberwise compactly supported numerical $K$-theory} as 
$$
K_c^\mathrm{{num}} (V /B):= K_c(V /B) / \sim_{\mathrm{num}}. 
$$ 
When $B = \pt$, we simply denote it by $K_c^\mathrm{{num}} (V)$.
\end{definition}
%\yl{there should be a statement saying when we consider curves/points classes, numerical equi classes on $X_0$ determinants  $X_t$'s}

%Let $(Y, D)$ be a locally smooth pair, and 
Let $\pi: X\to \bbA^1$ be a simple degeneration, with central fiber $X_0 = Y_- \cup_D Y_+$ and smooth fibers $X_t$ (for any $t\neq 0$).
Denote the closed embeddings:
$$
\iota_{D,\pm}: D\hookrightarrow Y_{\pm}, \quad \iota_D: D \hookrightarrow X_0, \quad \iota_\pm: Y_\pm \hookrightarrow X_0, \quad  i_t: X_t \hookrightarrow X, \quad t\in \mathbb{A}^1. 
$$
\begin{lemma}\label{k pullba}
The $K$-\textit{theoretic Gysin pullback} \cite[Ex.~15.1.8]{Fu}:
$$i_t^*: K_c(X/ \bbA^1)\to K_c(X_t), \quad \iota_{D,\pm}^*:  K_c(Y_\pm)\to K_c(D)$$
descends to the corresponding numerical $K$-theories
\begin{equation}\label{equ on k pullback}i_t^*: K^\mathrm{{num}}_c(X/ \bbA^1)\to K^\mathrm{{num}}_c(X_t), \quad \iota_{D,\pm}^*:  K^\mathrm{{num}}_c(Y_\pm)\to K^\mathrm{{num}}_c(D). 
\end{equation}
\end{lemma}
\begin{proof}
We prove the case for $i_t^*$, the other one is similar. 
Take $[F_i]\in K_c(X/ \bbA^1)$ ($i=1,2$) such that 
\begin{equation}\label{equ on pstarx}\pi_*^K (F_1\otimes E) = \pi_*^K (F_2\otimes E)\in K_0 (B), \end{equation} 
for any locally free $E$ on $X$. 
Then for any locally free sheaf $V$ on $X_t$, 
$$\chi(X_t,i_t^*F_i\otimes V)=\chi(X,i_{t*}(i_t^*F_i \otimes V))=\chi(X,F_i\otimes i_{t*} V), \quad i=1,2. $$
Since $X$ is smooth, $i_{t*} V$ has a resolution by a perfect complex $E^\bullet$ of locally free sheaves. 
By \eqref{equ on pstarx} and proper pushward to a point, we know $\chi(X,F_1\otimes i_{t*}V)=\chi(X,F_2\otimes i_{t*} V)$.
\end{proof}

Using the above Gysin pullback, we define the following: 
\begin{definition}\label{Defn split data}
Given $P_0\in K_c^{\mathrm{num}} (X_0)$, a \emph{splitting datum} for $P_0$ is a pair $(P_-, P_+)$, with $P_\pm \in 
K_c^\mathrm{{num}} (Y_\pm)$, such that 
%Given $P\in K_c^{\mathrm{num}} (X_0)$, and $P_t \in K_c^{\mathrm{num}} (X_t)$ (for some $t\neq 0$), a \emph{splitting datum} for $P$ is a pair $(P_-, P_+)$, with $P_\pm \in %K_c^\mathrm{{num}} (Y_\pm)$, such that 
\begin{equation} \label{equ on split data}
P_+ |_D = P_- |_D, \qquad \iota_{+*} P_+ + \iota_{-*} P_- - \iota_{D*}(P_+ |_D) = P_0, 
\end{equation} 
where $(-)|_D$ denotes the $K$-\textit{theoretic Gysin pullback} from $Y_\pm$ to $D$ \eqref{equ on k pullback}.
We denote the \textit{set of all splitting data} of $P_0$ by $\Lambda^{P_0}_{spl}$. 
%\footnote{It is easy to check that such pullback preserves numerically equivalent relations.}. 
\end{definition}
Recall that a $K$-group has a topological filtration:  
$$K_{0, \leq i} (V) \subset K_0 (V), $$
which consists of classes $[F]$ whose \textit{support has dimension at most} $i$. 
The filtration restricts to give corresponding compactly supported classes and descends to numerical $K$-theory classes: 
\begin{equation}\label{k class with dim} K_{c, \leq i} (V), \quad K_{c, \leq i}^{\mathrm{num}} (V). \end{equation}

\subsection{Relative Hilbert stacks}\label{sect on rel hilb stack}

\begin{definition}\label{def of normal}
Let $W$ be a quasi-projective scheme, $D \subset W$ be a closed subscheme, and $F$ be a coherent sheaf on $W$. 
$F$ is said to be \emph{normal} to $D$ if $\Tor_1^{\calO_W} (F, \calO_D) = 0$. 
\end{definition}

Let $(Y, D)$ be a smooth pair, and $\pi: X \to \bbA^1$ be a simple degeneration.
Let $Z$ be a properly supported closed subscheme on the expanded pair $Y[k]_0$ or the expanded degeneration $X[k]_0$.  
We treat $Z$ as an object in the Hilbert scheme by considering the quotient sheaf $ \calO_{Y[k]_0} \twoheadrightarrow \calO_Z$. 
An \emph{automorphism} of $Z$ is understood as an automorphism induced by the $(\bbC^*)^k$-action.

\begin{definition} \label{Defn-stable}
A properly supported closed subscheme $Z$ in $Y[k]_0$ or $X[k]_0$ is called \emph{stable} if:

\begin{itemize}

\item $\calO_Z$ is normal to every singular divisor $D$ in $Y[k]_0$ or $X[k]_0$, as well as the relative divisor $D[k]_0$ (in the case of $Y[k]_0$);

\item the automorphism group $\Aut (\calO_{Y[k]_0} \twoheadrightarrow \calO_Z)$ or $\Aut (\calO_{X[k]_0} \twoheadrightarrow \calO_Z)$ is finite.  

\end{itemize}
\end{definition}

The same definition can be made for families of expanded pairs and expanded degenerations, where we require the stable condition on geometric fibers, and the flatness of $Z$ over the base. 
In particular, one can consider stable subschemes on the universal families $\calY \to \calA$ and $\calX \to \calC$. 

\begin{definition}\label{defin on rel hilb stack}
Denote by $\Hilb (Y, D)$ the stack parameterizing families of stable subschemes $Z$ in the universal family $\calY / \calA$. 
Similarly, denote by $\Hilb (X/\bbA^1)$ the corresponding stack for $\calX / \calC$.

\begin{enumerate}
\item Given a numerical $K$-theory class $P \in K_c^{\mathrm{num}} (Y)$, let $\Hilb^P (Y, D)$  be the substack parameterizing stable 
subschemes $Z$ such that $p^K_*[\calO_Z] = P$. 
Here $p^K_*$ is the $K$-theoretic proper pushforward \eqref{k push} of the contraction map $p: Y[k]_0 \to Y$.

\item Suppose that we are given numerical $K$-theory classes $P_t \in K_c^\num (X_t)$ for all $t\in \bbA^1$. 
Recall that the standard families $X[k] \to \bbA^{k+1}$ has fibers either of the form $X[k]_0$ ($k\geq 0$), or isomorphic to $X_t$ with $t\neq 0$.
Let $\Hilb^{P_t} (X/\bbA^1)$ be the substack parameterizing stable 
subschemes $Z$ such that 
%\yl{change notation to $\Hilb^{P_t}(X/\bbA^1)$}
\begin{itemize}
\item if $Z$ is a subscheme in a fiber of the form $X[k]_0$ ($k\geq 0$), then $p^K_* [\calO_Z] =P_0$. Here $p^K_*$ is the $K$-theoretic proper pushforward \eqref{k push} of the contraction map $p: X[k]_0 \to X_0$.

\item if $Z$ is a subscheme in a fiber isomorphic to $X_t$ with $t\in \bbA^1 \backslash \{0\}$, then $[\calO_Z] = P_t$. 
\end{itemize}

\end{enumerate}
\end{definition}
\begin{remark}\label{rmk on nonempty mod}
To have \textit{nonempty} moduli stack $\Hilb^P (X/\bbA^1)$, there is necessarily a $P_\eta\in K_c^{\num}(X_\eta)$ where $\eta$ is the generic point of $\bbA^1$, so that for any  closed point $t\in\bbA^1$, $P_t\in K_c^\num (X_t)$ is the $K$-theoretic specialization of  $P_\eta$.
%The topological data given in (2) above are a family of numerical $K$-theory classes on $X_t$ ($t\in \bbA^1$). Equivalently, they can be described in terms of $K$-theoretic pushforward along the contraction map $p_k: X[k] \to X \times_{\bbA^1} \bbA^{k+1}$, and then restriction to fibers $X_t$. Therefore, to have \textit{nonempty} moduli stack $\Hilb^P (X/\bbA^1)$, it is necessary that there exists \emph{some} class $$P_{\bbA^1} \in K^\num_c (X /\bbA^1), $$ such that $P_t = i_t^* P_{\bbA^1}$. In particular, the class $i_{t*} P_t$ is constant on $t$. When we fix $P \in K_{c,\leqslant 0}^{\mathrm{num}} (Y)$ (resp.~$P_t \in K_{c,\leqslant0}^\num (X_t)$ for all $t\in \bbA^1$), having nonempty moduli stack $\Hilb^P (Y, D)$ (resp.~$\Hilb^{P_t} (X/\bbA^1)$) is equivalent to fixing a nonnegative integer $n\in \mathbb{Z}_{\geqslant 0}$.
%{\color{red} Moreover, since $i_{t*}$ is injective \zz{need to check}, classes $P_t$ are uniquely determined by $P_0$. }
%hence the notation $\Hilb^P (X / \bbA^1)$. 
\end{remark}

\begin{remark}
Although we focus on the case of Hilbert stacks, most of the constructions also work for moduli of stable quotients as \cite[Def.~4.2]{LW}, \cite[Def.~3.18]{Zhou1}, which includes many other interesting 
moduli stacks of objects in derived categories (e.g.~$Z_t$-stable pairs considered in \cite{CT1, CT3, CT4}).
\end{remark}
%We prove the following openness result.  
\begin{lemma}\label{lem on open}
$\Hilb^P (Y, D)$ $($resp.~$\Hilb^{P_t} (X/\bbA^1)$$)$ is an open and closed substack of $\Hilb (Y, D)$ $($resp.~$\Hilb(X/\bbA^1)$$)$. 
\end{lemma}

\begin{proof}
We prove the statement for $\Hilb^{P_t} (X/\bbA^1)$; the case for $\Hilb^P (Y, D)$ is simpler. 
Let $\calX_S \to S$ be a family of expanded degenerations, and $\calZ_\calS$ be a family of stable subschemes in $\calX_S$. 
Up to an \'etale base change, we may assume that $\calX_S \to S$ is the base change of a standard family of $X[k] \to \bbA^{k+1}$, i.e. the classifying map $S\to \cC$ factors through 
$S\stackrel{\xi}{\to} \bbA^{k+1} \to \cC$. 
%Denote the first map by $\xi: S\to \bbA^{k+1}$. 
Let $m: \bbA^{k+1} \to \bbA^1$ be the multiplication map. 

For any geometric point $s\in S$, the fiber $\calX_s$ satisfies 
\begin{align*}
\calX_s\cong
\left\{\begin{array}{rcl} X[k]_0 \,\, \mathrm{for}\,\,\mathrm{some}\,\, k\geq 0        &\mathrm{if} \,\,s\in (m \circ \xi)^{-1} (0), \\
& \\ X_t    \quad \quad\quad\quad\quad  &   \mathrm{if} \,\,s\in (m \circ \xi)^{-1} (t), \,\,  t\neq 0.
\end{array} \right. 
\end{align*}
%is isomorphic to $X[k]_0$ ($k\geq 0$), if $s\in (m \circ \xi)^{-1} (0)$; it is isomorphic to $X_t$ ($t\in \bbA^1 \backslash \{0\}$), if $s\in  (m \circ \xi)^{-1} (0)$. 
Let $\calZ_s$ be the corresponding stable subscheme in the fiber $\calX_s$. 
It suffices to prove that for all $s\in (m \circ \xi)^{-1} (0)$ (resp.~$s\in (m \circ \xi)^{-1} (t)$), the numerical $K$-theory class $p^K_* [\calO_{\calZ_s}] \in K_c^\num (X_0)$ (resp.~$[\calO_{\calZ_s}] \in K_c^\num (X_t)$) is locally constant in $s$. 
By definition, it suffices to prove that for any locally free sheaf $E$ on $\calX_s \cong X [k]_0 $ (resp.~$\calX_s \cong X_t$), the function 
$$
\chi (X_0 , p_* \calO_{\calZ_s}  \otimes E ) = \chi (X_0 , p_* ( \calO_{\calZ_s}  \otimes p^* E ) ) = \chi (X[k]_0 ,  \calO_{\calZ_s}  \otimes p^* E ), \,\,\, (\text{resp.~} \chi (X_t, \calO_{\calZ_s}  \otimes E ) )
$$
is locally constant in $s$.  
But this follows from the flatness of $\calZ_S$ over $S$. 
\end{proof}

\begin{prop}[{\cite[Thms.~4.14, 4.15]{LW}}] 
\label{prop-properness} ${}$ 

 \begin{enumerate}
\item $\Hilb (Y, D)$ $($resp.~$\Hilb (X/\bbA^1)$$)$ is a Deligne--Mumford stack, locally of finite type $($resp.~locally of finite type over $\bbA^1$$)$.

\item If $Y$ is proper $($resp.~$\pi: X\to \bbA^1$ is proper$)$, then $\Hilb^P (Y, D)$ $($resp.~$\Hilb^{P_t}  (X/\bbA^1)$$)$ is proper $($resp.~proper over $\bbA^1$$)$.
\end{enumerate}
\end{prop}

\begin{remark}\label{rmk on proper}
In practice, the following cases are often of particular interest, where neither the target spaces or moduli spaces are proper: 
%but torus equivariantly proper:
\begin{itemize}
\item local curves (see Eqn.~\eqref{equ of local curve}); 

\item toric varieties (see e.g.~\cite{CK1, CK2} for discussions on toric Calabi-Yau 4-folds).
\end{itemize}
The standard way out in such situations is to find a torus $T$ acting on $(Y,D)$ (resp.~$T$ acting on $X$, trivially on $\bbA^1$ and $\pi: X\to \bbA^1$ being equivariant), such that the fixed loci $\Hilb^P (Y, D)^T$ (resp.~$\Hilb^{P_t}  (X/\bbA^1)^T$) are \textit{proper}. 
% $T$ naturally induces an action on moduli spaces
Enumerative invariants can then be defined via $T$-equivariant localization. 
This is indeed the case in the two situations mentioned above. 
%In the case of local curves (resp.~toric varieties), we can take $T$ to be the torus which acts on fibers, or its Calabi-Yau subtorus (resp.~the full torus or its Calabi-Yau subtorus)
The reason is that one can then find a natural smooth $T$-equivariant compactification $\bar Y$ of $Y$, where the closure $\bar D$ of $D$ is still smooth,
such that the $T$-fixed loci $\Hilb^P (Y, D)^T$ form a \textit{union of connected components} of the $T$-fixed locus $\Hilb^P (\bar Y, \bar D)^T$, which is proper by Proposition \ref{prop-properness}, similarly for $X/\mathbb{A}^1$ case.  

For example, in the case of local curves, we can take $T$ to be the torus which acts on fibers, or its Calabi-Yau subtorus.

%\yl{say sth on calabi-yau torus and how the torus act on relative moduli}
\end{remark}
\iffalse
Given $P \in K_c^{\mathrm{num}} (Y)$, let $P|_D\in K_c^{\mathrm{num}} (D)$ be its restriction to $D$. 
The normal condition of $Z$ induces a well-defined restriction map
$$
 \Hilb^P (Y, D) \to \Hilb^{P|_D} (D)\times \calA, \qquad Z \mapsto Z \cap D, 
$$
where $\Hilb^{P|_D} (D)$ denotes the Hilbert scheme of $D$ with Hilbert polynomial $P|_D$.
\fi
Several variants are needed for applications in the present paper.
For $P \in K_c^{\mathrm{num}} (Y)$, there is a 
 stack $\calA^P$ parameterizing expanded pairs \emph{with a weight structure}, whose \emph{total continuous weight} is $P$. 
A weight structure on an expanded pair assigns 
each irreducible component and each singular divisor an element in $K_c^\mathrm{{num}}(Y)$, so that the summation of all weights is the fixed $K$-theory class $P$. The assignment is subject to a continuity condition.  We refer an interested reader to \cite[Def.~7.2]{Zhou1} and \cite[Def.~2.14]{LW} for details.  
Similarly, one has the stack $\calC^{P_t}$. There are \'etale maps forgetting the weight structure: 
$$\calA^P\to \calA, \quad  \calC^{P_t}\to \calC. $$
Denote the universal families over $\calA^P$ and $\calC^{P_t}$ by $\calY^P$ and $\calX^{P_t}$, with family maps 
\begin{equation}\label{equ on dec fam map} \calX^{P_t} \to \calC^{P_t}, \quad  \calY^P \to \calA^P. \end{equation}
The Hilbert stack $\Hilb^P (Y, D)$ in Definition \ref{defin on rel hilb stack} can also be described as 
the stack parameterizing families of stable subschemes $Z$ in the universal family $\calY^P / \calA^P$ satisfying 
$p^K_*[\calO_Z] = P$. This is because objects in $\Hilb^P (Y, D)$ have fixed topological data $P$, so $\Hilb^P (Y, D)\to \calA$ factors through $\calA^P$. 
 Similarly for $\Hilb^{P_t} (X/\bbA^1)$. 

We also have the weighted version of \eqref{equ on c0}, \eqref{equ on c0dagg}, \eqref{equ on x0dagg}: 
 $$\calC_0^{P_0} := \calC^{P_t} \times_{\bbA^1} 0, \quad \calC_0^{\dagger,P_0}\to \calC_0^{P_0}, \quad 
 \calX_0^{\dagger,P_0} := \calX^{P_t} \times_{\calC^{P_t}} \calC_0^{\dagger,P_0}.  $$
%Also, we have the weighted version for $\calX_0^{\dagger, P_0} \to \calC_0^{\dagger, P_0}$.
For each \emph{splitting datum} $(P_-, P_+)$ of $P_0$ in Definition \ref{Defn split data}, there is an open and closed substack $\calC_0^{\dagger, (P_-, P_+)} \subset \calC_0^{\dagger, P_0}$, and let $$\calX_0^{\dagger, (P_-, P_+)}= \calX^{P_t} \times_{\calC^{P_t}} \calC_0^{\dagger, (P_-, P_+)}. $$
Corollary \ref{cor-decomp} can be refined as follows.

\begin{corollary}[{\cite[Prop.~2.13]{LW}, \cite[Prop.~7.4]{Zhou1}}]

\label{cor-decomp-P}

Given a splitting datum $(P_-, P_+)$ of $P_0$, there is a canonical isomorphism $\calC_0^{\dagger, (P_-, P_+)} \cong \calA^{P_-} \times \calA^{P_+}$, such that 
\begin{equation}\label{equ on iso of x0 and ypm}
\calX_0^{\dagger, (P_-, P_+)} \cong (\calY_-^{P_-} \times \calA^{P_+} ) \cup_{\calA^{P_-} \times D \times \calA^{P_+} } (\calA^{P_-} \times \calY_+^{\circ, P_+}).  
\end{equation}
\end{corollary}

Denote by $\Lambda^P_{spl}$ the set of all splitting data of $P$.

\begin{lemma}[{\cite[Prop.~2.19]{LW}, \cite[Prop.~7.5]{Zhou1}}]

\label{lem on sect of line bdl}
There are canonically constructed line bundles with sections $(L_{P_-, P_+}, s_{P_-, P_+} )$ over $\calC^{P_t}$, indexed by splitting data $(P_-, P_+) \in \Lambda_{spl}^{P_0}$, such that
\begin{itemize}
\item we have
$$
\bigotimes_{(P_-, P_+) \in \Lambda^{P_0}_{spl} } L_{P_-, P_+} \cong \calO_{\calC^{P_t}}, \qquad \prod_{(P_-, P_+) \in \Lambda^{P_0}_{spl} } s_{P_-, P_+} \cong \pi^{*} t, 
$$
where $\pi: \cC^{P_t} \to \bbA^1 = \Spec \bbC[t]$ is the canonical map; 

\item $\cC_0^{\dagger, (P_-, P_+)}$ is the closed substack in $\cC^{P_t}$ defined by $s_{P_-, P_+} = 0$. 
\end{itemize}
\end{lemma}
The lemma leads to a similar decomposition of Hilbert stacks. 
More precisely, the Hilbert stack 
$$
\Hilb^{P_0}  (X/\bbA^1)_0 := \Hilb^{P_t}  (X/\bbA^1) \times_{\bbA^1} 0 
$$
parameterizes stable closed subschemes on singular expanded degenerations. 
We can consider its normalization $\Hilb^{P_0}  (X/\bbA^1 )_0^\dagger$, which parameterizes stable closed subschemes in the family $\calX_0^{\dagger, P_0} \to \calC_0^{\dagger, P_0}$. 
There are open and closed stacks $$\Hilb^{(P_-, P_+)} (X/\bbA^1)_0^\dagger \subset \Hilb^{P_0} (X/\bbA^1 )_0^\dagger, $$  
indexed by all possible splitting data of $P_0$, whose union is $\Hilb^{P_0} (X/\bbA^1 )_0^\dagger$. 
Since $\Hilb^{P_t}  (X/\bbA^1)$ is of finite type, so is $\Hilb^{P_0} (X/\bbA^1 )_0^\dagger$, therefore the above union is a finite union. 

For each splitting datum, we have the following decomposition.
\begin{prop}[{\cite[Prop.~2.20]{LW}, \cite[Prop.~7.4]{Zhou1}}]
\label{prop on iso of two mod}
The decomposition in Corollary \ref{cor-decomp-P} induces the following commutative diagram
$$
\xymatrix{
\Hilb^{P_-} (Y_-, D) \times_{\Hilb^{P_- |_D} (D)} \Hilb^{P_+} (Y_+, D) \ar[r]^{\quad \quad \quad \quad \,\,\,\, \cong} \ar[d] & \Hilb^{(P_-, P_+)} (X/\bbA^1)_0^\dagger \ar[d] \\
\calA^{P_-} \times \calA^{P_+} \ar[r]^\cong & \calC_0^{\dagger, (P_-, P_+)}. 
}
$$
\end{prop}

An alternative viewpoint towards the Hilbert scheme is to consider it as a moduli of ideal sheaves. 
The following lemma shows that it also works for in the relative setting.

\begin{lemma} \label{Lemma-Hilb-ideal}
Let $F$ be a torsion-free sheaf of rank $1$ with trivial determinant on $X[k]_0$, which is normal to all singular divisors $D$.
Then there exists a closed subscheme $Z \subset X[k]_0$ of codimension $\geq 2$, normal to all singular divisors, such that $F \cong I_Z$ is its ideal sheaf. 
\end{lemma}

\begin{proof}
We first prove it on a smooth quasi-projective variety $W$.
Given a torsion-free sheaf $F$ of rank $1$, consider its embedding into double dual $F \hookrightarrow F^{\vee\!\vee}$. 
Being a reflexive sheaf of rank $1$,  $F^{\vee\!\vee}$ is a line bundle up to codimension $3$, which implies that it is actually a line bundle by the proof of 
\cite[Lem.~6.13]{Kol}. 
It follows that $F = I_Z$ for some $Z$ with $\codim Z \geq 2$, since $\det F = \calO_W$, and $F \hookrightarrow F^{\vee\!\vee}$ is an isomorphism up to codimension $2$.

Now consider $F$ on $X[k]_0$. 
For simplicity, we may assume that $k=0$, i.e. $X[k]_0 = X_0 = Y_- \cup_D Y_+$. 
The case for a general $k$ is similar. 
Since $F$ is normal to $D$, we have a short exact sequence $$0 \to F \to F |_{Y_-} \oplus F |_{Y_+} \to F |_D \to 0.$$ 
Now by \cite[Lem.~3.4]{Zhou1}, both $F |_{Y_\pm }$ are torsion-free (and also of rank $1$ with trivial determinant).
Therefore, there exists $Z_\pm \subset Y_\pm$, such that $F |_{Y_\pm} \cong I_{Z_\pm}$. 
It is straightforward to see that the normality of $Z_\pm$ to $D$ is equivalent to that of $I_{Z_\pm}$ to $D$. 
$Z_\pm$ then glue up to $Z \subset X_0$, and we have $F \cong I_Z$. 
\end{proof}

\section{Shifted symplectic structures}
We construct \textit{shifted symplectic structures} and \textit{Lagrangian structures} (in the sense of \cite{PTVV, CPTVV}) on the derived enhancement of (relative) Hilbert stacks constructed in the previous section. They will be used to construct \textit{virtual structures} on those moduli stacks in later sections. 

\iffalse
The derived fiber product (or homotopy fiber product) of two maps $X\to Z$, $Y\to Z$ between derived stacks is denoted by $X\times^{\bfL}_Z Y$ or simply by $X\times_Z Y$ if there 
is no confusion. 
For a map $f: X\to Y$ between derived stacks, the relative tangent (resp.~cotangent) complex is denoted by $\bbT_f$ (resp.~$\bbL_f$) or $\bbT_{X/Y}$ (resp.~$\bbL_{X/Y}$) 
when $f$ is understood from the context. Operations such as $f_*, f^*,\otimes$ are to be understood in the derived sense unless stated otherwise.
In the present paper, we make the convention that 
all derived Artin stacks are assumed to be locally of finite presentation. 
\fi

From now on, we will always impose the ``codimension $\geq 2$" assumption. 
More precisely, we consider the open and closed substack $\mathrm{Hilb}_{\codim \geq 2} (X/\bbA^1)$ (resp.~$\mathrm{Hilb}_{\codim \geq 2} (Y, D)$) parameterizing subschemes $Z$ with codimension $\geq 2$. 
To avoid heavy notations, we will omit the subscripts and denote them still by $\mathrm{Hilb} (X/\bbA^1)$ (resp.~$\mathrm{Hilb} (Y, D)$).

\subsection{Shifted symplectic structure on $\textbf{Hilb}(X/\bbA^1)/\cC$}
%In this section, we work in the following setting: 
%\begin{setting}\label{setting 1}
Let $\pi : X \to \bbA^1$ be a Calabi-Yau simple degeneration (Definitions \ref{defi of dege}, \ref{cy simple dege})
\iffalse
\begin{itemize}
\item  $\omega_{\pi}\cong \oO_X$ and $\dim_\mathbb{C} X=5$.
\item  $D\in |K_{Y_{\pm}}^{-1}|$ are anticanonical divisors of $Y_{+}$ and $Y_{-}$. 
\end{itemize}
\fi
%\end{setting}
and 
$$\pi: \calX\to \calC$$
be the induced map from the universal expanded degeneration $\calX$ to the stack $\calC$ of expanded degenerations. By Proposition \ref{prop on canon bdl}, we know 
$$\omega_{\calX/\calC}\cong \oO_{\calX}. $$
Let $\textbf{RPerf}$ be the derived (classifying) stack of perfect complexes of quasi-coherent sheaves, and $\Omega_{\textbf{RPerf}}$ be the natural 2-shifted symplectic form on $\textbf{RPerf}$
\cite[Thm.~0.3]{PTVV}. Consider the derived mapping stack over $\calC$: 
$$\textbf{RPerf}(\calX/\calC):=\textbf{Map}_{\mathrm{dst}/\calC}(\calX,\textbf{RPerf}\times \calC), $$
where we treat classical stacks $\calX$ and $\calC$ as derived stacks by the natural inclusion functor. 
%This mapping stack is a derived Artin stack locally of finite presentation over $\calC$ by Lurie's representability theorem \cite[\S 2.2.6.3]{TV} (see also \cite[pp.~311]{PTVV}). 

There is a determinant map (e.g.~\cite[\S 3.1]{STV}):
\begin{equation*}\det: \textbf{RPerf}(\calX/\calC)\to \textbf{RPic}(\calX/\calC) \end{equation*}
to the derived moduli stack of line bundles on $\calX/\calC$. The following homotopy pullback diagram
\begin{equation}\label{equ on def trless part}
\xymatrix{
\textbf{RPerf}(\calX/\calC)_{0} \ar[r]^{ } \ar[d] \ar@{}[dr]|{\Box} & \{\oO_{\calX}\} \ar[d] \\
\textbf{RPerf}(\calX/\calC) \ar[r]^{\det} &  \textbf{RPic}(\calX/\calC) 
}
\end{equation}
defines the derived moduli stack $\textbf{RPerf}(\calX/\calC)_{0}$ of perfect complexes on $\calX/\calC$ with trivial determinant. 
Let $\textbf{RPerf}(\calX/\calC)_{0, \rk=1}$ be the open and closed substack of perfect complexes with (virtual) rank $1$.
 
\begin{lemma} \label{Lemma-open-Hilb-Perf}
The classical truncation 
$$\mathrm{RPerf}(\calX/\calC)_{0, \rk=1}:=t_0\left(\textbf{\emph{RPerf}}(\calX/\calC)_{0, \rk=1} \right)$$
contains $\Hilb (X/\bbA^1)$ as an open substack. 
\end{lemma}

\begin{proof}
Let $\mathrm{Perf}(\calX / \calC)^{\geq 0}_{0, \rk=1}$ be the open substack of $\mathrm{RPerf}(\calX/\calC)_{0, \rk=1}$ consisting of perfect complexes $F$, such that $\mathrm{Ext}^i (F, F) = 0$ for $i<0$. 
We will prove that $\mathrm{Perf}(\calX / \calC)^{\geq 0}_{0, \rk=1}$ contains $\Hilb (X/\bbA^1)$ as an open substack. 
Let $\calX \to S$ be a family of expanded degenerations over $S$, and $Z \subset \calX$ be a family of stable subschemes, which gives an object in $\Hilb (X/\bbA^1) (S)$. 
According to \cite[Cor.~5.9]{Wu} (see also \cite[Cor.~2.9]{Wu2}), the ideal sheaf $I_Z$ is a perfect complex fiberwisely. 
Since we've assumed that $\codim Z \geq 2$, we obtain a map
$$
\Hilb (X/\bbA^1) \to \mathrm{Perf}(\calX / \calC)^{\geq 0}_{0, \rk=1},  
$$
sending $Z$ to $I_Z$.
By Lemma \ref{Lemma-Hilb-ideal}, an ideal sheaf $I_Z$ in $\Hilb (X/\bbA^1)$ can be characterized as a torsion-free sheaf of rank $1$ with trivial determinant. 
Choosing a relatively ample line bundle on $X/C$, we may replace the ``torsion-free" condition by ``Gieseker stable". 
The substack $\Hilb (X/\bbA^1)$ then consists of perfect complexes whose Tor-amplitudes concentrate at degree $0$, i.e.~they are actually coherent sheaves, as well as properly supported, Gieseker stable and stable (in the sense of Definition \ref{Defn-stable}).
However, these are all open conditions\footnote{For the openness of properly supported-ness, see \cite[\S 4.a]{Gro}.}. 
Hence we obtain an open embedding.
\end{proof}
 
 There is a derived enhancement $\textbf{Hilb} (X/\bbA^1)$ of $\Hilb (X/\bbA^1)$ 
 making the following diagram  
\begin{equation}\label{equ def derive enh of hilb(X/A)}
\xymatrix{
\Hilb (X/\bbA^1) \ar[r]^{ } \ar[d] \ar@{}[dr]|{\Box} & \mathrm{RPerf}(\calX/\calC)_{0} \ar[d] \\
\textbf{Hilb} (X/\bbA^1) \ar[r]^{} &  \textbf{RPerf}(\calX/\calC)_{0}
}
\end{equation}
a homotopy pullback diagram \cite[Prop.~2.1]{STV}.

Similarly, fix $K$-theory classes $P_t\in K_{c, \leq n -2}^{\mathrm{num}} (X_t)$ for all $t\in \mathbb{A}^1$, 
we also have the open substack $\textbf{Hilb}^{P_t}(X/\bbA^1)$ of $\textbf{Hilb}(X/\bbA^1)$ which gives 
a derived enhancement of $\Hilb^{P_t} (X/\bbA^1)$ (Definition \ref{defin on rel hilb stack}).
\begin{theorem}\label{thm on sft symp str}
Write $\dim_{\mathbb{C}} X=n+1$ and choose a trivialization $\mathrm{vol}_{\calX/\calC}: \oO_{\calX}\stackrel{\cong}{\to} \omega_{\calX/\calC}$.  Then as a derived stack over $\calC$, 
$\textbf{\emph{Hilb}}(X/\bbA^1)$ has a canonical $(2-n)$-shifted symplectic structure $\Omega_{\textbf{\emph{Hilb}}(X/\bbA^1)/\calC}$.
Similarly, $\textbf{\emph{Hilb}}^{P_t}(X/\bbA^1)/\calC^{P_t}$ has a canonical $(2-n)$-shifted symplectic structure. 
\end{theorem}
% This is a derived Artin stack locally of finite presentation. Should be proved by considering compactification $\bar{X}$ of $X$ and Hilb stack on $\bar{X}$, then cotangent cpx of 
% Hilb stack on $\bar{X}$ pullback to cotangent cpx of Hilb stack on $X$.  
This follows from a family version of the construction of Preygel \cite{Pre}\footnote{We refer to \cite[Cor.~to~the~main~thm.]{BD} for another approach.}, which adapts the construction of \cite[Thm.~2.5]{PTVV} to a moduli of perfect complexes $(\oO\to F)$ on a quasi-projective Calabi-Yau variety, with fixed determinant and the support of $F$ being proper. We sketch the argument and adapts it to our case. 
\begin{proof}
We deal with $\textbf{{Hilb}}(X/\bbA^1)/\calC$, the decorated case follows similarly.
We use differential calculus on derived stacks following \cite{PTVV, CPTVV, Pre}. 

\textit{Step~1:~Pullback.} 
Let $\mathcal{E}$ be the universal family over $\textbf{RPerf}$. 
%Recall its Chern character \cite[\S 4.2]{TV}  \yl{from $\oO_\calC$?} \[\ch(\mathcal{E}): \calO_\calC\to NC(\textbf{RPerf}\times \calC)/ \calC).\]
%Let $\ch_2(\mathcal{E})$ be the  composition of $\ch(\mathcal{E})$ with the projection \[NC(\textbf{RPerf}\times \calC)/ \calC)\to NC(\textbf{RPerf}\times \calC)/ \calC)(2).\]
By \cite[Theorem~2.12]{PTVV}, the 2-shifted symplectic form on $\textbf{RPerf}$ is  
$$\Omega_{\textbf{RPerf}}=\ch_2(\mathcal{E}): \C\to NC(\textbf{RPerf})(2).$$
Here $NC$ denotes the weighted negative cyclic complex, i.e.~$NC^{w}$ used in \cite[Theorem~2.12]{PTVV}, and 
$\ch_2$ denotes the second Chern character, see  \cite[pp.~317]{PTVV}.

Consider the evaluation map 
$$\ev: \calX\times_{\calC}\textbf{Hilb}(X/\bbA^1) \to \textbf{RPerf}\times \calC, $$
%In the following diagram (where $\ev$ is the evaluation map and $\pi$ is the projection): 
%$$\xymatrix{\textbf{Hilb}(X/\bbA^1)\times_{\calC} \calX  \ar[r]^{\quad \,\, \ev} \ar[d]^{\pi}   &  \textbf{RPerf}\times \calC   \\\textbf{Hilb}(X/\bbA^1), & } $$
and the pullback form 
$$\ev^*\Omega_{\textbf{RPerf}}=\ev^*\ch_2(\mathcal{E}):\C \to NC\left(\calX\times_{\calC}\textbf{Hilb}(X/\bbA^1) \right)(2). $$
By an abuse of notation, we denote its image in the negatively cyclic complex (relative to $\calC$) by 
$$\ev^*\Omega_{\textbf{RPerf}}=\ev^*\ch_2(\mathcal{E}):\C \to NC\left((\calX\times_{\calC}\textbf{Hilb}(X/\bbA^1) )/\calC\right)(2). $$

\textit{Step~2:~Lift to a compactly supported form.}
Let $\mathcal{F}:=\ev^*\mathcal{E}$ and $\mathcal{Z}\subset \calX\times_{\calC}\textbf{Hilb}(X/\bbA^1)$ be the universal substack. It is the 
support of the cone of the trace map 
$$\dR \mathcal{H}om_{\calX\times_{\calC}\textbf{Hilb}(X/\bbA^1)}(\mathcal{F},\mathcal{F}) \stackrel{\tr}{\to} \oO_{\calX\times_{\calC}\textbf{Hilb}(X/\bbA^1) }, $$
which is proper over $\textbf{Hilb}(X/\bbA^1)$. In other words, outside the support $\mathcal{Z}$, the trace map is an isomorphism:
$$\tr: \dR \mathcal{H}om_{\calX\times_{\calC}\textbf{Hilb}(X/\bbA^1)}(\mathcal{F},\mathcal{F})|_{\left(\calX\times_{\calC}\textbf{Hilb}(X/\bbA^1)\right)\setminus \mathcal{Z}}\cong \oO. $$
We now construct a lift of $\ev^*\Omega_{\textbf{RPerf}}$ to a form which is supported on $\mathcal{Z}$. 
%compactly supported form. 

Recall for a map $M_X\to M$ of derived stacks over some base derived stack $B$, and a closed derived substack $K\subset M_X$, with open immersion $i: M_X\setminus K\to M_X$,
we define 
\begin{equation}\label{equ on nck}NC_{K}(M_X/B)(p):=\mathrm{fib}\{i^*: NC_{}(M_X/B)(p)\to NC_{}((M_X\setminus K)/B)(p)\}. \end{equation}
A section of $NC_{K}(M_X/B)(p)$ is called a closed $p$-form on $M_X/B$ supported on $K$. 
%If $K$ is proper over $M$, we also call such a $p$-form compactly supported over $M$. 
%\yl{this integration map overlaps with the one in step 3} The same construction as in \cite[Thm.~3.0.6]{Pre} shows that for any trivialization $$\mathrm{vol}_{M_X/M}: \oO_{M_X}\stackrel{\cong}{\to} \omega_{M_X/M}, $$ defines an integration map \begin{equation}\label{integ map0} \int_{M_X/M}\mathrm{vol}_{M_X/M}\wedge(-):NC_{K}(M_X/B)(p)\to NC(M/B)[-n](p).   \end{equation}

Let $M_X=\calX\times_{\calC}\textbf{Hilb}(X/\bbA^1)$, $M=\textbf{Hilb}(X/\bbA^1)$ and $B=\calC$ in the above. 
%$NC_{c/\textbf{Hilb}(X/\bbA^1)}^0\left((\textbf{Hilb}(X/\bbA^1)\times_{\calC} \calX)/\calC\right)(2)$. 
There is a canonical null-homotopy $\ch_2(\oO_{\pt})\sim 0$ in $NC(\pt)(2)$. 
Pulling back it along the natural projection 
$p: \left(\calX\times_{\calC}\textbf{Hilb}(X/\bbA^1)\right)\setminus \mathcal{Z}\to \pt$ 
gives a null-homotopy in $NC\left(\left(\calX\times_{\calC}\textbf{Hilb}(X/\bbA^1)\right)\setminus \mathcal{Z}\right)(2)$:
\[\ch_2(\mathcal{F}|_{\left(\calX\times_{\calC}\textbf{Hilb}(X/\bbA^1)\right)\setminus \mathcal{Z}})=\ch_2(\oO)=p^*\ch_2(\oO_{\pt})\sim 0. \]
The fiber sequence 
\[NC_\mathcal{Z}\left(\calX\times_{\calC}\textbf{Hilb}(X/\bbA^1)\right)(2)\to NC\left(\calX\times_{\calC}\textbf{Hilb}(X/\bbA^1)\right)(2)\to NC(\left(\calX\times_{\calC}\textbf{Hilb}(X/\bbA^1)\right)\setminus \mathcal{Z})(2)\]
then turns the null-homotopy above into 
a canonical lift
\[\widetilde{\ch_2}(\calF):\C\to NC_\mathcal{Z}\left(\calX\times_{\calC}\textbf{Hilb}(X/\bbA^1)\right)(2).\]
By an abuse of notation, we write this lift in the relative negatively cyclic homology as 
$$\ev^*\Omega_{\textbf{RPerf}}:=\widetilde{\ch_2}(\calF)\in HN_{\calZ}^0\left(\left(\calX\times_{\calC}\textbf{Hilb}(X/\bbA^1)\right)/\calC\right)(2). $$

\textit{Step~3:~Integration.} 
For any quasi-coherent complex $F$ on $\textbf{Hilb}(X/\bbA^1)$, there is a trace map\footnote{We refer to \cite[Thm.~A.0.10]{Pre} for more details. See also 
\cite[Chapter 5, \S 3 \& \S4]{GR} and \cite{LZ} for the six-functor formalism of derived schemes/stacks. }: 
\begin{equation}\label{equ on trace map}\tr_{\calX/\calC}: \pi_{M*}\circ \underline{\dR\Gamma_{\mathcal{Z}}}(\omega_{\calX/\calC}[n]\boxtimes F)\to F, \end{equation}
where $\pi_M: \calX\times_{\calC}\textbf{Hilb}(X/\bbA^1)\to \textbf{Hilb}(X/\bbA^1)$ is the projection and 
$$\underline{\dR\Gamma_{\mathcal{Z}}}: \QCoh(\calX\times_{\calC}\textbf{Hilb}(X/\bbA^1))\to \QCoh(\calX\times_{\calC}\textbf{Hilb}(X/\bbA^1)) $$
is the functor\,\footnote{We refer to \cite{ATJLL} for the definition in the scheme case and \cite{BHV, Sit} for generalization to the stack case.} of section with support $\mathcal{Z}$. By the  argument in \cite[Proof~of~Cor.~A.~0.12]{Pre}, the complex $F$ may be taken to be the de Rham complex.  
The same construction as in \cite[Thm.~3.0.6]{Pre} shows that the map $\tr_{\calX/\calC}$ and a  trivialization 
$$\mathrm{vol}_{\calX/\calC}: \oO_{\calX}\stackrel{\cong}{\to} \omega_{\calX/\calC}, $$
defines an integration map 
\begin{equation}\label{integ map}\int_{\calX/\calC}\mathrm{vol}_{\calX/\calC}\wedge(-):NC_{\calZ}\left((\calX\times_{\calC}\textbf{Hilb}(X/\bbA^1))/\calC\right)(2)\to 
NC_{}\left(\textbf{Hilb}(X/\bbA^1)/\calC\right)[-n](2). \end{equation}
We define
$$\Omega_{\textbf{Hilb}(X/\bbA^1)/\calC}:=\int_{\calX/\calC}\mathrm{vol}_{\calX/\calC}\wedge\left(\ev^*\Omega_{\textbf{RPerf}}\right)\in HN_{}^{-n}\left(\textbf{Hilb}(X/\bbA^1)/\calC\right)(2). $$
The verification of non-degeneracy is a routine check using Grothendieck-Serre duality. 
\end{proof}

\subsection{Lagrangian structure on $(\textbf{Hilb} (Y, D)\to \textbf{Hilb} (\calD))/\calA$}
\label{Sec-lag-str-Hilb(Y,D)}

Let $(Y,D)$ be a smooth log Calabi-Yau pair (Definition \ref{defi of log cy pair}) so that $D$ is a Calabi-Yau variety. 
%such that $D\in |\omega_{Y}^{-1}|$ is an anticanonical divisor of $Y$. 
Consider the induced map 
$$\pi : \calY \to \calA$$
from the universal expanded pair $\calY$ to the stack $\calA$ of expanded pairs, with universal relative divisor 
$$ i_{\calD\to \calY}: \calD\cong D\times \calA  \hookrightarrow  \calY$$
as in Definition \ref{def of exp pair}.  

Consider the restriction map between derived mapping stacks:
%$$r_{\calD\to \calY}: \textbf{RPerf}(\calY/\calA):=\textbf{Map}_{\mathrm{dst}/\calA}(\calY,\textbf{RPerf}\times \calA)\to \textbf{Map}_{\mathrm{dst}/\calA}(\calD,\textbf{RPerf}\times 
%\calA).$$ 
$$r_{\calD\to \calY}: \textbf{RPerf}(\calY/\calA) \to\textbf{RPerf}(\calD/\calA).$$
This map is compatible with determinant maps, i.e. the diagram 
$$
\xymatrix{
\textbf{RPerf}(\calY/\calA) \ar[r]^{r_{\calD\to \calY}} \ar[d]_{\det}   & \textbf{RPerf}(\calD/\calA) \ar[d]^{\det} \\
\textbf{RPic}(\calY/\calA)  \ar[r]^{r_{\calD\to \calY}} &  \textbf{RPic}(\calD/\calA) 
}
$$
is commutative. Therefore we have an induced restriction map
$$r_{\calD\to \calY}: \textbf{RPerf}(\calY/\calA)_0 \to\textbf{RPerf}(\calD/\calA)_0, $$
of derived moduli stack of perfect complexes with trivial determinant. 

Similarly to Lemma ~\ref{Lemma-open-Hilb-Perf}, $\textbf{Hilb}(Y,D)$ is an open substack in $\textbf{RPerf}(\calD/\calA)_{0, \rk=1}$. Further restriction to the open substack $\textbf{Hilb}(Y,D)$ gives 
\begin{equation}\label{rest map}r_{\calD\to \calY}:\textbf{Hilb}(Y,D)\to \textbf{Hilb} (\calD)=\textbf{Hilb} (D)\times \calA.  \end{equation}
For a fixed $K$-theory class $P\in K_{c, \leq \dim_{\mathbb{C}}Y-2}^{\mathrm{num}} (Y)$, we also have the open substack $\textbf{Hilb}^P(Y,D)$ of $\textbf{Hilb}(Y,D)$ which gives 
a derived enhancement of $\Hilb^P (Y,D)$ (Definition \ref{defin on rel hilb stack}). There is a similar restriction map  
\begin{equation}\label{rest map2}r_{\calD\to \calY}:\textbf{Hilb}^P(Y,D)\to \textbf{Hilb}^{P|_D}(D)\times \calA^P. \end{equation}
Fixing a trivialization $\mathrm{vol}_D: \oO_D \stackrel{\cong}{\to} \omega_{D}$, 
by a base-change to $\calD\cong D\times\calA$, we obtain a trivialization 
\begin{equation}\label{choice of volD}\mathrm{vol}_{\calD/\calA}:\oO_{\calD}\cong \omega_{\calD/\calA}.\end{equation}
As in Theorem \ref{thm on sft symp str}, $\mathrm{vol}_D$ induces a $(2-\dim_{\mathbb{C}}D)$-shifted symplectic structure on $\textbf{Hilb}(D)$. Pulling back to 
$(\textbf{Hilb}(D)\times \calA)/\calA$, we obtain a $(2-\dim_{\mathbb{C}}D)$-shifted symplectic structure on $(\textbf{Hilb}(D)\times \calA)/\calA$.

%In this section, we work in the following setting: 
%\begin{setting}\label{setting 2}
%Let $(Y,D)$ be a smooth pair as in Definition \ref{defi of smooth pair} 
%such that $D\in |\omega_{Y}^{-1}|$ is an anticanonical divisor of $Y$. 
%\begin{itemize} \item  $D\in |K_{Y}^{-1}|$ is an anticanonical divisor of $Y$. \end{itemize}
%\end{setting}
We have a commutative diagram
\begin{equation}\label{comm diag on lag evmap}
\xymatrix{
\calD\times_{\calA}\textbf{Hilb}(Y,D) \ar[r]^{i_{\calD\to \calY}} \ar[d]_{r_{\calD\to \calY}}  & \calY\times_{\calA}\textbf{Hilb}(Y,D) \ar[d]^{\ev_{Y,D}} \\
\calD\times_{\calA}(\textbf{Hilb}(D)\times \calA)  \ar[r]^{\quad\quad \ev_{D}} &  \textbf{RPerf}\times \calA, }
\end{equation}
where $\ev_{Y,D}, \ev_{D}$ are evaluation maps. 
By Proposition \ref{prop on anti can div}, $\calD\in |\omega_{\calY/\calA}^{-1}|$ is an anticanonical divisor of $\calY/\calA$, therefore we have the following fiber sequence
\begin{equation}\label{eqn:ses}
      \omega_{\calY/\calA}\to \oO_{\calY}\to i_{\calD\to \calY*}\oO_{\calD}.
\end{equation}

\begin{theorem}\label{thm on lag}
The sequence \eqref{eqn:ses} induces a Lagrangian structure on the map \eqref{rest map}  relative to $\calA$. 
Similarly, it induces a Lagrangian structure on the map \eqref{rest map2}  relative to $\calA^P$.
\end{theorem}
This follows from adapting the argument of \cite[\S 2.2.3 \& \S 3.2.1]{Cal}\footnote{Again we refer to \cite[Cor.~to~the~main~thm.]{BD} for another approach.} to the family case. 
\begin{proof}
We prove the claim for \eqref{rest map}, the other case follows similarly. 
Let 
$$\mathcal{Z}_{Y,D}\subset \calY\times_{\calA}\textbf{Hilb}(Y,D)$$ be the universal substack. It is the support of the cone of the trace map ($\mathcal{F}_{Y,D}:=\ev_{Y,D}^*\mathcal{E}$):
$$\dR \mathcal{H}om_{\calY\times_{\calA}\textbf{Hilb}(Y,D)}(\mathcal{F}_{Y,D},\mathcal{F}_{Y,D}) \stackrel{\tr}{\to} \oO_{\calY\times_{\calA}\textbf{Hilb}(Y,D)}, $$
which is proper over $\textbf{Hilb}(Y,D)$. 
We have commutative diagrams (with three boxes being homotopy pullback diagrams): 
\begin{equation}\label{comm diag on lag evmap2}
\xymatrix{
\mathcal{Z}_{D} \ar[d]_{ } \ar@{}[dr]|{\Box}  & \mathcal{Z}  \ar[l]^{ } \ar[r]^{ }  \ar[d]_{ } \ar@{}[dr]|{\Box} & \mathcal{Z}_{Y,D}  \ar[d]^{ } \\
\calD\times_{\calA} (\textbf{Hilb}(D)\times \calA) \ar[d]_{\pi_{M_D}}  \ar@{}[dr]|{\Box} & \calD\times_{\calA} \textbf{Hilb}(Y,D) \ar[d]_{\pi^{\calD}_{M_{Y,D}}} \ar[l]_{\quad r_{\calD\to \calY}} \ar[r]^{i_{\calD\to \calY}\,} & \calY\times_{\calA}\textbf{Hilb}(Y,D) \ar[d]_{\pi^{\calY}_{M_{Y,D}}} \\
\textbf{Hilb}(D)\times\calA  & \textbf{Hilb}(Y,D) \ar[l]_{\quad r_{\calD\to \calY}} \ar[r]^{=} & \textbf{Hilb}(Y,D),   
}
\end{equation}
where vertical maps in the lower half of the diagram  are natural projections, and
$\mathcal{Z}_{D}\subset \calD\times_{\calA} (\textbf{Hilb}(D)\times \calA)$ is the universal substack of $\textbf{Hilb}(D)\times \calA$
and $\calZ$ is given by the pullback diagram. 

Base change implies that for any $F\in\QCoh(\textbf{Hilb}(Y,D))$, $G\in \QCoh(\textbf{Hilb}(D)\times\calA)$, we have 
\begin{align}\label{bc equ1}
\pi^{\calY}_{M_{Y,D}*}\underline{\dR\Gamma_{\mathcal{Z}_{Y,D}}}(i_{\calD\to \calY*}\oO_{\calD}\boxtimes F)&\cong 
\pi^{\calY}_{M_{Y,D}*}i_{\calD\to \calY*}\underline{\dR\Gamma_{\mathcal{Z}_{}}}(\oO_{\calD}\boxtimes F) \\  \nonumber
&\cong \pi^{\calD}_{M_{Y,D}*}\underline{\dR\Gamma_{\mathcal{Z}_{}}}(\oO_{\calD}\boxtimes F), \nonumber
\end{align}
\begin{align}\label{bc equ2}
\pi^{\calD}_{M_{Y,D}*}\underline{\dR\Gamma_{\mathcal{Z}_{}}}(\oO_{\calD}\boxtimes r_{\calD\to \calY}^*G) &\cong  
\pi^{\calD}_{M_{Y,D}*}\underline{\dR\Gamma_{\mathcal{Z}_{}}}\,r_{\calD\to \calY}^*(\oO_{\calD}\boxtimes G) \\ \nonumber 
&\cong \pi^{\calD}_{M_{Y,D}*}\,r_{\calD\to \calY}^*\,\underline{\dR\Gamma_{\mathcal{Z}_{D}}}(\oO_{\calD}\boxtimes G) \\ \nonumber
&\cong r_{\calD\to \calY}^*\,\pi_{M_{D}*}\,\underline{\dR\Gamma_{\mathcal{Z}_{D}}}(\oO_{\calD}\boxtimes G) \\ \nonumber
&\overset{\eqref{choice of volD}}{\cong} r_{\calD\to \calY}^*\,\pi_{M_{D}*}\,\underline{\dR\Gamma_{\mathcal{Z}_{D}}}(\omega_{\calD/\calA}\boxtimes G). \nonumber
\end{align}
Applying $\pi^{\calY}_{M_{Y,D}*}\underline{\dR\Gamma_{\mathcal{Z}_{Y,D}}}$ to \eqref{eqn:ses},
we obtain, for any $G\in \QCoh(\textbf{Hilb}(D)\times\calA)$ a commutative diagram
\begin{equation}\label{diag on nullhom}
{\tiny
\xymatrix{
\pi^{\calY}_{M_{Y,D}*}\underline{\dR\Gamma_{\mathcal{Z}_{Y,D}}}(\oO_{\calY}\boxtimes r_{\calD\to \calY}^*G) \ar[r]^{ } \ar[dd]_{ }  & \pi^{\calY}_{M_{Y,D}*}\underline{\dR\Gamma_{\mathcal{Z}_{Y,D}}}(i_{\calD\to \calY*}\oO_{\calD}\boxtimes r_{\calD\to \calY}^*G)  \ar[r]^{ } \ar@{=}[d]^{\eqref{bc equ1},\,\eqref{bc equ2}}    &  \pi^{\calY}_{M_{Y,D}*}\underline{\dR\Gamma_{\mathcal{Z}_{Y,D}}}(\omega_{\calY/\calA}[1]\boxtimes r_{\calD\to \calY}^*G)  \ar[dd]^{\tr_{\calY/\calA}} \\
   & r_{\calD\to \calY}^*\,\pi_{M_{D}*}\,\underline{\dR\Gamma_{\mathcal{Z}_{D}}}(\omega_{\calD/\calA}\boxtimes G)   \ar[d]^{\tr_{\calD/\calA}} &  \\ 
0    \ar[r]^{ } &  r_{\calD\to \calY}^*G[-\dim_{\mathbb{C}} D] \ar@{=}[r]   &r_{\calD\to \calY}^*G[-\dim_{\mathbb{C}} D], } }
\end{equation}
where $\tr_{(-)}$ is the trace map \eqref{equ on trace map}. The top and bottom row of the diagram are both fiber sequences. 

Let $\mathcal{E}$ be the universal family over $\textbf{RPerf}$, then 
$$\Omega_{\textbf{RPerf}}=\ch_2(\mathcal{E})\in HN^0(\textbf{RPerf})(2)\cong HN^0((\textbf{RPerf}\times \calA)/ \calA)(2). $$
The pullback of the shifted symplectic form of $\textbf{Hilb}(D)\times \calA$ to $\textbf{Hilb}(Y,D)$ is 
\begin{align}\label{pb of sym form}
r_{\calD\to \calY}^*\int_{\calD/\calA}\mathrm{vol}_{\calD/\calA}\wedge\left(\ev_D^*\Omega_{\textbf{RPerf}}\right)
&=\int_{\calD/\calA}(\id\times r_{\calD\to \calY})^*\left(\mathrm{vol}_{\calD/\calA}\wedge\left(\ev_D^*\Omega_{\textbf{RPerf}}\right)\right) \\ \nonumber 
&=\int_{\calD/\calA}\mathrm{vol}_{\calD/\calA}\wedge i_{\calD\to \calY}^*\left(\ev_{Y,D}^*\Omega_{\textbf{RPerf}}\right),
\end{align} 
where the second equality uses the commutativity of diagram \eqref{comm diag on lag evmap}, and the upper boxes in diagram \eqref{comm diag on lag evmap2}  
(which ensures the compactly supported lifts also coincide).

To construct a  null-homotopy of \eqref{pb of sym form}, it suffices to 
construct a null-homotopy of the map (in below $d=\dim_\bbC D$):
\begin{equation}\label{nullhom map}\int_{\calD/\calA}\mathrm{vol}_{\calD/\calA}\wedge i_{\calD\to \calY}^*(-): 
NC_{\calZ_{Y,D}}\left((\calY\times_{\calA}\textbf{Hilb}(Y,D))/\calA\right)(2)\to 
NC_{}\left(\textbf{Hilb}(Y,D)/\calA\right)[-d](2). \end{equation}
This is provided by $\tr_{\calY/\calA}$ in diagram \eqref{diag on nullhom} following the argument in \cite[\S 3.2.1, \S 2.2.3]{Cal}. 
\end{proof}

\subsection{Composition with a Lagrangian fibration}\label{sect on lag fib}
%Continue with setting of the previous section. 
The restriction map $r_{\calD\to \calY}$
%$\textbf{Hilb}(Y,D) \textbf{Hilb} (D)\times \calA$ 
has a Lagrangian structure (Theorem \ref{thm on lag}). 
Fix a numerical $K$-theory class $P_D\in K_{c,\leqslant \dim_{\mathbb{C}}D-2}^\mathrm{{num}}(D)$ and consider 
the open substack $\textbf{Hilb}^{P_D}(D)$ of $\textbf{Hilb} (D)$.
We are interested in the case when $\textbf{Hilb}^{P_D}(D)$ has a \textit{Lagrangian fibration} structure in the following sense. 
\begin{definition}$($\cite[Def.~1.9]{Cal2}$)$
Let $\pi:M\to N$ be a map of derived Artin stacks locally of finite presentations. The structure of an $n$-\textit{shifted Lagrangian fibration} on $\pi$ is an $n$-shifted symplectic structure $\Omega_{M}$ on $M$
together with a null-homotopy of its image under the natural map $\mathcal{A}^{2,\emph{cl}}(M,n)\to \mathcal{A}^{2,\emph{cl}}(M/N,n)$, such that the 
induced sequence 
\begin{equation}\label{fib seq of lag fib}\bbT_{M/N}\to \bbT_M  \stackrel{\Omega_{M}}{\cong} \bbL_M[n]  \to \bbL_{M/N}[n] \end{equation}
is a fiber sequence. 
\end{definition}
\begin{remark}
%It is straightforward to define shifted Lagrangian fibration structures 
For a map $\pi$ relative to a base $B$, one defines it similarly replacing $\bbT_M$ by $\bbT_{M/B}$.
\end{remark}
Here is an example of Lagrangian fibration (in the $n=-1$ case). 
\begin{example}\label{ex of lag fib}
Let $W$ be a smooth complex algebraic variety and $\phi$ is a regular function. We define its \textit{derived critical locus} $\textbf{Crit}(\phi)$ by the following homotopy pullback diagram
\begin{equation*}\begin{xymatrix}{
\bCrit^{}(\phi)  \ar[r]^{ } \ar[d]^{ } \ar@{}[dr]|{\Box} &W \ar[d]^{d\phi}  \\
W \ar[r]^{0  \,\,} & \bfT^*W.
}\end{xymatrix}\end{equation*}
$\textbf{Crit}(\phi)$ has a canonical $(-1)$-shifted symplectic structure by Lagrangian intersection theorem of \cite[Thm.~0.5]{PTVV}. 
The natural embedding $$\textbf{Crit}(\phi)\to W$$ is a $(-1)$-shifted Lagrangian fibration \cite[Rmk.~3.12]{Gra}. 
%Combining \eqref{equ on folklore} and Proposition \ref{prop on iso condition2}, we obtain a $(-2)$-shifted symplectic structure on the composition map 
%$$\textbf{Hilb} (Y,D)\to (W \times \calA). $$
\end{example}

%such that $D\in |\omega_{Y}^{-1}|$ is an anticanonical divisor of $Y$. 
%The reason we are interested in Lagrangian fibrations is that by composing a Lagrangian with it, we obtain a (relative) shifted symplectic structure. 
\begin{lemma}\label{lem on comp lag fib}
Let $A$ be a derived Artin stack and $L \stackrel{f}{\to} M \stackrel{\pi}{\to} B$ be maps of derived Artin stacks (over $A$) locally of finite presentations.
Assume $f$ has an $n$-shifted Lagrangian structure relative to $A$ and $\pi$ has an $n$-shifted Lagrangian fibration structure relative to $A$. 
Then $\pi\circ f: L\to B$ has an induced $(n-1)$-shifted symplectic structure. 
\end{lemma}
\begin{proof}
This is a relative version of \cite[Prop.~1.10]{Saf}. We have a homotopy pullback diagram 
\begin{equation*}\begin{xymatrix}{
L  \ar[r]^{\id\times (\pi\circ f)\, \quad } \ar[d]_{ }  & L\times_A B \ar[r]^{ }  \ar[d]^{}  &  L\times B   \ar[d]^{}    \\
M \ar[r]^{\id\times \pi \quad \,\,} & M\times_A B   \ar[r]^{ }  & M\times B.    
}\end{xymatrix}\end{equation*}
Note that $L\times_A B \to M\times_A B $ and $M\to M\times_A B$ have $n$-shifted Lagrangian structures relative to $B$. 
As Lagrangian intersection, we know $L\to B$ has an induced $(n-1)$-shifted symplectic structure. 
\end{proof}
As a consequence, we have: 
\begin{prop}\label{prop on iso condition2}
Fix $P\in K_{c, \leq \dim_{\mathbb{C}}Y-2}^{\mathrm{num}} (Y)$ and $P_D:=P|_D\in K_c^{\mathrm{num}} (D)$. 
Assume the shifted symplectic stack $\textbf{\emph{Hilb}}^{P_D}(D)$ admits a Lagrangian fibration structure $\textbf{\emph{Hilb}}^{P_D} (D)\to W$, then its composition 
with $r_{\calD\to \calY}$ \eqref{rest map2}: 
$$\textbf{\emph{Hilb}}^P(Y,D)\to (W \times \calA^P)$$ 
has an induced $(1-\dim_{\mathbb{C}}D)$-shifted symplectic structure.  
\end{prop}
\begin{remark}\label{rmk on quiver with poten}
There are many quasi-projective Calabi-Yau 3-folds $D$ satisfying that the classical truncation $\Hilb^{P_D} (D):=t_0(\textbf{Hilb}^{P_D} (D))$ is a global critical locus and the  
ambient space and regular function have canonical choice given by the so-called \textit{quivers with potentials} \cite{DWZ, Gin, KS2}.
%It is interesting to know whether the description can be enhanced to the derived level. 
\end{remark}
\begin{example}\label{fl ex}
When $D=\mathbb{C}^3$, it is believed that one has an equivalence of $(-1)$-shifted symplectic derived scheme: 
\begin{equation}\label{equ on folklore}\textbf{Hilb}^n(D)\cong \textbf{Crit}(\phi: W\to \mathbb{C}). \end{equation}
%Here $W$ is a smooth algebraic variety and $\phi$ is a regular function, 
Here $W$ is the non-commutative Hilbert scheme:
\begin{equation}\label{W space}W=\left(\Hom(\mathbb{C}^n,\mathbb{C}^n)^{\times 3}\times \mathbb{C}^n\right)/\!\!/\GL(n), \end{equation}
and $\phi$ is the potential function: 
$$\phi(b_1,b_2,b_3,v)=\tr (b_1[b_2,b_3]). $$  
By Example \ref{ex of lag fib}, $\textbf{Hilb}^n(D)\to W$ then has a Lagrangian fibration structure. 
\end{example}

\iffalse
\subsection{Image of symplectic form in periodic cyclic homology II}
For the application of defining relative $\DT_4$ invariants, we consider the case when $\dim_{\mathbb{C}}D=3$ in above.
Below is an analogous result of Proposition \ref{iso cond 1}
concerning the image of the $(-2)$-shifted symplectic structure (constructed in Proposition \ref{prop on iso condition2})
in the periodic cyclic homology. 
\fi

\section{Orientations on moduli stacks}
%In \S~\ref{} we constructed shifted symplectic (resp.~Lagrangian) structures on derived Hilbert stacks on the stacks of expanded degenerations $\calX/\calC$ (resp.~stacks of expanded pairs $\calY/\calA$). 
In this section, we collect notions of orientations associated with shifted symplectic stacks and shifted Lagrangians. We prove that canonical orientation exists on local CY 4-folds 
\eqref{intro X to A1}. 
%(see also~\cite[\S 5]{KS}, \cite[\S 2.4]{BJ}, \cite[\S 1.4]{CGJ}, \cite{CL2}, \cite[\S 6]{CL3}). 

\subsection{Orientation on $(\Hilb(X/\bbA^1)\to \calC)$}
Recall that a Calabi-Yau simple degeneration (Definitions \ref{defi of dege}, \ref{cy simple dege}) of quasi-projective $n$-folds is a family  $$\pi : X \to \bbA^1, $$ 
where  generic fibers of $\pi$ are smooth quasi-projective Calabi-Yau $n$-folds. Associated to this data, using \S \ref{sect on exp dege}, 
we have the stack $\calC$ of expanded degeneration with universal family
$$\pi: \calX \to \calC. $$
By Theorem \ref{thm on sft symp str}, $\textbf{Hilb}(X/\bbA^1)/\calC$ has a $(2-n)$-shifted symplectic structure, which induces
 an isomorphism 
\begin{equation}\label{der serre dual}\bbT_{\textbf{Hilb}(X/\bbA^1)/\calC}\cong \bbL_{\textbf{Hilb}(X/\bbA^1)/\calC}[2-n]. \end{equation}
When $n$ is \textit{even}, taking determinant induces an isomorphism 
\begin{equation*}\det\left(\bbT_{\textbf{Hilb}(X/\bbA^1)/\calC}\right)^{\otimes 2}\cong \oO_{\textbf{Hilb}(X/\bbA^1)}.  \end{equation*}
Here $\det\left(\bbT_{\textbf{Hilb}(X/\bbA^1)/\calC}\right)$ is the \textit{determinant line bundle}.
Restricting to the classical truncation $\Hilb(X/\bbA^1)=t_0(\textbf{Hilb}(X/\bbA^1))$, we obtain 
\begin{equation}\label{serre iso even}\det\left(\bbT_{\textbf{Hilb}(X/\bbA^1)/\calC}|_{\Hilb(X/\bbA^1)}\right)^{\otimes 2}\cong \oO_{\Hilb(X/\bbA^1)}. \end{equation}
By the construction of shifted symplectic structure in Theorem~\ref{thm on sft symp str}, this isomorphism agress with the one provided by 
 the Grothendieck-Serre duality. 

As in \cite[\S 2.4]{BJ}, \cite[\S 1.4]{CGJ}, \cite{CL1, CL2}, we define: 
\begin{definition}\label{ori on even cy}
Assume $n$ is even. An \textit{orientation} on $\Hilb(X/\bbA^1)/\calC$ is a square root of the isomorphism \eqref{serre iso even}. 
%That is, an isomorphism $\det\left(\bbT_{\textbf{Hilb}(X/\bbA^1)/\calC}|_{\Hilb(X/\bbA^1)}\right)\cong \calO$, the square of which gives \eqref{serre iso even}.
\end{definition}

\begin{remark}\label{rmk on local cy4 ori}
In this paper, we focus on $n=4$ case, and a generic fiber of  
$$\Hilb(X/\bbA^1)\to \calC$$
is the Hilbert scheme on a smooth quasi-projective Calabi-Yau 4-fold.  The restriction of the isomorphism \eqref{serre iso even} to such fiber
has a square root under some assumptions \cite{CGJ, CGJ2, JU}\footnote{See \cite{CL2} for a discussion in other dimensions.}. 
The argument in~\textit{loc}.~\textit{cit}.~is differential geometric in natural and uses gauge theoretic moduli spaces. Although it might be possible to adapt the argument to a family, 
here we provide a canonical orientation useful for calculations on \eqref{intro X to A1}. This is motivated by \cite{C}, \cite[\S 4.2]{CKM2}. 
%(involving choice of square roots) in our interested examples. 
%there is no known orientability result on the `global' relative moduli stack. 
%Even the argument in~\textit{loc}.~\textit{cit}.~could be extended to a family version, it is usually difficult to be used in ,  
%in the case when $X/\bbA^1$ is the total space of the relative canonical bundle of a simple degeneration of 3-folds.  
%\begin{setting}\label{setting of rel can dege}Let $U\to \bbA^1$ be a simple degeneration of 3-folds and $$X=\Tot(\omega_{U/\bbA^1})\to \bbA^1$$ be the associated Calabi-Yau %simple degeneration as in Proposition \ref{prop on rel can dege}. \end{setting}
\end{remark}
\begin{theorem}\label{ori of cy4 family}
Let $U\to \bbA^1$ be a simple degeneration of 3-folds, and $$X=\Tot(\omega_{U/\bbA^1})\to \bbA^1$$ the associated simple degeneration of 4-folds. 
Then $\Hilb(X/\bbA^1)/\calC$ has an orientation. Similarly, $\Hilb^{P_t}(X/\bbA^1)/\calC^{P_t}$ has an orientation.
\end{theorem}
\begin{proof}
Let  $M_{\mathrm{cpt}}(\calX/\calC)$ be the moduli stack of  sheaves on $\calX$ with support proper over $ \calC$.
Consider the forgetful map 
\begin{equation}\label{forget map}f: \Hilb(X/\bbA^1)\to M_{\mathrm{cpt}}(\calX/\calC), \quad (\oO\to F) \mapsto F.\end{equation}
Let $\bar{f}:=f\times_{\id_{\calC}} \id_{\calX}$, and let 
$$\pi_{M_X}:M_{\mathrm{cpt}}(\calX/\calC)\times_{\calC} \calX\to M_{\mathrm{cpt}}(\calX/\calC), \quad \pi_{H_X}: \Hilb(X/\bbA^1)\times_{\calC} \calX\to \Hilb(X/\bbA^1)$$ 
be the natural projections.

Let $\bbF$ be the universal sheaf on $M_{\mathrm{cpt}}(\calX/\calC)\times_{\calC} \calX$ and $\bbI=\left(\oO_{\Hilb(X/\bbA^1)\times_{\calC} \calX}\to \bar{f}^*\bbF\right)$ 
be the universal pair. Then we have the following natural isomorphism
\begin{align}\label{identify two det}\det\left(\bbT_{\textbf{Hilb}(X/\bbA^1)/\calC}|_{\Hilb(X/\bbA^1)}\right)&\cong \det\left(\pi_{H_X*}\dR\calH om(\bbI,\bbI)_0[1]\right) \\  \nonumber 
&\cong \det\left(\pi_{H_X*}\dR\calH om(\bar{f}^*\bbF,\bar{f}^*\bbF)[1]\right) \\ \nonumber
&\cong f^*\det\left(\pi_{M_X*}\dR\calH om(\bbF,\bbF)[1]\right),   \nonumber 
\end{align}
where the second isomorphism uses the distinguished triangles (e.g.~\cite[\S 3.2]{CMT2}): 
$$\dR\calH om(\bbI,\bar{f}^*\bbF) \to \dR\calH om(\bbI,\bbI)_0[1]\to \dR\calH om(\bar{f}^*\bbF, \oO)[2],  $$
$$ \dR\calH om(\oO,\bar{f}^*\bbF) \to \dR\calH om(\bbI, \bar{f}^*\bbF)\to \dR\calH om(\bar{f}^*\bbF,\bar{f}^*\bbF)[1],  $$
and Grothendieck-Serre duality 
$$\pi_{H_X*}\dR\calH om(\oO,\bar{f}^*\bbF)\cong \left(\pi_{H_X*}\dR\calH om(\bar{f}^*\bbF, \oO)\right)^\vee[-4]. $$
Similar to \eqref{serre iso even}, we also have 
\begin{equation}\label{serre pairing on tor}\det\left(\pi_{M_X*}\dR\calH om(\bbF,\bbF)[1]\right)^{\otimes 2}\cong \oO, \end{equation}
given by Grothendieck-Serre duality on $\pi_{M_X*}\dR\calH om(\bbF,\bbF)[1]$. 
Notice that a choice of square root of \eqref{serre pairing on tor} gives a choice of square root of 
\eqref{serre iso even} under the pullback map \eqref{identify two det} (e.g.~\cite[Thm.~6.3]{CMT2}). 
%In what follows, we find a square root of \eqref{serre pairing on tor}. 

We provide two (equivalent) constructions of a square root of \eqref{serre pairing on tor}. 
Similar to Theorem~\ref{thm on sft symp str} (see also \cite{BD}), 
by adapting \cite[Thm.~2.5]{PTVV} to (relatively) compactly supported sheaves, 
the natural derived enhancement $\textbf{M}_{\mathrm{cpt}}(\calX/\calC)$ of $M_{\mathrm{cpt}}(\calX/\calC)$ has a shifted symplectic structure relative to $\calC$.
Since $\calX=\Tot(\omega_{\calU/\calC})$,
 $\textbf{M}_{\mathrm{cpt}}(\calX/\calC)$ is equivalent to the $(-2)$-shifted symplectic cotangent bundle (relative to $\calC$) of the derived moduli stack $\textbf{M}_{\mathrm{cpt}}(\calU/\calC)$ of compactly supported sheaves on $\calU/\calC$ \cite[Thm.~6.17]{BCS}. 
Let $\textbf{p}$ be the projection 
$$\textbf{p}:\textbf{M}_{\mathrm{cpt}}(\calX/\calC)\to \textbf{M}_{\mathrm{cpt}}(\calU/\calC). $$
Then there is a fiber sequence 
\begin{equation}\label{equ on -2shift cota}\textbf{p}^*\bbL_{\textbf{M}_{\mathrm{cpt}}(\calU/\calC)/\calC}[-2]\cong\bbT_{\textbf{p}}\to \bbT_{\textbf{M}_{\mathrm{cpt}}(\calX/\calC)/\calC}\to \textbf{p}^*\bbT_{\textbf{M}_{\mathrm{cpt}}(\calU/\calC)/\calC}. \end{equation}
Restricting to the classical truncation and taking determinant, we obtain  
\begin{equation}\label{iso of det abs-1}\det\left(\bbT_{\textbf{M}_{\mathrm{cpt}}(\calX/\calC)/\calC}|_{M_{\mathrm{cpt}}(\calX/\calC)}\right)\cong \oO, \end{equation}
by Grothendieck-Serre duality. 
%\guf{By construction, the square of this isomorphism is \eqref{serre pairing on tor}. }

%\gufang{What do you think about just keeping the second construction? Or we could make the first construction into a remark?}

Alternatively, let $\pi: \calX=\Tot(\omega_{\calU/\calC}) \to \calU$ be the projection, consider the following diagram
%and $p:=t_0(\textbf{p})$ be the classical truncation of $\textbf{p}$ which pushforward sheaves via $\pi_*$.  Consider the following diagram  
%\begin{equation*}\begin{xymatrix}{M_{\mathrm{cpt}}(\calX/\calC)\times \calX  \ar[r]^{\,\,\bar{\pi}:=\id\times \pi  } \ar[dr]_{\pi_{M_{X}}}   & M_{\mathrm{cpt}}(\calX/\calC)\times \calU  \ar[r]^{\,\,
%\bar{p}:=p\times \id }  \ar[d]^{\pi_{M_{X}}} \ar@{}[dr]|{\Box} & M_{\mathrm{cpt}}(\calU/\calC)\times \calU  \ar[d]^{\pi_{M_{U}}} \\ & M_{\mathrm{cpt}}(\calX/\calC)  \ar[r]^{p} &  M_{\mathrm{cpt}}
%(\calU/\calC).}\end{xymatrix}\end{equation*}
%Let $\bbF_{\calU}$ be the universal sheaf on $M_{\mathrm{cpt}}(\calU/\calC)\times \calU$, then we have 
\begin{equation*}\begin{xymatrix}{
M_{\mathrm{cpt}}(\calX/\calC)\times_{\calC} \calX  \ar[r]^{\,\,\bar{\pi}:=\id\times \pi  } \ar[dr]_{\pi_{M_{X}}}   & M_{\mathrm{cpt}}(\calX/\calC)\times_{\calC} \calU     \ar[d]^{\pi_{M_{U}}}   \\
 & M_{\mathrm{cpt}}(\calX/\calC).}\end{xymatrix}\end{equation*}
%$$\bar{p}_*\bar{\pi}_*\bbF\cong \bbF_{\calU}. $$
Spectral construction gives a distinguished triangle 
\begin{equation}\label{equ on spectral cons}\pi_{M_U*}\dR\calH om(\bar{\pi}_*\bbF,\bar{\pi}_*\bbF\boxtimes \omega_{\calU/\calC})[-1]\to \pi_{M_X*}\dR\calH om(\bbF,\bbF)\to \pi_{M_U*}\dR\calH om(\bar{\pi}_*\bbF,\bar{\pi}_*\bbF). \end{equation}
Taking determinant gives an isomorphism
\begin{align}\label{iso of det abs}\det\left(\pi_{M_X*}\dR\calH om(\bbF,\bbF)[1]\right)& \cong \frac{\det\left(\pi_{M_U*}\dR\calH om(\bar{\pi}_*\bbF,\bar{\pi}_*\bbF\boxtimes \omega_{\calU/\calC})\right)}{\det\left(\pi_{M_U*}\dR\calH om(\bar{\pi}_*\bbF,\bar{\pi}_*\bbF) \right)}\cong \oO  
\end{align}
where the second isomorphism is given by  Grothendieck-Serre duality. It is easy to see that \eqref{iso of det abs} and \eqref{iso of det abs-1} are the same. 
Up to a constant determined in \eqref{cho of sign} below, the square of this isomorphism  gives  \eqref{serre pairing on tor}.  
The case of $\Hilb^{P_t}(X/\bbA^1)/\calC^{P_t}$ is similar, by replacing $\calX^{}/\calC$ in above by  $\calX^{P_t}/\calC^{P_t}$.
\end{proof}
\begin{remark}\label{rmk on ori}
The orientation on $\Hilb(X/\bbA^1)/\calC$ is constructed using the pullback of an orientation on the moduli stack $M_{\mathrm{cpt}}(\calX/\calC)$ of compactly supported sheaves on $\calX/\calC$ via the forgetful map \eqref{forget map}. 
By spectral construction \eqref{equ on spectral cons}, the symmetric obstruction complex on $M_{\mathrm{cpt}}(\calX/\calC)$ is of form 
$$\bbE=(\mathbb{V}\xrightarrow{\varphi} \mathbb{V}^\vee), $$ 
where $\mathbb{V}$ is a perfect complex and $\varphi$ is self-dual under the symmetric pairing on $\bbE$ (defined by Grothendieck-Serre duality). 
Therefore there is a canonical isomorphism 
$\det(\bbE)\cong \calO$ and \textit{orientations} of $\bbE$ are given by 
\begin{equation}\label{cho of sign}\calO\xrightarrow{\pm (-\sqrt{-1})^{\rk(\mathbb{V})}} \calO \end{equation}
on each connected component of $\calX$ (as~\cite[Eqns.~(59),~(63)]{OT}). We choose the plus sign in above as the \textit{canonical choice} of orientation. 
Similarly $\Hilb^{P_t}(X/\bbA^1)/\calC^{P_t}$ has a canonical choice of orientation. 
\end{remark}

\subsection{Orientation on $(\Hilb(Y,D)\to\Hilb(\calD))/\calA$}
Given a smooth log Calabi-Yau pair $(Y,D)$ (Definition~\ref{defi of log cy pair}) so that $D$ is a Calabi-Yau variety, we have a stack $\calA$ of expanded pairs
and universal expanded pair $\calY$ and universal divisor $\calD$ (Definition~\ref{def of exp pair}): 
$$\pi : \calY \to \calA, \quad \calD\cong D \times \calA. $$ 
By Theorem \ref{thm on lag}, the restriction map  
$$r_{\calD\to \calY}: \textbf{Hilb}(Y,D)\to \textbf{Hilb} (\calD)=\textbf{Hilb} (D)\times \calA$$
has a Lagrangian structure relative to $\calA$.  
Hence there is an isomorphism (\cite[Def.~2.8]{PTVV}):
\begin{equation}\label{lag iso}\bbT_{r_{\calD\to \calY}}\cong \bbL_{\textbf{Hilb}(Y,D)/\calA}[1-\dim_{\mathbb{C}}D]. \end{equation}
Combining with the fiber sequence 
$$\bbT_{r_{\calD\to \calY}} \to \bbT_{\textbf{Hilb}(Y,D)/\calA}\to r_{\calD\to \calY}^*\bbT_{(\textbf{Hilb}(D)\times \calA)/\calA}, $$ 
we obtain an isomorphism of determinant line bundles:
\begin{equation}\label{rel ori iso der}\left(\det\left(\bbT_{\textbf{Hilb}(Y,D)/\calA}\right)\right)^{\otimes 2}\cong r_{\calD\to \calY}^*\det\left(\bbT_{(\textbf{Hilb}(D)\times \calA)/\calA} \right), 
\,\, \mathrm{if}\,\dim_{\mathbb{C}}D\,\, \mathrm{is}\,\,\mathrm{odd}, \end{equation} 
and $r_{\calD\to \calY}^*\det\left(\bbT_{(\textbf{Hilb}(D)\times \calA)/\calA}\right)\cong \oO$ if $\dim_{\mathbb{C}}D$ is even, which reduces to Definition \ref{ori on even cy}. 

Restricting \eqref{rel ori iso der} to the classical truncation, we obtain
\begin{equation}\label{rel ori iso}\left(\det\left(\bbT_{\textbf{Hilb}(Y,D)/\calA}|_{\Hilb(Y,D)}\right)\right)^{\otimes 2}\cong 
r_{\calD\to \calY}^*\det\left(\bbT_{(\textbf{Hilb}(D)\times \calA)/\calA}|_{\Hilb(D)\times \calA} \right), \end{equation}
where $r_{\calD\to \calY}$ denotes its classical truncation by an abuse of notation. 

As in \cite[Def.~4.4]{CL2},  \cite[Def.~6.1]{CL3}, we define: 
\begin{definition}\label{def of rel or}
Assume $\dim_{\mathbb{C}}D$ is odd. A \textit{relative orientation} on the map $$r_{\calD\to \calY}: \Hilb(Y,D)\to \Hilb(D)\times \calA$$
is the following data:
a line bundle $K^{-1/2}_{\Hilb(D)}$, an isomorphism 
\[\left(K^{-1/2}_{\Hilb(D)}\right)^{\otimes 2}\cong \det\left(\bbT_{(\textbf{Hilb}(D)\times \calA)/\calA}|_{\Hilb(D)\times \calA} \right),  \] and an isomorphism 
$$\det\left(\bbT_{\textbf{Hilb}(Y,D)/\calA}|_{\Hilb(Y,D)}\right)\cong r_{\calD\to \calY}^*K^{-1/2}_{\Hilb(D)}, $$
whose square is the isomorphism \eqref{rel ori iso}. 
\end{definition}
As in Theorem \ref{ori of cy4 family}, we have the following: 
\begin{theorem}\label{relative ori of cy4 family}
Let $(U,S)$ be a smooth pair (Definition \ref{defi of smooth pair}) with $\dim_{\mathbb{C}} U=3$ and 
$$Y=\Tot(\omega_{U}(S)), \quad D=\Tot(\omega_S)$$ 
be the (associated) log Calabi-Yau pair (Definition \ref{defi of log cy pair}). Then 
$$r_{\calD\to \calY}: \Hilb(Y,D)\to \Hilb(D)\times \calA$$ has a relative orientation.  
Similarly, $\Hilb^P(Y,D)\to \Hilb^{P|_D}(D)\times \calA^P$ has a relative orientation.
\end{theorem}
This is similar to the proof of Theorem \ref{ori of cy4 family}.
\begin{proof} 
Let  $M_{\mathrm{cpt}}(\calY/\calA)$ be the moduli stack of compactly supported sheaves on $\calY/\calA$ which are normal to $\calD$ in the sense of Definition \ref{def of normal}. Let $M_{\mathrm{cpt}}(D)$ be the moduli stack of compactly supported sheaves on $D$. 
 We have a commutative diagram 
\begin{equation}\label{diag on hilb to m}\begin{xymatrix}{
\Hilb(Y,D)  \ar[r]^{f_Y} \ar[d]_{r_{\calD\to \calY}}  &M_{\mathrm{cpt}}(\calY/\calA)  \ar[d]^{r}  \\
\Hilb(D)\times \calA \ar[r]^{f_D} & M_{\mathrm{cpt}}(D)\times \calA,
}\end{xymatrix}\end{equation}
where $r$ is the restriction map, and  both  $f_Y$ and $f_D$ are forgetful maps sending $(\oO\to F)$ to $F$.
Let $\bar{f}_Y:=f_Y\times_{\id_{\calA}} \id_{\calY}$, and  
$$\pi_{M_Y}:M_{\mathrm{cpt}}(\calY/\calA)\times_{\calA} \calY\to M_{\mathrm{cpt}}(\calY/\calA), \quad \pi_{H_Y}: \Hilb(Y,D)\times_{\calA} \calY\to \Hilb(Y,D)$$ 
be the  projections. 
The embedding $i:\calD\hookrightarrow \calY$ induces
$\bar{i}: M_{\mathrm{cpt}}(\calY/\calA)\times_{\calA} \calD\hookrightarrow  M_{\mathrm{cpt}}(\calY/\calA)\times_{\calA} \calY$.
Let $\bbF$ be the universal sheaf on $M_{\mathrm{cpt}}(\calY/\calA)\times_{\calA} \calY$ and $\bbI=\left(\oO_{\Hilb(Y,D)\times_{\calA} \calY}\to \bar{f}_Y^*\bbF\right)$ 
be the universal pair. 
As in \eqref{identify two det}, we have 
\begin{align}\label{identify two det2}&\quad \, \det\left(\bbT_{\textbf{Hilb}(Y,D)/\calA}|_{\Hilb(Y,D)}\right)\cong \det\left(\pi_{H_Y*}\dR\calH om(\bbI,\bbI)_0[1]\right) \\  \nonumber 
&\cong f_Y^*\left(\det\left(\pi_{M_Y*}\dR\calH om(\bbF,\bbF)[1]\right)\otimes \det\left(\pi_{M_Y*} \bar{i}_*\bar{i}^*\bbF\right) \right),
\end{align}
%where the second isomorphism uses the distinguished triangles (e.g.~\cite[\S 3.2]{CMT2}): 
%$$\dR\calH om(\bbI,\bar{f}_Y^*\bbF) \to \dR\calH om(\bbI,\bbI)_0[1]\to \dR\calH om(\bar{f}_Y^*\bbF, \oO)[2],  $$
%$$ \dR\calH om(\oO,\bar{f}_Y^*\bbF) \to \dR\calH om(\bbI, \bar{f}_Y^*\bbF)\to \dR\calH om(\bar{f}_Y^*\bbF,\bar{f}_Y^*\bbF)[1],  $$
%and Grothendieck-Serre duality 
%$$\pi_{H_Y*}\dR\calH om(\bar{f}_Y^*\bbF, \oO)\cong \left(\pi_{H_Y*}\dR\calH om(\oO,\bar{f}_Y^*\bbF\boxtimes \omega_{\calY/\calA})\right)^\vee[-4]. $$
%\begin{equation}\label{equ useful proof1}\frac{\det\left(\pi_{M_Y*}\dR\calH om(\oO,\bbF)\right)}{\det\left(\pi_{M_Y*}\dR\calH om(\oO,\bbF\boxtimes \omega_{\calY/\calA})\right)}\cong \det\left(\pi_{M_Y*} i_*i^*\bbF\right), \end{equation}
Consider the natural projection and the inclusion of a divisor:
$$\pi: \calY=\Tot(\omega_{\calU/\calA}(\calS)) \to \calU, \quad j: \calS:=S\times \calA\hookrightarrow \calU. $$ 
We have the commutative diagram
%Therefore \eqref{identify two det2} becomes 
%\begin{equation}\label{identify two det3}\det\left(\bbT_{\textbf{Hilb}(Y,D)/\calA}|_{\Hilb(Y,D)}\right)\cong f_Y^*\left(\det\left(\pi_{M_Y*}\dR\calH om(\bbF,\bbF)[1]\right)\otimes \det\left(\pi_{M_Y*} i_*i^*\bbF\right) \right). \end{equation}
\begin{equation}\label{diag on Myus}{\tiny
\xymatrix{
M_{\mathrm{cpt}}(\calY/\calA)\times_{\calA} \calY  \ar[r]^{\,\,\bar{\pi}:=\id\times \pi  } \ar[dr]_{\pi_{M_{Y}}}   & M_{\mathrm{cpt}}(\calY/\calA)\times_{\calA} \calU \ar@{}[dr]|{\Box} \ar[r]^{\pi_{\dagger}}    \ar[d]^{\pi_{M_{U}}} & M_{\mathrm{cpt}}(\calU/\calA)\times_{\calA} \calU    \ar[d]^{\pi_{M_{U}}}& M_{\mathrm{cpt}}(\calU/\calA)\times_{\calA} \calS   \ar@{}[dr]|{\Box}  \ar[l]_{\bar{j}:=\id\times j}  \ar[r]^{j_{\dagger}}  \ar[d]^{\pi_{M_{S}}}& M_{\mathrm{cpt}}(\calS/\calA)\times_{\calA} \calS    \ar[d]^{\pi_{M_{S}}}   \\
 & M_{\mathrm{cpt}}(\calY/\calA)  \ar[r]^{\pi_{\dagger}}    & M_{\mathrm{cpt}}(\calU/\calA) \ar@{=}[r] & M_{\mathrm{cpt}}(\calU/\calA) \ar[r]^{j_{\dagger}} & M_{\mathrm{cpt}}(\calS/\calA),
 }}\end{equation}
 where $\pi_{\dagger}$ denotes the pushforward map from $\calY$ to $\calU$ by $\pi$ and $j_{\dagger}$ denotes the pullback map from $\calU$ to $\calS$ by $j$. Note that  
 we have 
 \begin{equation}\label{equ useful proof1.2}\bar{\pi}_*\bbF=\pi_{\dagger}^*\bbF_U, \quad \bar{j}^*\bbF_U=j_{\dagger}^*\bbF_S,  \end{equation}
 where $\bbF_U\to M_{\mathrm{cpt}}(\calU/\calA)\times_{\calA} \calU$ is the universal object of the moduli stack $M_{\mathrm{cpt}}(\calU/\calA)$ of compactly supported sheaves on $\calU/\calA$ (which are normal to $\calS$) and $\bbF_S\to M_{\mathrm{cpt}}(\calS/\calA)\times_{\calA} \calS$ is the universal object of the moduli stack $M_{\mathrm{cpt}}(\calS/\calA)=M_{\mathrm{cpt}}(S)\times \calA$ of compactly supported sheaves on $(S\times\calA)/\calA$. 
 
Similar to \eqref{iso of det abs}, by using Serre duality, diagrams \eqref{diag on Myus}, \eqref{equ useful proof1.2}, we obtain
\begin{align}\label{iso of det abs2}&\quad \,\,\det\left(\pi_{M_Y*}\dR\calH om(\bbF,\bbF)[1]\right)
\cong \pi_{\dagger}^*j_{\dagger}^*\left(\det\left(\pi_{M_S*}\dR\calH om(\bbF_S,\bbF_S)  \right)\right)^{-1}.
\end{align} 
Combining \eqref{identify two det2} and \eqref{iso of det abs2}, we obtain 
\begin{equation}\label{identify two det3}\det\left(\bbT_{\textbf{Hilb}(Y,D)/\calA}|_{\Hilb(Y,D)}\right)\cong \frac{f_Y^*\det\left(\pi_{M_Y*} \bar{i}_*\bar{i}^*\bbF\right)}{f_Y^*\pi_{\dagger}^*j_{\dagger}^*\det\left(\pi_{M_S*}\dR\calH om(\bbF_S,\bbF_S)\right)}. \end{equation}
Consider the projection and the embedding: $$p: \calD=\Tot(\omega_{\calS/\calA}) \to \calS, \quad i:\calD\hookrightarrow \calY. $$ 
We have a commutative diagram 
\begin{equation}\label{diag on Myus2}{\tiny
\xymatrix{
 M_{\mathrm{cpt}}(\calY/\calA)\times_{\calA} \calY    \ar[dr]_{\pi_{M_{Y}}} &   M_{\mathrm{cpt}}(\calY/\calA)\times_{\calA} \calD   \ar@{}[dr]|{\Box}   \ar[l]_{\bar{i}:=\id\times i} \ar[r]^{\bar{r}:=r\times \id }    \ar[d]^{\pi_{M_{D}}}& M_{\mathrm{cpt}}(\calD/\calA)\times_{\calA} \calD \ar[r]^{\,\,\bar{p}:=\id\times p}    \ar[d]^{\pi_{M_{D}}}& M_{\mathrm{cpt}}(\calD/\calA)\times_{\calA} \calS   
 \ar@{}[dr]|{\Box}  \ar[r]^{p_{\dagger}}  \ar[d]^{\pi_{M_{S}}}& M_{\mathrm{cpt}}(\calS/\calA)\times_{\calA} \calS    \ar[d]^{\pi_{M_{S}}}   \\
  &   M_{\mathrm{cpt}}(\calY/\calA)  \ar[r]^{r}    &
 M_{\mathrm{cpt}}(\calD/\calA) \ar@{=}[r] & M_{\mathrm{cpt}}(\calD/\calA) \ar[r]^{p_{\dagger}} & M_{\mathrm{cpt}}(\calS/\calA),
 }}\end{equation}
where $p_{\dagger}$ is the pushforward of sheaves via $p$ and $r$ is the restriction map via $i$. 
As in \eqref{equ useful proof1.2}, we have 
\begin{equation}\label{equ useful proof1.3}\bar{i}^*\bbF=\bar{r}_{}^*\bbF_D  , \quad \bar{p}_*\bbF_D=p_{\dagger}^*\bbF_S,  \end{equation}
where $\bbF_D$ (resp.~$\bbF_S$) is the universal sheaf on $M_{\mathrm{cpt}}(\calD/\calA)\times_{\calA}\calD$ (resp.~on $M_{\mathrm{cpt}}(\calS/\calA)\times_{\calA}\calS$).

Similar to \eqref{identify two det2}, \eqref{iso of det abs2}, with $f_D$ defined in diagram \eqref{diag on hilb to m},  we have 
\begin{align}\label{equ on Thild}&\quad \,\,  \det\left(\bbT_{(\textbf{Hilb}(D)\times \calA)/\calA}|_{\Hilb(D)\times \calA} \right) \\ \nonumber 
&\cong f_D^*\left(\det\left(\pi_{M_D*}\dR\calH om(\bbF_D,\bbF_D)[1]\right)\otimes \left(\det\left(\pi_{M_D*}\bbF_D)\right)\right)^{  2}  \right) \\ \nonumber
%&\cong f_D^*\left(\det\left(\pi_{M_S*}\bar{p}_*\dR\calH om(\bbF_D,\bbF_D)[1]\right)\otimes \left(\det\left(\pi_{M_D*}\bbF_D)\right)\right)^{  2}  \right) \\ \nonumber
%&\cong f_D^*\left(\frac{\det\left(\pi_{M_S*}\dR\calH om(\bar{p}_*\bbF_D,\bar{p}_*\bbF_D\boxtimes \omega_{\calD/\calS})\right)}{\det\left(\pi_{M_S*}\dR\calH om(\bar{p}_*\bbF_D,\bar{p}_*\bbF_D)\right)}\otimes \left(\det\left(\pi_{M_D*}\bbF_D)\right)\right)^{  2}  \right)  \\ \nonumber
&\cong f_D^*\left(\det\left(\pi_{M_S*}\dR\calH om(\bar{p}_*\bbF_D,\bar{p}_*\bbF_D)\right)^{-2}\otimes \left(\det\left(\pi_{M_D*}\bbF_D)\right)\right)^{  2}  \right)  \\ \nonumber
&\overset{\eqref{equ useful proof1.3}}{\cong}  f_D^*\left(\det\left(\pi_{M_S*}p_{\dagger}^*\dR\calH om(\bbF_S,\bbF_S)\right)^{-2}\otimes \left(\det\left(\pi_{M_D*}\bbF_D)\right)\right)^{  2}  \right)  \\ \nonumber 
&\overset{\eqref{diag on Myus2}}{\cong} f_D^*\left(\det\left(p_{\dagger}^*\pi_{M_S*}\dR\calH om(\bbF_S,\bbF_S)\right)^{-2}\otimes \left(\det\left(\pi_{M_D*}\bbF_D)\right)\right)^{  2}  \right), 
\end{align}
where the 2nd isomorphism is derived similarly as \eqref{equ on spectral cons}. 
We define
\begin{equation}\label{square root1}K^{-1/2}_{\Hilb(D)}:=\frac{f_D^*\det\left(\pi_{M_D*}\bbF_D\right)}{f_D^*p_{\dagger}^*\det\left(\pi_{M_S*}\dR\calH om(\bbF_S,\bbF_S)\right)}.\end{equation}
Then, \eqref{equ on Thild} provides an isomorphism \[\left(K^{-1/2}_{\Hilb(D)}\right)^{\otimes 2}\cong \det\left(\bbT_{(\textbf{Hilb}(D)\times \calA)/\calA}|_{\Hilb(D)\times \calA} \right).\]
Finally,  we compute $r_{\calD\to \calY}^*K^{-1/2}_{\Hilb(D)}$:
\begin{align*}& \quad \quad \quad r_{\calD\to \calY}^*\,f_D^*\det\left(\pi_{M_D*}\bbF_D\right)\overset{\eqref{diag on hilb to m}}{\cong}f_Y^*\,r^*\det\left(\pi_{M_D*}\bbF_D\right) \\ 
& \overset{\eqref{diag on Myus2}}{\cong} f_Y^*\det\left(\pi_{M_Y*}\bar{i}_*\bar{r}^*\bbF_D\right)  \overset{\eqref{equ useful proof1.3}}{\cong}f_Y^*\det\left(\pi_{M_Y*}\bar{i}_*\bar{i}^*\bbF\right), 
\end{align*}
$$r_{\calD\to \calY}^*\,f_D^*\,p_{\dagger}^*\overset{\eqref{diag on hilb to m}}{=}f_Y^*\,r^*\,p_{\dagger}^*=f_Y^*\, \pi_{\dagger}^*\,j_{\dagger}^*, $$
where the last equality is by base change along the following Cartesian diagram 
\begin{equation*}
\begin{xymatrix}{
\calD=\Tot(\omega_{\calS/\calA}) \ar@{}[dr]|{\Box}  \ar[r]^{i} \ar[d]_{p}  & \calY=\Tot(\omega_{\calU/\calA})  \ar[d]^{\pi}  \\
\calS \ar[r]^{j} & \calU. }\end{xymatrix}   
\end{equation*} 
Comparing with \eqref{identify two det3}, we conclude that 
\begin{equation}\label{final iso}\det\left(\bbT_{\textbf{Hilb}(Y,D)/\calA}|_{\Hilb(Y,D)}\right)\cong r_{\calD\to \calY}^*K^{-1/2}_{\Hilb(D)}.   \end{equation} 
It is straightforward to show that this provides a relative orientation.   

The case of $\Hilb^P(Y,D)\to \Hilb^{P|_D}(D)\times \calA^P$ is similar, by replacing $\calY^{}/\calA$ in above by $\calY^{P}/\calA^{P}$.
%One can show that its square is the isomorphism \eqref{rel ori iso}. 
\end{proof} 
\begin{remark}
We have isomorphisms 
\begin{align}\label{square root1.2}K^{-1/2}_{\Hilb(D)}& \overset{\eqref{square root1}}{:=}\frac{f_D^*\det\left(\pi_{M_D*}\bbF_D\right)}{f_D^*p_{\dagger}^*\det\left(\pi_{M_S*}\dR\calH om(\bbF_S,\bbF_S)\right)} \overset{\eqref{diag on Myus2}}{\cong}  \frac{f_D^*\det\left(\pi_{M_S*}\bar{p}_*\bbF_D\right)}{f_D^*p_{\dagger}^*\det\left(\pi_{M_S*}\dR\calH om(\bbF_S,\bbF_S)\right)}  \\ \nonumber
&\overset{\eqref{equ useful proof1.3}}{\cong}  \frac{f_D^*\det\left(\pi_{M_S*}p_{\dagger}^*\bbF_S\right)}{f_D^*p_{\dagger}^*\det\left(\pi_{M_S*}\dR\calH om(\bbF_S,\bbF_S)\right)}  \overset{\eqref{diag on Myus2}}{\cong} f_D^*p_{\dagger}^*\frac{\det\left(\pi_{M_S*}\bbF_S\right)}{\det\left(\pi_{M_S*}\dR\calH om(\bbF_S,\bbF_S)\right)},
\end{align}
which is written using data purely on $S$. 
\end{remark}

\subsection{Orientation on $(\Hilb(Y,D)\to W\times \calA) $} 

\begin{definition}\label{def on s cp lag fib}
Let $S$ be a quasi-projective surface and $D=\Tot(\omega_S)$, fix $P_D\in K_{c,\leqslant 2}^\mathrm{{num}}(D)$. A Lagrangian fibration 
$$p: \textbf{Hilb}^{P_D}(D)\to W$$
is $S$-\textit{compatible} if the line bundle \eqref{square root1.2} satisfies 
\begin{equation}\label{equ on Scpt} K^{-1/2}_{\Hilb^{P_D}(D)}\cong p^*\det\left(\bbT_W|_{t_0(W)}\right), \end{equation}
where $p$ also denotes its classical truncation by an abuse of notation. 
\end{definition}
Recall Proposition \ref{prop on iso condition2} that the composition map 
$$p\circ r_{\calD\to \calY}: \textbf{Hilb}^{P}(Y,D)\to W\times \calA^{P}$$
has a (relative) $(-2)$-shifted symplectic structure. By an abuse of notation, we denote its classical truncation by 
$$p\circ r_{\calD\to \calY}: \Hilb^{P}(Y,D)\to W\times \calA^{P}. $$
\begin{prop}\label{rel ori by composition}
Let $(U,S)$ be a smooth pair (Definition \ref{defi of smooth pair}) with $\dim_{\mathbb{C}} U=3$ and 
\begin{equation}\label{equ on yd}Y=\Tot(\omega_{U}(S)), \quad D=\Tot(\omega_S). \end{equation}
If $p$ is a $S$-\textit{compatible} Lagrangian fibration given in Definition \ref{def on s cp lag fib}, then the map $p\circ r_{\calD\to \calY}$ has an orientation in the sense of Definition \ref{ori on even cy}. 
\end{prop}
\begin{proof}
The diagram 
\begin{equation*}
\begin{xymatrix}{
\textbf{Hilb}^{P}(Y,D)    \ar[r]^{r_{\calD\to \calY}\quad } \ar[dr]_{p\circ r_{\calD\to \calY}}  & \textbf{Hilb}^{P|_D}(D)\times \calA^{P}  \ar[d]^{p}  \\
 & W\times \calA^{P} }\end{xymatrix}    \end{equation*} 
induces a fiber sequence 
$$r_{\calD\to \calY}^*\bbL_{p} \to \bbL_{p\circ r_{\calD\to \calY}} \to  \bbL_{r_{\calD\to \calY}}. $$
Therefore 
\begin{align*}\det\bbL_{p\circ r_{\calD\to \calY}}&\cong r_{\calD\to \calY}^*\det \bbL_{p}\otimes \det\bbL_{r_{\calD\to \calY}} \\
&\cong r_{\calD\to \calY}^*\left(\frac{\det \bbL_{\textbf{Hilb}^{P|_D}(D)\times \calA^{P}}}{p^*\det \bbL_{W\times \calA^{P}}}\right)\otimes 
\left(\frac{\det \bbL_{\textbf{Hilb}^{P}(Y,D)}}{r_{\calD\to \calY}^*\det(\bbL_{\textbf{Hilb}^{P|_D}(D)\times \calA^{P}})}\right) \\
&\cong \frac{\det \bbL_{\textbf{Hilb}^{P}(Y,D)}}{r_{\calD\to \calY}^*p^*\det \bbL_{W\times \calA^{P}}} \cong \frac{\det \bbL_{\textbf{Hilb}^{P}(Y,D)/\calA^{P}}}{r_{\calD\to \calY}^*p^*\det \bbL_{(W\times \calA^{P})/\calA^{P}}}.
\end{align*}
By \eqref{final iso} in Theorem \ref{relative ori of cy4 family} and Definition \ref{def on s cp lag fib}, we know 
$$\det\left(\bbL_{p\circ r_{\calD\to \calY}}|_{\Hilb^{P}(Y,D)}\right)\cong \oO_{\Hilb^{P}(Y,D)}, $$
which can be checked to be an orientation.  
\end{proof} 
\begin{remark}\label{rmk on ori2}
In this remark, we explain how to construct a canonical orientation on $p\circ r_{\calD\to \calY}$.

Let $(Y,D)$ be the log Calabi-Yau pair given by \eqref{equ on yd} and 
$$\calY=\Tot(\omega_{\calU/\calA}(\calS)), \quad \calD=\Tot(\omega_{\calS/\calA})\cong \Tot(\omega_{S})\times \calA$$ 
be the universal expanded pair over the stack  $\calA$ of expanded pairs. Denote $\textbf{M}_{\mathrm{cpt}}(-)$  to be the derived moduli stack of compactly supported sheaves on $(-)=\calY/\calA$, $\calU/\calA$, $\calD/\calA$, $\calS/\calA$. Then we have a pushward map
$$\pi_\dagger: \bM_{\mathrm{cpt}}(\calY/\calA) \to \bM_{\mathrm{cpt}}(\calU/\calA), $$
%$$\bM_{\mathrm{cpt}}(\calD/\calA)\to \bM_{\mathrm{cpt}}(\calS/\calA), $$
and a restriction map 
$$r: \bM_{\mathrm{cpt}}(\calY/\calA)\to \bM_{\mathrm{cpt}}(\calD/\calA). $$
Note that 
$\textbf{M}_{\mathrm{cpt}}(\calD/\calA)$ is the $(-1)$-shifted cotangent bundle of $\textbf{M}_{\mathrm{cpt}}(\calS/\calA)$ over $\calA$:
$$\textbf{M}_{\mathrm{cpt}}(\calD/\calA)=\bbT_{\textbf{M}_{\mathrm{cpt}}(\calS/\calA)/\calA}^*[-1], $$
with projection 
$$p_\dagger: \textbf{M}_{\mathrm{cpt}}(\calD/\calA)\to \textbf{M}_{\mathrm{cpt}}(\calS/\calA). $$
As in Theorem \ref{thm on lag}, the map $r$ has a Lagrangian structure (over $\calA$). By \cite[Thm.~2.2]{Cal2}, $p_\dagger$ has a Lagrangian fibration structure (over $\calA$). 
Therefore the composed map 
$$\bM_{\mathrm{cpt}}(\calY/\calA) \to \bM_{\mathrm{cpt}}(\calS/\calA)$$ 
has a $(-2)$-shifted symplectic structure (Lemma \ref{lem on comp lag fib}). In particular, there is an isomorphism 
\begin{equation}\label{equ rmk y/d pair} \bbT_{\bM_{\mathrm{cpt}}(\calY/\calA)/\bM_{\mathrm{cpt}}(\calS/\calA)}\cong  \bbL_{\bM_{\mathrm{cpt}}(\calY/\calA)/\bM_{\mathrm{cpt}}(\calS/\calA)}[-2]. \end{equation}
As in \eqref{equ on -2shift cota}, there is a fiber sequence 
\begin{equation}\label{equ rmk y/d fiber}\pi_\dagger^*\bbL_{\bM_{\mathrm{cpt}}(\calU/\calA)/\bM_{\mathrm{cpt}}(\calS/\calA)}[-2]\to 
\bbT_{\bM_{\mathrm{cpt}}(\calY/\calA)/\bM_{\mathrm{cpt}}(\calS/\calA)} \to \pi_\dagger^*\bbT_{\bM_{\mathrm{cpt}}(\calU/\calA)/\bM_{\mathrm{cpt}}(\calS/\calA)}, \end{equation}
which is self-dual under the pairing \eqref{equ rmk y/d pair}. 
Its restriction to the classical truncation $M_{\mathrm{cpt}}(\calY/\calA)$ provides a canonical isomorphism:
$$\bbE:=\bbL_{\bM_{\mathrm{cpt}}(\calY/\calA)/\bM_{\mathrm{cpt}}(\calS/\calA)}|_{M_{\mathrm{cpt}}(\calY/\calA)}\cong \oO. $$
Denote $\mathbb{V}:=\pi_\dagger^*\bbT_{\bM_{\mathrm{cpt}}(\calU/\calA)/\bM_{\mathrm{cpt}}(\calS/\calA)}|_{M_{\mathrm{cpt}}(\calY/\calA)}$. 
Similar to Remark \ref{rmk on ori}, we choose
\begin{equation}\label{cho of sign2}\calO\xrightarrow{+(-\sqrt{-1})^{\rk(\mathbb{V})}} \calO \end{equation}
as the \textit{canonical choice} of orientation on $M_{\mathrm{cpt}}(\calY/\calA) \to M_{\mathrm{cpt}}(\calS/\calA)$.  

There is a forgetful map 
$$f_Y: \Hilb^{P}(Y,D) \to M_{\mathrm{cpt}}(\calY/\calA). $$
By comparing \eqref{identify two det2}, \eqref{iso of det abs2}, \eqref{equ on Thild}, \eqref{square root1} and using \eqref{equ on Scpt},
the above canonical orientation 
on $$M_{\mathrm{cpt}}(\calY/\calA) \to M_{\mathrm{cpt}}(\calS/\calA)$$ induces a \textit{canonical orientation} on $$\Hilb^{P}(Y,D)\to W \times \calA^{P}.$$ 
%obtained by Proposition \ref{rel ori by composition}.
%When $S=\emptyset$, this choice of orientation is consistent with the choice specified by Remark \ref{rmk on ori}.
\end{remark}

We give examples when \eqref{equ on Scpt} in Definition \ref{def on s cp lag fib} is satisfied. 
\begin{example}\label{example on Scp}
Consider the smooth pair:
$$U=\Tot(\oO_{\mathbb{P}^1}(a_{})\oplus \oO_{\mathbb{P}^1}(b_{})), \quad S=\mathbb{C}^2, $$
where the associated log Calabi-Yau pair is 
$$Y=\Tot(\omega_{U}(S))=\Tot(\oO_{\mathbb{P}^1}(a)\oplus \oO_{\mathbb{P}^1}(b)\oplus \oO_{\mathbb{P}^1}(-1-a-b)), \quad D=\Tot(\omega_{S})=\C^3.$$
By Example \ref{fl ex}, we have an inclusion into the non-commutative Hilbert scheme: 
$$p: \Hilb^n(D)\to W. $$
Then \eqref{equ on Scpt} holds,~i.e. $p^*\det(\bbT_W)\cong K^{-1/2}_{\Hilb^n(D)}$
for the square root $K^{-1/2}_{\Hilb^n(D)}$ in \eqref{square root1}. Therefore 
$$p\circ r_{\calD\to \calY}: \Hilb^P(Y,D)\to W\times \calA^{P} $$
 has an orientation by Proposition \ref{rel ori by composition}. 

\end{example}
\begin{proof}
As $S=\C^2$ is symplectic, therefore  
\begin{align}
\label{equation on sym ksquareroot}K^{-1/2}_{\Hilb^n(D)}&\cong f_D^*\det\left(\pi_{M_D*}\bbF_D\right) 
\end{align}
is the determinant of the tautological bundle of $\Hilb^n(D)$. 
By \eqref{W space}, we have 
$$\bbT_W\cong \left(\mathfrak {gl}_n\otimes \oO\to (\mathfrak {gl}_n^{\oplus 3}\oplus \mathbb{C}^n)\otimes \oO\right), $$
where the RHS is written as a $\GL(n)$-equivariant complex. Then 
\begin{equation}\label{equ on det tw}\det(\bbT_W)\cong \det(\mathbb{C}^n\otimes \oO), \end{equation}
where the RHS is the determinant of the tautological bundle. After restricting to $\Hilb^n(D)$, we know \eqref{equation on sym ksquareroot} and \eqref{equ on det tw} are isomorphic. 
%It is easy to see $p^*\det(\bbT_W)$ is the determinant of the tautological bundle.  
\end{proof} 

\section{Relative invariants and degeneration formulae}
We recollect the general theory of virtual pullbacks \cite{Park1} arising from 
$(-2)$-shifted symplectic structures and then apply to our setting to define relative invariants and family invariants. Finally we prove a degeneration formula relating relative and family invariants.

\subsection{Virtual pullback via symmetric obstruction theory}\label{sect on vir pullback}

%In this section and the next, we write $QM_{g,n}^{R_{\chi}=\omega_{\mathrm{log}}}(\Crit(\phi)/\!\!/G,\beta)$ simply as $QM_{g,n}(\beta)$ or $QM$. We drop $\beta$ from the notation if we take union over all possible $\beta$'s.
%First recall relevant notions and results from \cite{Park1}. 

\begin{definition}(\cite[Prop.~1.7,~\S A.2]{Park1})\label{defi on sym cx}
A \textit{symmetric complex}  $\bbE$ on an algebraic stack
$\calX$ consists of the following data:
\begin{enumerate}
    \item A perfect complex $\bbE$ of tor-amplitude $[-2,0]$ on $\calX$.
    \item A non-degenerate symmetric form $\theta$ on $\bbE$, i.e. a morphism
$$\theta : \calO_{\calX} \to (\bbE \otimes\bbE)[-2]$$
in the derived category of $\calX$, invariant under the transposition
$\sigma : \bbE \otimes \bbE \to  \bbE \otimes \bbE$, and the induced map 
$\iota_{\theta} : \bbE^\vee\to\bbE[-2]$ is an isomorphism.
    \item An orientation $o$ of $\bbE$,~i.e.~an isomorphism~$o:\calO_\calX \to \det(\bbE)$ of line
bundles such that $\det(\iota_\theta)=o\circ o^\vee$.
\end{enumerate}
\end{definition}
%We write its virtual normal cone as $\fC_{\bbE}$.
For a symmetric complex $\bbE$, there is a quadratic function (\cite[\S A.2]{Park1}) from the virtual normal cone $\fC_{\bbE}$ of $\bbE$: 
\begin{equation}\label{def of qua eq1}\fq_\bbE:\fC_{\bbE}\to \bbA^1_{\calX}, \end{equation}
characterized by some naturality conditions. 

\begin{definition}(\cite[Def.~1.9,~\S A.2]{Park1})\label{def on sym ob}
A \textit{symmetric obstruction
theory} for a Deligne-Mumford morphism $f : \calX\to \calY$ between algebraic stacks   is a morphism $\phi : \bbE \to \bfL_f$ in
the derived category of $\calX$ such that
\begin{enumerate}
    \item $\bbE$ is a symmetric complex.
    \item $\phi$ is an obstruction theory in the sense of Behrend-Fantechi \cite{BF}, i.e., $h^0(\phi)$ is an isomorphism and $h^{-1}(\phi)$ is surjective,
where $\bfL_f := \tau^{\geqslant  -1}\bbL_f$ is the truncated cotangent complex.
\end{enumerate}
\end{definition}
%The cone stack $\fC_{\bbE}$ is called the virtual normal cone.
The obstruction theory $\phi$ induces a closed
embedding of the intrinsic normal cone 
$$\fC_f\inj \fC_{\bbE}. $$
\begin{definition}\label{def on iso sym ob}
A symmetric obstruction theory $\phi : \bbE \to \bfL_f$  is {\it isotropic} if the intrinsic normal
cone $\fC_f$ is isotropic in the virtual normal cone $\fC_{\bbE}$, i.e. the restriction
$\fq_{\bbE}|_{\fC_f}
: \fC_f\inj \fC_{\bbE}\to \bbA^1$
 vanishes.
\end{definition}
%\begin{remark} By \cite[Ex.~1.17]{Park1} which uses \cite{BBBJ}, having $(-2)$-shifted symplectic structure automatically implies isotropic condition. \end{remark}
Isotropic symmetric obstruction theory implies the existence of square root virtual pullback which we now briefly recall. 
For a symmetric complex $\bbE$ on an algebraic stack $\calX$, let $\fQ(\bbE)$ be the zero locus of the quadratic function $\fq_{\bbE}:\fC_\bbE\to \bbA^1_\calX$,
there is a \textit{square root Gysin pullback} \cite[Def.~A.2]{Park1} 
$$\sqrt{0^!_{\fQ(\bbE)}}: A_*(\fQ(\bbE)) \to A_*(\calX), $$
if $\calX$ is a quotient of a separated Deligne-Mumford stack by an algebraic group.
\begin{definition}
 Assume that $f:\calX\to\calY$ is a Deligne-Mumford morphism between algebraic stacks with an 
isotropic symmetric obstruction theory $\phi:\bbE\to\bfL_f$. It induces  a closed embedding $a:\fC_f\to \fQ(\bbE)$.
The \textit{square root virtual pullback} is the composition 
\begin{equation}\label{eqn:sqrt_pull}
    \sqrt{f^!}:A_*(\calY)\xrightarrow{sp_f}A_*(\fC_f )\xrightarrow{a_*} A^{*} (\fQ(\bbE))\xrightarrow{\sqrt{0^!_{\fQ(\bbE)}}} A_{*}(\calX ),
\end{equation}
 where $sp_f:A_*(\calY)\to A_*(\fC_f)$
is the specialization map (\cite[Const.~3.6]{Man}).
\end{definition}
The map $\sqrt{f^!}$ commutes with \textit{projective pushforwards, smooth pullbacks, and Gysin pullbacks for regular immersions} and can be used 
to define the square root virtual pullback for any base change of $f$ (e.g.~\cite[Eqn.~(1.14)]{Park1}). 
Moreover, it has a \textit{functoriality} with respect to morphisms compatible with symmetric obstruction theories \cite[Thm.~A.4]{Park1} as explained below.

Let $f:\calX\to\calY$ be a Deligne-Mumford (DM) morphism of algebraic stacks having reductive stabilizer groups and affine diagonals\footnote{The original assumptions of \cite[Thm.~A.4]{Park1} are (1) $\calY$ is the quotient of a DM stack by a linear algebraic group, (2) $\calX$ has the resolution property and 
(3) $f$ is quasi-projective. Hyeonjun Park informed us that (1)--(3) can be replaced by the assumption stated above where details will appear in a forthcoming work \cite{BP}.}, 
which are satisfied if $\calX$ and $\calY$ are quotient stacks of 
separated DM stacks by algebraic tori.
Let $g:\calY\to\calZ$ be a DM morphism of algebraic stacks.
Assume $\phi_g:\bbE_g\to \bfL_g$, $\phi_{g\circ f}:\bbE_{g\circ f}\to \bfL_{g\circ f}$ are isotropic symmetric obstruction theories, 
$\phi_f:\bbE_f\to \bfL_f$ is a perfect obstruction theory \cite{BF} and they are {\it compatible},~i.e.~there exists a perfect complex $\bbD$ and morphisms 
$\alpha:\bbE_{g\circ f}\to \bbD$ and $\beta:f^*\bbE_g\to \bbD$ fitting into diagram of exact triangles:
\begin{equation}
   \label{eqn:triang_obst}
\xymatrix{\bbD^\vee[2]\ar[r]^{\alpha^\vee}\ar[d]_{\beta^\vee}&\bbE_{g\circ f}\ar[d]_\alpha\ar[r]^\delta&\bbE_f\ar@{=}[d]\\
f^*\bbE_g\ar[r]^\beta\ar[d]_{f^*\phi_g}&\bbD\ar[r]^\gamma\ar[d]_{\phi'_{g\circ f}}&\bbE_f\ar[d]_{\phi'_f}\\
\tau^{\geqslant -1}f^*\bfL_g\ar[r]&\bfL_{g\circ f}\ar[r]&\bfL_f',}
\end{equation}
where orientation of $\bbE_{g\circ f}$ is compatible with orientation of $\bbE_{g}$ (\cite[Def.~2.1]{Park1}). 
%preserves orientation 
%(i.e.~the orientation of $\bbE_{g\circ f}$ is induced by the orientation of $\bbE_{g}$).
Here $\phi_{g\circ f}=\phi_{g\circ f}'\circ\alpha$, $\phi_f=r\circ \phi'_f$ with $\bfL'_f$ is the cone of 
$\tau^{\geqslant  -1}f^*\bfL_g\to\bfL_{g\circ f}$ and $r:\bfL'_f\to \bfL_f$ the truncation. 
Then we have \begin{equation}\label{eqn:Park}
\sqrt{(g\circ f)^!}=f^!\circ\sqrt{g^!},\end{equation}
where $f^!$ is the virtual pullback of Manolache \cite{Man}. 
In the presence of a torus action, the above extends to the torus equivariant setting.

\subsection{Relative invariants}\label{sect on rel invs}

Let $(Y,D)$ be a smooth log Calabi-Yau pair (Definition \ref{defi of log cy pair}) so that $D$ is a Calabi-Yau 3-fold. 
Fix a $K$-theory class $P\in K_{c, \leq 2}^{\mathrm{num}, } (Y)$, we have restriction map \eqref{rest map2}: 
\begin{equation}r_{\calD\to \calY}:\textbf{Hilb}^P(Y,D)\to \textbf{Hilb}^{P|_D}(D)\times \calA^P,  \end{equation}
which has a Lagrangian structure (Theorem \ref{thm on lag}). 

In this section, we work under the following assumption. 
\begin{ass}\label{ass on lag fib0}
We assume there is 
a quasi-projective smooth scheme $W$, and  a regular function\footnote{We assume the classical critical locus $\Crit(\phi)\subseteq Z(\phi)$ embeds inside the zero locus $Z(\phi)$ of the function, there is no loss of generality to do it (see \cite[Rmk~2.2]{CZ}).} $\phi$ on $W$,
such that there is
an equivalence 
\begin{equation}\label{equ on hilb equ dcri}\textbf{Hilb}^{P|_D}(D)\cong \textbf{Crit}(\phi:W\to \mathbb{C}) \end{equation} 
of $(-1)$-shifted symplectic derived scheme.  
%under which  $K^{-1/2}_{\Hilb^{P|_D}(D)}$ \eqref{square root1.2} satisfies 
%$$K^{-1/2}_{\Hilb^{P|_D}(D)}\cong \det\left(\bbT_W|_{\Hilb^{P|_D}(D)}\right).$$
\end{ass}
Under the above assumption, the composition (below $p$ denotes the inclusion deduced by \eqref{equ on hilb equ dcri}): 
\begin{equation}\label{equ on cop lag fib}p\circ r_{\calD\to \calY}: \textbf{Hilb}^{P}(Y,D) \stackrel{r_{\calD\to \calY}}{\to} \textbf{Hilb}^{P|_D}(D)\times \calA^P\stackrel{p}{\to} W\times \calA^P \end{equation}
has a $(-2)$-shifted symplectic structure (Proposition \ref{prop on iso condition2}). 
%Here $p$ denotes the natural inclusion given by \eqref{equ on hilb equ dcri}.

%$$p\circ r_{\calD\to \calY}: \textbf{Hilb}^{P}(Y,D) \to W\times \calA^P $$
Restricting the cotangent complex of $p\circ r_{\calD\to \calY}$ to the classical truncation, which, by an abuse of notation, we still write as 
%$p\circ r_{\calD\to \calY}$
$$p\circ r_{\calD\to \calY}: \Hilb^{P}(Y,D) \to W\times \calA^P, $$
we get a canonical symmetric obstruction theory in the sense of Definition \ref{def on sym ob}: 
\begin{equation}\label{obs theory relW}\bbE_{\Hilb^{P}(Y,D)/(W\times \calA^P)} \to \bbL_{\Hilb^{P}(Y,D)/(W\times \calA^P)}. \end{equation}
By the footnote of Assumption \ref{ass on lag fib}, we have $\Crit(\phi)\subseteq Z(\phi)$. By the construction, $p\circ r_{\calD\to \calY}$ factors through $\Crit(\phi)$. Therefore we obtain a Cartesian diagram of classical stacks: 
\begin{equation*}\begin{xymatrix}{
\Hilb^{P}(Y,D)  \ar[r]^{ }  \ar@{=}[d] \ar@{}[dr]|{\Box} & Z(\phi)\times \calA^P  \ar[d]^{}  \\
\Hilb^{P}(Y,D) \ar[r]^{\,\,\, p\circ r_{\calD\to \calY}} & W\times \calA^P.
}\end{xymatrix}\end{equation*}
Hence we can pullback \eqref{obs theory relW} to the upper line and get a canonical symmetric obstruction theory
\begin{equation}\label{obs theory relW2}\bbE_{\Hilb^{P}(Y,D)/(Z(\phi)\times \calA^P)} \to \bbL_{\Hilb^{P}(Y,D)/(Z(\phi)\times \calA^P)}. \end{equation}
%The following isotropic result will be proven in \cite{CZZ}. 
The following isotropic result will be proven in Appendix \ref{app on proof}.
\begin{prop}\label{iso cond 2}
The symmetric obstruction theory \eqref{obs theory relW2} is isotropic in the sense of Definition~\ref{def on iso sym ob} if $P\in K_{c, \leq 1}^{\mathrm{num}} (Y)$. 
\end{prop}
Therefore, we have  the following:
%To summarize, we obtain the following:
\begin{theorem}\label{thm on relative invs}
Let $P\in K_{c,\leq 1}^{\mathrm{num}} (Y)$.
Assume Assumption \ref{ass on lag fib} holds and the classical truncation of \eqref{equ on cop lag fib} has an orientation, then  
%the map $$p\circ r_{\calD\to \calY}: \Hilb^{P}(Y,D) \to W\times \calA^P $$
%has a canonical symmetric obstruction theory in sense of Definition \ref{def on sym ob}, which is isotropic after base change via the zero locus $Z(\phi)\hookrightarrow W$ of $\phi: W\to \C$. 
%In particular, 
there is a square root virtual pullback
\begin{equation}\label{sqr pb of hilb}\sqrt{(p\circ r_{\calD\to \calY})^!} : A_*(Z(\phi)\times \calA^P)\to A_*(\Hilb^{P}(Y,D)), \end{equation}
which satisfies usual bivariance properties and functoriality spelled out in \S \ref{defi on sym cx}. 
%and \cite[Prop.~1.15,~Prop.~1.18,~Thm.~2.2]{Park1}.
\end{theorem}

Given a smooth pair (Definition \ref{defi of smooth pair}) $(U,S)$ with $\dim_{\mathbb{C}} U=3$, let 
$$Y=\Tot(\omega_{U}(S)), \quad D=\Tot(\omega_S)$$ 
be the (associated) log Calabi-Yau 4-fold pair. Recall the line bundle $K^{-1/2}_{\Hilb^{P|_D}(D)}$ \eqref{square root1.2}. 
\begin{ass}\label{ass on lag fib}
Given 
$\textbf{Hilb}^{P|_D}(D)\cong \textbf{Crit}(\phi:W\to \mathbb{C})$ as in Assumption \ref{ass on lag fib0}. 
We assume 
$$K^{-1/2}_{\Hilb^{P|_D}(D)}\cong \det\left(\bbT_W|_{\Hilb^{P|_D}(D)}\right).$$
\end{ass}

By Proposition \ref{rel ori by composition}, the classical truncation of \eqref{equ on cop lag fib} then has an orientation in the sense of Definition \ref{ori on even cy} and 
we choose the canonical orientation as specified in Remark \ref{rmk on ori2}. 

\begin{definition}\label{defi of rel inv}
Let $Y=\Tot(\omega_{U}(S))$, $D=\Tot(\omega_S)$ be as above. 
Fix a $K$-theory class $P \in K_{c,\leq 1}^\mathrm{{num}} (Y)$. Under Assumption \ref{ass on lag fib}, we define \textit{relative invariants}:
\begin{equation}\label{equ of rel inv} \Phi^{P}_{Y,D}:=\sqrt{(p\circ r_{\calD\to \calY})^!}\circ \boxtimes : A_*(Z(\phi))\otimes A_{*}(\calA^{P}) \to A_*(\Hilb^P(Y,D)),  \end{equation}
where $\boxtimes$ is the exterior product: 
$$\boxtimes: A_*(Z(\phi))\otimes A_{*}(\calA^{P})  \to A_*(Z(\phi)\times \calA^{P}), \quad (\alpha,\beta)\mapsto (\alpha \times \beta). $$ 
\end{definition}
\begin{remark}\label{rmk on toru on pair}
When there is a torus $T$ acting on the pair $(U,S)$ and the equivalence \eqref{equ on hilb equ dcri}, it naturally induces an action on $(Y,D)$, the above pullback map \eqref{sqr pb of hilb} lifts to a map between $T$-equivariant Chow groups. 
\end{remark}

\subsection{Family invariants}\label{sect on fami invs}

Let $\pi : X \to \bbA^1$ be a  simple degeneration of Calabi-Yau 4-folds.
Fix $K$-theory classes $P_t\in K_{c, \leq 2}^{\mathrm{num}} (X_t)$ for all $t\in \mathbb{A}^1$, the map 
\begin{equation}\label{equ on hilb stack on family}\textbf{Hilb}^{P_t}(X/\bbA^1)\to \calC^{P_t} \end{equation}
has a (relative) $(-2)$-shifted symplectic structure  (Theorem \ref{thm on sft symp str}).
%We choose the orientation as specified by Remark \ref{rmk on ori}. 
%We denote the classical truncation of $\pi$ by the same notation. 
By restricting the cotangent complex of $\textbf{Hilb}^{P_t}(X/\bbA^1)/\calC^{P_t}$ to the classical truncation, we obtain a canonical symmetric obstruction theory 
\begin{equation}\label{obs theory relW3}\mathbb{E}_{\Hilb^{P_t}(X/\bbA^1)/\calC^{P_t}}\to \mathbb{L}_{\Hilb^{P_t}(X/\bbA^1)/\calC^{P_t}}. \end{equation}
%The following isotropic result will be proven in \cite{CZZ}. 
The following isotropic result will be proven in Appendix \ref{app on proof}.
\begin{prop}\label{iso cond 1}
The symmetric obstruction theory \eqref{obs theory relW3} is isotropic in the sense of Definition~\ref{def on iso sym ob} if $P_t\in K_{c, \leq 1}^{\mathrm{num}} (X_t)$ for all $t\in \mathbb{A}^1$. 
\end{prop}
\begin{remark}
Here we restrict to curves/points classes on Calabi-Yau 4-folds. One can consider surface classes and other cases by a more detailed study of Hodge filtrations as \cite[Rmk.~5.2]{BKP}.  
\end{remark}
%To summarize, we obtain the following: 
Therefore, we have  the following:
\begin{theorem}\label{thm pb on family}
Fix $K$-theory classes $P_t \in K_{c,\leq 1}^\mathrm{{num}} (X_t)$ for all $t\in \mathbb{A}^1$ and assume the classical truncation of \eqref{equ on hilb stack on family} has an orientation. 
%,~i.e.~their support has dimension not bigger than one \eqref{k class with dim}. 
Then there is a square root virtual pullback
\begin{equation}\label{sqr pb of hilb family} \sqrt{\pi^!} : A_*(\calC^{P_t})\to A_*(\Hilb^{P_t}(X/\bbA^1)), \end{equation}
which satisfies usual bivariance properties and functoriality spelled out in \S \ref{defi on sym cx}. 

\end{theorem}
%\begin{proof}Similar to Theorem \ref{thm on relative invs} and use Proposition \ref{iso cond 1} to conclude isotropic condition. \end{proof}
%\begin{remark}\label{rmk on ori on family inv}\eqref{equ on spectral cons}\end{remark}
Let $U\to \bbA^1$ be a simple degeneration of 3-folds and $$X=\Tot(\omega_{U/\bbA^1})\to \bbA^1$$ 
be the associated simple degeneration of Calabi-Yau 4-folds. 
By Theorem \ref{ori of cy4 family}, 
the classical truncation of \eqref{equ on hilb stack on family} has an orientation in the sense of Definition \ref{ori on even cy}. 
We choose the canonical orientation as specified in Remark \ref{rmk on ori}. 
\begin{definition}\label{defi of fam inv}
Let $X=\Tot(\omega_{U/\bbA^1})\to \bbA^1$ be as above. 
Given $K$-theory classes $P_t \in K_{c,\leqslant 1}^\mathrm{{num}} (X_t)$ for all $t\in \bbA^1$, we define \textit{family invariants}: 
\begin{equation}\label{equ of fam inv}\Phi^{P_t}_{X/\bbA^1}:=\sqrt{\pi^!} : A_*(\calC^{P_t}) \to A_*(\Hilb^{P_t}(X/\bbA^1)).  \end{equation}
\end{definition}
\begin{remark}
When there is a torus $T$ acting on the family $U\to \bbA^1$, it naturally induces an action on $X\to \bbA^1$ preserving Calabi-Yau volume forms on fibers, the above pullback map \eqref{equ of fam inv} lifts to a map between $T$-equivariant Chow groups (where $T$-action on $\calC^{P_t}$ is trivial). 
\end{remark}

\subsection{Degeneration formulae} \label{Sec-deg-formula}

In this section, we will relate relative invariants (Definition~\ref{defi of rel inv}) of $(Y_\pm,D)$ and family invariants (Definition~\ref{defi of fam inv}) 
of $X/ \bbA^1$ in a degeneration formula. We work in the following setting. 

\begin{setting}\label{set of dege for}
Let $U\to \bbA^1$ be a simple degeneration of 3-folds (Definition \ref{defi of dege}) with central fiber $U_0=U_+\cup_S U_-$.
Consider
$$X=\Tot(\omega_{U/\bbA^1})\to \bbA^1.
$$ 
It is a CY simple degeneration (Definition \ref{cy simple dege}) whose fiber satisfies that  
\begin{itemize}
\item $X_c=\Tot(\omega_{U_c})$ if $c\neq 0$, 
\item $X_0=Y_+\cup_D Y_-$, where $Y_{\pm}=\Tot(\omega_{U_\pm}(S))$ and $D=\Tot(\omega_S)$. 
\end{itemize}
%Fix classes $P_t\in K_{c, \leq 1}^{\mathrm{num}} (X_t)$ for $t\in \bbA^1$. 
\end{setting}

\begin{data}\label{data on auto}
In Setting \ref{set of dege for} with Assumption \ref{ass on lag fib}, assume there is an \textit{automorphism}
\begin{equation}\label{ord 2 autom1}\bar{\sigma}: X\to X,   \end{equation}
which preserves fibers of $X\to \bbA^1$ and also $Y_{\pm}$, $D$ in $X_0$, and there is an \textit{automorphism} 
\begin{equation}\label{ord 2 autom}\sigma:W\to W \end{equation}
on the $W$ in \eqref{equ on hilb equ dcri} such that 
%$\sigma^{*}\phi=-\phi$ and the induced automorphism $(\bar{\sigma}|_{D})_*: \Hilb(D)\to \Hilb(D)$ coincides with $\sigma|_{\Hilb(D)}$.  
%which obviously preserves $\Crit(\phi)$ and $Z(\phi)$:
%$$\sigma:\Crit(\phi)\to \Crit(\phi), \quad \sigma: Z(\phi) \to Z(\phi). $$ 
\begin{itemize}
\item $\sigma^{*}\phi=-\phi$, 
\item the induced automorphism $(\bar{\sigma}|_{D})_*: \Hilb^{P_D}(D)\to \Hilb^{P_D}(D)$ coincides with $\sigma|_{\Hilb^{P_D}(D)}$. 
\end{itemize}
\end{data}
\begin{example}
As $X=\Tot(\omega_{U/\bbA^1})$, we can choose the order two element in the $\mathbb{C}^*$-action on fibers of $X\to U$ as a candidate of $\bar{\sigma}$ \eqref{ord 2 autom1}. We refer to \cite[Rmk.~4.18]{CZ} for the choice of automorphism \eqref{ord 2 autom} 
in examples given by quivers with potentials. 
In particular, in Example \ref{fl ex}, we can take  
$$\sigma(b_1,b_2,b_3,v)=(b_1,b_2,-b_3,v), $$
where the third coordinate corresponds to the fiber of $\mathbb{C}^3=D\to S=\mathbb{C}^2$. 
Then $\sigma$ is homotopy to the identity map and compatible with $\bar{\sigma}$ as required in Data \ref{data on auto}. 
\end{example}
We define the sum function  
$$\boxplus^n \phi: W^n\to \mathbb{C}, \quad (x_1,x_2,\ldots,x_n)\mapsto \sum_{i=1}^n\phi(x_i), $$
and denote $Z(\boxplus^n \phi)$ to be its zero locus. 

Consider \textit{diagonal embedding} $$\Delta: W\to W\times W$$ and \textit{anti-diagonal embedding} 
\begin{equation}\label{equ on bar del}\bar{\Delta}: W\to W\times W, \quad x\mapsto (\sigma x,x). \end{equation}
%$$\Delta(x)=(x,x), $$
%and $\sigma$, $\bar{\Delta}$ are defined by 
%$$\sigma(x,y)=(\sigma x,y), \quad \bar{\Delta}(x)=(\sigma x,x).$$ 
%  \ar@{}[dr]|{\Box}
We have the  commutativity of  following diagrams  
\begin{equation} \label{eqn:chasing}
\xymatrix{
\Hilb^{P_-}(Y_-,D)\times_{\bar{\Delta}(W)} \Hilb^{P_+}(Y_+,D) \ar@/_10pc/[ddd]_{i_{\bar{\Delta}}}    \ar@/^0.5pc/[rd]^{\quad \quad r_{\bar{\Delta}}} & \\
 \Hilb^{P_-}(Y_-,D)\times_{\Delta(W)} \Hilb^{P_+}(Y_+,D) \ar[d]^{i_{\Delta}} \ar[r]^{\quad \quad \quad\quad  r_{\Delta}}  \ar@{}[dr]|{\Box}  &  W \times \calA^{P_-} \times \calA^{P_+} \ar[d]^{\Delta\times \id} \\
      \Hilb^{P_-}(Y_-,D)\times  \Hilb^{P_+}(Y_+,D)  \ar[d]^{\bar{\sigma}\times\id:\, (a,b)\mapsto (\bar{\sigma} a,b)}    \ar[r]^{r_{\calD\to \calY_+}\times r_{\calD\to \calY_-} } &  W \times W\times \calA^{P_-}\times \calA^{P_+} \ar[d]^{\sigma^{}: \, (x,y,z,w)\mapsto (\sigma^{}x,y,z,w)}    \\
%Z(d\phi)/\!\!/G  \ar[r]^{\Delta \quad \quad } \ar[d]_{ }  & Z(d\phi)/\!\!/G  \times Z(d\phi)/\!\!/G \ar[d]_{ }    \\
  \Hilb^{P_-}(Y_-,D)\times  \Hilb^{P_+}(Y_+,D)   \ar[r]^{  }_{ }  &  W \times W\times \calA^{P_-}\times \calA^{P_+},     }
\end{equation}
where the automorphism  
\begin{equation}\label{equ on auto sig}\bar{\sigma}: \Hilb^{P_-}(Y_-,D)\to \Hilb^{P_-}(Y_-,D) \end{equation}
is induced by the automorphism \eqref{data on auto}. The bottom square commutes because automorphisms $\sigma$ and $\bar{\sigma}$ are compatible in Data \ref{data on auto}. 
where the map $r_{\calD\to \calY_-}\times r_{\calD\to \calY_+}$ goes to $\Hilb^{P_D}(D) \times \Hilb^{P_D}(D)$ (where $P_D=P_-|_D=P_+|_D$) and we omit the natural inclusion 
$$\Hilb^{P_D}(D) \times \Hilb^{P_D}(D)\cong\Crit(\phi)\times \Crit(\phi)\hookrightarrow Z(\phi\boxplus \phi)$$ 
in the above diagram. By a diagram chasing along \eqref{eqn:chasing}, we obtain:
\begin{lemma}\label{lem on diag iso to antidiag}
There is a canonical isomorphism $\bar{\sigma}$ making the following diagram commutative
\begin{align*} 
%\label{diag cpr Deltabar}
\xymatrix{
\Hilb^{P_-}(Y_-,D)\times_{\bar{\Delta}(W)} \Hilb^{P_+}(Y_+,D) \ar[r]^{\quad\quad\quad\quad  r_{\bar{\Delta}} } \ar[d]^{\bar{\sigma}}_{\cong}  & W \times \calA^{P_-} \times \calA^{P_+} \ar@{=}[d]    \\
\Hilb^{P_-}(Y_-,D)\times_{\Delta(W)} \Hilb^{P_+}(Y_+,D) \ar[r]^{\quad\quad\quad\quad  r_{\Delta} }  &  W \times \calA^{P_-} \times \calA^{P_+}.    }
\end{align*}
\end{lemma}
%By Proposition \ref{prop on iso of two mod}, we have a  commutative diagram
%$$\xymatrix{ \Hilb^{P_-} (Y_-, D) \times_{\Delta(W)} \Hilb^{P_+} (Y_+, D)   \ar[r]^{\quad \quad \quad  \,\, \cong} \ar[d]^{r_{\Delta}} & \Hilb^{(P_-, P_+)} (X/\bbA^1)_0^\dagger \ar[d] \\
% \calA^{P_-} \times \calA^{P_+}\times W   \ar[r]^\cong & \calC_0^{\dagger, (P_-, P_+)}\times W, }$$ where $r_{\Delta}$ is as in diagram \eqref{eqn:chasing}. 

Recollecting notions in \S \ref{sect on rel hilb stack}, we get the following Cartesian diagrams: 
\begin{equation}\label{diag on cpr two maps0}
\xymatrix{
\coprod_{(P_-, P_+)} \Hilb^{(P_-, P_+)} (X/\bbA^1)_0^\dagger \ar@{=}[r] \ar@{}[dr]|{\Box} \ar[d]_{\pi_{\mathrm{node}}} & \Hilb^{P_0} (X/\bbA^1 )_0^\dagger \ar[r]^{n} \ar[d] \ar@{}[dr]|{\Box}   & \Hilb^{P_0} (X/\bbA^1 )_0 \ar[r]^{ } \ar[d]^{\pi_{\mathrm{node}}'} \ar@{}[dr]|{\Box}   &  \Hilb^{P_t} (X/\bbA^1) \ar[d]^{\pi}  \\
\coprod_{(P_-, P_+)} \calC_0^{\dagger, (P_-, P_+)} \ar@{=}[r] & \calC_0^{\dagger,P_0} \ar[r]^{n}_{\mathrm{normalization}} & \calC_0^{P_0} \ar[r]^{i_0}  \ar[d] \ar@{}[dr]|{\Box} & \calC^{P_t} \ar[d] \\
& & 0 \ar[r]^{i_0} & \mathbb{A}^1.}
\end{equation}
Using the following shorthand
$$\calH:=\Hilb^{P_-} (Y_-, D) \times_{\Delta(W)} \Hilb^{P_+} (Y_+, D), \quad \bar{\calH}:=\Hilb^{P_-} (Y_-, D) \times_{\bar{\Delta}(W)} \Hilb^{P_+} (Y_+, D), $$
 diagram \eqref{eqn:chasing}, and Lemma \ref{lem on diag iso to antidiag} lead to the following commutative diagram
\begin{equation}\label{diag on cpr two maps}
\xymatrix{
\calH\ar[d]_{\id_{\calH}\times ev_\Delta}\ar[drr]^{r_{\Delta} \quad }  & &   \bar{\calH} \ar[ll]_{\bar{\sigma}}   \ar[d]^{r_{\bar{\Delta}}} \\
\calH\times W\ar[d]_{p_{\calH}} \ar[rr]^{\pi_{\mathrm{node}}\times\id_W\quad \quad } \ar@{}[drr]|{\Box} & & \calC_0^{\dagger, (P_-, P_+)}\times W  \ar[d]^{p_{W}} \\
\calH\ar[rr]^{\pi_{\mathrm{node}} \quad \quad }& & \calC_0^{\dagger, (P_-, P_+)},
}\end{equation}
where the lower square is Cartesian. 
Here $ev_\Delta$ is the composition of the restriction map (at the common divisor) and the embedding $\Crit(\phi)\hookrightarrow W$.   

By Theorem \ref{thm pb on family}, the map $\pi$ in diagram \eqref{diag on cpr two maps0} admits a square root virtual pullback whose base change (via diagrams \eqref{diag on cpr two maps0} and \eqref{diag on cpr two maps}) defines $\sqrt{(\pi_{\mathrm{node}}\times\id_W)^!}$. 
As in Theorem \ref{thm on relative invs}, the map $r_{\calD\to \calY_+}\times r_{\calD\to \calY_-}$ in diagram \eqref{eqn:chasing} admits a square root virtual pullback (see also \eqref{tensor of rel inv}) whose base change 
defines $\sqrt{r_{\bar{\Delta}}^!}$. The map $\id_{\calH}\times ev_\Delta$ is a regular embedding \cite[Def.~1.20]{Vis}, \cite[App.~B.7.3]{Fu}, so there is a Gysin pullback $(\id_{\calH}\times ev_\Delta)^!$. 

We have the following compatibility between the above pullbacks. 
%which is similar to the compatibility in quasimap theory for quivers with potentials \cite[Prop.~4.22]{CZ}. 
\begin{prop}\label{prop on gluing form}
Notations as above, we have
\begin{align*}\bar{\sigma}^*\circ (\id_{\calH}\times ev_\Delta)^!\circ \sqrt{(\pi_{\mathrm{node}}\times\id_W)^!}=\sqrt{r_{\bar{\Delta}}^!}. \end{align*}
\end{prop}
\begin{proof}
The proof follows the same approach as the proof of \cite[Prop.~4.22]{CZ}. It is enough to construct diagram \eqref{eqn:triang_obst} for maps 
$$(\id_{\calH}\times ev_\Delta) \circ \bar{\sigma}, \,\,\, (\pi_{\mathrm{node}}\times\id_W), \,\,\, r_{\bar{\Delta}}. $$
As $\bar{\sigma}$ is an isomorphism, one is reduced to construct diagram \eqref{eqn:triang_obst} for maps
$$e:=(\id_{\calH}\times ev_\Delta), \,\,\, h:=(\pi_{\mathrm{node}}\times\id_W), \,\,\, g:=r_{\Delta}. $$
Consider the derived enhancement of diagram \eqref{diag on cpr two maps0}: 
\begin{align}\label{diag on derived enh1}
%\label{diag cpr Deltabar}
\xymatrix{
\textbf{Hilb}^{(P_-, P_+)} (X/\bbA^1)_0^\dagger  \ar[r]^{ } \ar[d]_{\pi_{\mathrm{\textbf{node}}}} \ar@{}[dr]|{\Box}  & \textbf{Hilb}^{P_t} (X/\bbA^1) \ar[d]^{ \pi}    \\
\calC_0^{\dagger, (P_-, P_+)} \ar[r]^{i_{(P_-, P_+)}}  &  \calC^{P_t},    }
\end{align}
where $\textbf{Hilb}^{P_t} (X/\bbA^1)$ is defined using diagram  \eqref{equ def derive enh of hilb(X/A)} and $\textbf{Hilb}^{(P_-, P_+)} (X/\bbA^1)_0^\dagger $ is defined 
by the above homotopy pullback diagram. Further homotopy pullback via diagram \eqref{diag on cpr two maps} 
defines a derived enhancement of $h:=\pi_{\mathrm{node}}\times\id_W$. Let $\bbE_h$ be the restriction of the derived cotangent complex to its classical truncation, then we obtain 
a symmetric obstruction theory (since $\pi$ has a symmetric obstruction theory by Theorem \ref{thm pb on family}): 
$$\phi_h: \bbE_h\to \bfL_h:=\tau^{\geqslant  -1}\bbL_{h}.$$
Consider two derived enhancement of $g:=r_{\Delta}$. One of them is constructed via the following derived enhancement of part of diagram \eqref{eqn:chasing}: 
\begin{equation}\label{diag on derived enh2}
\xymatrix{
\textbf{Hilb}^{P_-}(Y_-,D)\times_{\Delta(W)} \textbf{Hilb}^{P_+}(Y_+,D)  \ar[dd]_{\bf r_{\Delta}}   \ar[r]^{\quad \bf i_{\Delta}}  \ar@{}[ddr]|{\Box}  &   
 \textbf{Hilb}^{P_-}(Y_-,D)\times \textbf{Hilb}^{P_+}(Y_+,D)  
\ar[d]^{r_{\calD\to \calY_-}\times r_{\calD\to \calY_+}}   \\
 &  \textbf{Hilb}(D)^{P_- |_D} \times \textbf{Hilb}(D)^{P_+ |_D} \times \calA^{P_-} \times \calA^{P_+}  
\ar[d]^{\mathrm{inclusion}\times \id_{\calA^2}}  \\
W\times \calA^{P_-} \times \calA^{P_+}  \ar[r]^{\Delta\quad \quad  }_{(x,y,z)\mapsto (x,x,y,z) \quad \,\, } &  W\times   W \times \calA^{P_-} \times \calA^{P_+}.    }
\end{equation}
The restriction of the cotangent complex of $\bf r_{\Delta}$ to the classical truncation gives rise to a symmetric obstruction theory 
$$\phi_g: \bbE_g\to \bfL_g. $$
%since the right vertical map has a symmetric obstruction theory (Theorem \ref{thm on relative invs}).
The other derived enhancement of $g=r_{\Delta}$ is given by the following homotopy pullback diagram:
\begin{equation}\label{diag on derived enh3}
\xymatrix{
\textbf{Hilb}^{P_-}(Y_-,D)\times_{\textbf{Hilb}^{P_-|_D}(D)} \textbf{Hilb}^{P_+}(Y_+,D)  \ar[d]_{\bf \tilde{r}_{\Delta}}   \ar[r]^{\quad\quad\,\, \bf \tilde{i}_{\Delta}}  \ar@{}[dr]|{\Box}  &    \textbf{Hilb}^{P_-}(Y_-,D)\times \textbf{Hilb}^{P_+}(Y_+,D)  
\ar[d]^{r_{\calD\to \calY_-}\times r_{\calD\to \calY_+}}   \\
\textbf{Hilb}^{P_- |_D}(D)\times \calA^{P_-} \times \calA^{P_+} \ar[r]^{\Delta\quad \quad \quad \,\,\, }_{(x,y,z)\mapsto (x,x,y,z) \quad \quad \quad\,\,\,}  &  \textbf{Hilb}^{P_- |_D}(D) \times \textbf{Hilb}^{P_+ |_D}(D)\times 
\calA^{P_-} \times \calA^{P_+}.   }
\end{equation}
These two derived enhancements are related by the following commutative diagram of derived stacks: 
\begin{equation}\label{diag on two der}
\xymatrix{
\textbf{Hilb}^{P_-}(Y_-,D)\times_{\textbf{Hilb}^{P_-|_D}(D)} \textbf{Hilb}^{P_+}(Y_+,D)  \ar[d]_{\bf j}   \ar[r]^{\quad\quad\quad\,\,\,\,\,  \bf \tilde{r}_{\Delta}   }  &     
\textbf{Hilb}^{P_- |_D}(D)\times  \calA^{P_-}\times \calA^{P_+} \ar[d]^{\iota}   \\
\textbf{Hilb}^{P_-}(Y_-,D)\times_{\Delta(W)} \textbf{Hilb}^{P_+}(Y_+,D) \ar[r]^{\quad\quad\quad\quad \bf r_{\Delta} }  & W \times \calA^{P_-} \times \calA^{P_+},  }
\end{equation}
where $\iota$ is induced by the inclusion $p:\textbf{Hilb}^{P_- |_D}(D)\to W$. The classical truncation of $\bf j$ induces an isomorphism of classical stacks 
\begin{align*}&\quad \, t_0\left(\textbf{Hilb}^{P_-}(Y_-,D)\times_{\textbf{Hilb}^{P_-|_D}(D)}\textbf{Hilb}^{P_+}(Y_+,D)  \right)\\ 
&\cong t_0\left(\textbf{Hilb}^{P_-}(Y_-,D)\times_{\Delta(W)} \textbf{Hilb}^{P_+}(Y_+,D) \right)=\calH, \end{align*}
because the (classical) restriction map already lies in $\Hilb^{P_- |_D}(D)$. 

By Lemma \ref{lem on ynode der cmp}, we have commutative diagram
 \begin{equation}\label{diag def fnote bar der}
\xymatrix{
\textbf{Hilb}^{(P_-, P_+)} (X/\bbA^1)_0^\dagger \ar[d]_{\textbf{e}:=\id_{}\times ev_\Delta}   & &  \textbf{Hilb}^{P_-}(Y_-,D)\times_{\textbf{Hilb}^{P_-|_D}(D)} \textbf{Hilb}^{P_+}(Y_+,D)  \ar[ll]_{\textbf{r} \quad \quad\quad \quad} \ar[d]^{\textbf{g}:=\iota\circ \bf \tilde{r}_\Delta=\bf r_{\Delta}\circ \textbf{j}} \\
\textbf{Hilb}^{(P_-, P_+)} (X/\bbA^1)_0^\dagger\times W \ar[d] \ar[rr]_{\textbf{h}:=\pi_{\textbf{node}}\times\id_W} \ar@{}[drr]|{\Box}& &  \calA^{P_-} \times \calA^{P_+} \times W \ar[d] \\
\textbf{Hilb}^{(P_-, P_+)} (X/\bbA^1)_0^\dagger \ar[rr]^{\bf \pi_{\textbf{node}} \,\,}& & \calC_0^{\dagger, (P_-, P_+)}\cong \calA^{P_-} \times \calA^{P_+}.
}\end{equation}
Restricting cotangent complexes to the classical truncations and using the fact that $t_0(\textbf{r})$ is an isomorphism and 
$\bbL_{\bf e\circ \bf r}|_{\mathcal{H}}\cong \bbL_{\bf e}|_{\mathcal{H}}$, we obtain the commutativity of the lower 3-by-3 block in \eqref{vir pull dia3}. The 
middle column of \eqref{vir pull dia3} is given by the classical truncation of the cotangent complexes of \eqref{diag on two der}. 
The commutativity of other parts of \eqref{vir pull dia3} is a direct calculation,
similar to the proof of \cite[Prop.~4.22]{CZ}. 
\begin{align}\label{vir pull dia3}
\xymatrix{
\bbD^\vee[2] \ar[d]_{\beta}\ar[r]^{\quad \alpha^\vee[2] }& \bbE_g \ar[r]\ar[d]^{\alpha}& \bbE_e \ar@{=}[d]  \\
(\textbf{e}^*\bbL_{\bf h})|_{\mathcal{H}}\cong e^*\bbE_h \ar[d]_{ }\ar[r]^{\quad \,\,\,\beta^\vee[2] }&\bbL_{\bf g}|_{\mathcal{H}}= \bbD\ar[r]\ar[d]_{ }&\bbL_{\bf e}|_{\mathcal{H}}=:\bbE_e\ar[d]_{ } \\
e^*\bbL_h\ar[r] \ar[d]_{ }  &\bbL_g\ar[r] \ar[d]_{ } &\bbL_e \ar[d]_{ }  \\
\tau^{\geqslant -1}e^*\bfL_h\ar[r]   &\bfL_g \ar[r]  & \bfL'_e, 
}\end{align}
The above diagram gives diagram \eqref{eqn:triang_obst} for maps
$$e:=(\id_{\calH}\times ev_\Delta), \,\,\, h:=(\pi_{\mathrm{node}}\times\id_W), \,\,\, g:=r_{\Delta}. $$
One can check that orientations chosen in Remarks \ref{rmk on ori} \& \ref{rmk on ori2} induce
compatible orientations of $\bbE_g$ and $\bbE_h$ in diagram \eqref{vir pull dia3} as mentioned in \S \ref{defi on sym cx}.

Finally notice that we have 
$$\bar{\sigma}^*\bbL_g\cong \bbL_{g\circ \bar{\sigma}}=\bbL_{r_{\bar{\Delta}}}, \quad  \bar{\sigma}^*\bbL_e\cong \bbL_{e\circ \bar{\sigma}}, \quad \bar{\sigma}^*\bbE_{g}\cong \bbE_{r_{\bar{\Delta}}}, $$
where $\bbE_{r_{\bar{\Delta}}}$ is the symmetric obstruction theory used to define $\sqrt{r_{\bar{\Delta}}^!}$. Here the last isomorphism is due to the following commutative diagram
of derived stacks
\begin{align*} 
\xymatrix{
&  W \times \calA^{P_-} \times \calA^{P_+} \ar[d]_{\Delta\times \id_{\calA^{P_-} \times \calA^{P_+}}} \ar@/^5pc/[dd]^{\bar{\Delta}\times \id_{W\times \calA^{P_-} \times \calA^{P_+}}}\\
\textbf{Hilb}^{P_-}(Y_-,D)  \times \textbf{Hilb}^{P_+}(Y_+,D) \ar[d]^{\bar{\sigma}\times \id} \ar[r]^{ }  &  
W \times W \times \calA^{P_-} \times \calA^{P_+} \ar[d]_{\sigma\times \id_{W\times \calA^{P_-} \times \calA^{P_+}}}   \\
%Z(d\phi)/\!\!/G  \ar[r]^{\Delta \quad \quad } \ar[d]_{ }  & Z(d\phi)/\!\!/G  \times Z(d\phi)/\!\!/G \ar[d]_{ }    \\
\textbf{Hilb}^{P_-}(Y_-,D)  \times \textbf{Hilb}^{P_+}(Y_+,D) \ar[r]^{ } & 
W \times W \times \calA^{P_-} \times \calA^{P_+}.  }
\end{align*}
Consider the pullback of \eqref{vir pull dia3} by the map $\bar{\sigma}$, we obtain the desired diagram \eqref{eqn:triang_obst} for maps 
$(\id_{\calH}\times ev_\Delta)\circ \bar{\sigma}$, $(\pi_{\mathrm{node}}\times\id_W)$, $r_{\bar{\Delta}}$. 
Therefore we are done. 
\end{proof}

\begin{lemma}\label{lem on ynode der cmp}
Let $\textbf{\emph{Hilb}}^{(P_-, P_+)} (X/\bbA^1)_0^\dagger$ and $\textbf{\emph{Hilb}}^{P_-}(Y_-,D)\times_{\textbf{\emph{Hilb}}^{P_-|_D}(D)} \textbf{\emph{Hilb}}^{P_+}(Y_+,D)$ be derived stacks defined by diagrams \eqref{diag on derived enh1}, \eqref{diag on derived enh3} respectively.  Then there is an isomorphism of derived stacks
\begin{equation}\label{equ on map der r}\textbf{r}: \textbf{\emph{Hilb}}^{P_-}(Y_-,D)\times_{\textbf{\emph{Hilb}}^{P_-|_D}(D)} \textbf{\emph{Hilb}}^{P_+}(Y_+,D)
\cong  \textbf{\emph{Hilb}}^{(P_-, P_+)} (X/\bbA^1)_0^\dagger.  \end{equation}
%whose classical truncation is an isomorphism. Moreover, the restriction of the cotangent complex of $\textbf{r}$ to the classical truncation is zero. 
\end{lemma}
\begin{proof}
%can be interpreted as a pushout diagram of classical stacks: 
%\begin{equation}\label{equ on pushout}
%\xymatrix{
%\calA^{P_-} \times D \times \calA^{P_+} \ar[d]_{ } \ar[r]_{ }  \ar@{}[dr]|{\Box}  & \calY_-^{P_-} \times \calA^{P_+} \ar[d]_{ } \\
%\calA^{P_-} \times \calY_+^{\circ, P_+} \ar[r]_{ } & \calX_0^{\dagger, (P_-, P_+)}.
%}\end{equation}
Consider the homotopy pushout diagram of derived stacks: 
\begin{equation}\label{equ on pushout2}
\xymatrix{
\calA^{P_-} \times D \times \calA^{P_+} \ar[d]_{ } \ar[r]_{ }     & \calY_-^{P_-} \times \calA^{P_+} \ar[d]_{ } \\
\calA^{P_-} \times \calY_+^{\circ, P_+} \ar[r]_{ } & \calX_{0,\mathrm{der}}^{\dagger, (P_-, P_+)}, 
}\end{equation}
whose classical truncation recovers the following isomorphism (ref.~Eqn.~\eqref{equ on iso of x0 and ypm}):
\begin{equation*}\calX_0^{\dagger, (P_-, P_+)} \cong (\calY_-^{P_-} \times \calA^{P_+}) \cup_{\calA^{P_-} \times D \times \calA^{P_+} } (\calA^{P_-} \times \calY_+^{\circ, P_+}). \end{equation*}
As the inclusion functor preserves colimits, so it preserves pushout diagrams, and we have  
$$\calX_{0}^{\dagger, (P_-, P_+)}=t_0\left(\calX_{0,\mathrm{der}}^{\dagger, (P_-, P_+)}\right)\cong \calX_{0,\mathrm{der}}^{\dagger, (P_-, P_+)}. $$ 
So we obtain an isomorphism  of derived stacks 
\begin{equation}\label{equ on pushout5}
\textbf{Map}_{\mathrm{dst}/\calA_{P}}\left(\calX_{0,\mathrm{der}}^{\dagger, (P_-, P_+)},\textbf{RPerf}\times \calA_{P}\right)\cong  
\textbf{Map}_{\mathrm{dst}/\calA_{P}}\left(\calX_{0}^{\dagger, (P_-, P_+)},\textbf{RPerf}\times \calA_{P}\right).
\end{equation}
Taking shorthand $\calA_{P}:=\calA^{P_-}\times \calA^{P_+}$ and applying $\textbf{Map}_{\mathrm{dst}/\calA_{P}}\left(-,\textbf{RPerf}\times \calA_{P}\right)$ 
to diagram \eqref{equ on pushout2}, we obtain 
a homotopy pullback diagram
\begin{equation}\label{equ on pushout3}
\xymatrix{
\textbf{Map}_{\mathrm{dst}/\calA_{P}}\left(\calX_{0,\mathrm{der}}^{\dagger, (P_-, P_+)},\textbf{RPerf}\times \calA_{P}\right)
 \ar[d]_{ } \ar[r]_{ }  \ar@{}[dr]|{\Box}  & \textbf{Map}_{\mathrm{dst}/\calA_{P}}\left( \calY_-^{P_-} \times \calA^{P_+} ,\textbf{RPerf}\times \calA_{P}\right) \ar[d]_{ } \\
\textbf{Map}_{\mathrm{dst}/\calA_{P}}\left(\calA^{P_-} \times \calY_+^{\circ, P_+},\textbf{RPerf}\times \calA_{P}\right)
 \ar[r]_{ } &  \textbf{Map}_{\mathrm{dst}/\calA_{P}}\left( D \times \calA_{P} ,\textbf{RPerf}\times \calA_{P}\right).} \end{equation}
For any derived stack $T$ over a base stack $S$, base change implies a canonical equivalence: 
\begin{equation*} \textbf{Map}_{\mathrm{dst}/T}(\bullet_{T},\bullet_{T})\cong \textbf{Map}_{\mathrm{dst}/S}(\bullet, \bullet)\times_ST, \end{equation*}
where $\bullet_T=\bullet\times_ST$. 
%Therefore 
%$$\textbf{Map}_{\mathrm{dst}/\calA_{P}}\left( \calY_-^{P_-} \times \calA^{P_+},\textbf{RPerf}\times \calA_{P}\right)\cong
%\textbf{Map}_{\mathrm{dst}/\calA^{P_-}}\left( \calY_-^{P_-},\textbf{RPerf}\times \calA^{P_-}\right)\times_{\calA^{P_-}} \calA_{P}, $$
%$$\textbf{Map}_{\mathrm{dst}/\calA_{P}}\left(\calA^{P_-} \times \calY_+^{\circ, P_+},\textbf{RPerf}\times \calA_{P}\right)\cong 
%\textbf{Map}_{\mathrm{dst}/\calA^{P_+}}\left(\calY_+^{\circ, P_+},\textbf{RPerf} \times_{} \calA^{P_+} \right)\times_{\calA^{P_+}} \calA_{P},$$
%$$\textbf{Map}_{\mathrm{dst}/\calA_{P}}\left( D \times \calA_{P} ,\textbf{RPerf}\times \calA_{P}\right)\cong 
%\textbf{Map}_{\mathrm{dst}/\pt}\left( D,\textbf{RPerf}\right)\times \calA_{P}. $$
Combined with diagram \eqref{equ on pushout3}, we obtain 
\begin{align}\label{equ on pushout4}
& \quad \,\, \textbf{Map}_{\mathrm{dst}/\calA_{P}}\left(\calX_{0,\mathrm{der}}^{\dagger, (P_-, P_+)},\textbf{RPerf}\times \calA_{P}\right) \\ \nonumber
& \cong \textbf{Map}_{\mathrm{dst}/\calA^{P_-}}\left( \calY_-^{P_-},\textbf{RPerf}\times \calA^{P_-}\right)\times_{\textbf{Map}_{\mathrm{dst}/\pt}\left( D,\textbf{RPerf}\right)} \textbf{Map}_{\mathrm{dst}/\calA^{P_+}}\left(\calY_+^{\circ, P_+},\textbf{RPerf} \times_{} \calA^{P_+} \right).
\end{align}
%Via the inclusion of the classical truncation
%whose classical truncation is an isomorphism. 
In Eqns.~\eqref{equ on pushout5}, \eqref{equ on pushout4}, by fixing the determinant using diagram \eqref{equ on def trless part} and restricting to the open (relative) Hilbert substacks \eqref{equ def derive enh of hilb(X/A)}, \eqref{rest map}, we obtain the desired isomorphism $\textbf{r}$ \eqref{equ on map der r}. 
%The statement on the cotangent complex of $\textbf{r}$ is straightforward to check (e.g.~\cite[diagram in pp.~516]{Zhou1}). 
\end{proof}

Recall Lemma \ref{lem on sect of line bdl} and Proposition \ref{prop on iso of two mod}, we have commutative diagrams
\begin{equation}\label{big diag}
{ 
\xymatrix{ 
\coprod_{(P_-, P_+)} \Hilb^{(P_-, P_+)} (X/\bbA^1)_0^\dagger \ar@{}[dr]|{\Box}    \ar[d]_{\pi_{\mathrm{node}}}  \ar[r]^{}    &  \Hilb^{P_t} (X/\bbA^1) \ar[d]^{\pi} \ar@{}[ddr]|{\Box} & 
\Hilb^{P_r} (X_r)  \ar[l]  \ar[dd] \\
\coprod_{(P_-, P_+)} \calC_0^{\dagger, (P_-, P_+)}  \ar[d] \ar[r]^{\quad \quad \quad \coprod_{(P_-, P_+)} i_{(P_-, P_+)}} & \calC^{P_t} \ar[d] &  \\
  \{0\} \ar[r]^{i_0} & \mathbb{A}^1 & \{r\}. \ar[l]_{i_r \,\, (r\neq 0)}  } 
 }
\end{equation} 
\iffalse
\begin{equation}\label{big diag}
{\Tiny
\xymatrix@C=10pt{ 
\coprod_{(P_-, P_+)}\Hilb^{P_-}(Y_-,D)\times \Hilb^{P_+}(Y_+,D)     & \coprod_{(P_-, P_+)}\Hilb^{P_-}(Y_-,D)\times \Hilb^{P_+}(Y_+,D) \ar[l]_{\bar{\sigma}\times \id } \ar[r]^{\quad \,\,\,\, r_{\calD\to \calY_-}\times r_{\calD\to \calY_+}} \ar@{}[dr]|{\Box}  & \coprod_{(P_-, P_+)} Z(\phi\boxplus \phi)\times \calC_0^{\dagger, (P_-, P_+)} \\
\coprod_{(P_-, P_+)}\Hilb^{P_-}(Y_-,D)\times_{\Delta(W)} \Hilb^{P_+}(Y_+,D) \ar[u]_{i_{\Delta}}   & \coprod_{(P_-, P_+)}\Hilb^{P_-}(Y_-,D)\times_{\bar{\Delta}(W)} \Hilb^{P_+}(Y_+,D) \ar[l]_{\bar{\sigma}} \ar[u]_{i_{\bar{\Delta}}} \ar[r]^{\quad \quad \quad \quad\,\, r_{\bar{\Delta}}} & \coprod_{(P_-, P_+)} W\times \calC_0^{\dagger, (P_-, P_+)} \ar[u]_{\bar{\Delta}}  \\
\coprod_{(P_-, P_+)} \Hilb^{(P_-, P_+)} (X/\bbA^1)_0^\dagger \ar@{}[dr]|{\Box} \ar[u]_{\cong}  \ar[d]_{\pi_{\mathrm{node}}}  \ar[r]^{}    &  \Hilb^{P_t} (X/\bbA^1) \ar[d]^{\pi} \ar@{}[ddr]|{\Box} & 
\Hilb^{P_r} (X_r)  \ar[l]  \ar[dd] \\
\coprod_{(P_-, P_+)} \calC_0^{\dagger, (P_-, P_+)}  \ar[d] \ar[r]^{\coprod_{(P_-, P_+)} i_{(P_-, P_+)}} & \calC^{P_t} \ar[d] &  \\
  \{0\} \ar[r]^{i_0} & \mathbb{A}^1 & \{r\}. \ar[l]_{i_r \,\, (r\neq 0)}  } 
 }
\end{equation}
\fi 
Similar to $\Phi^{P_\pm}_{Y_\pm,D}$ \eqref{equ of rel inv}, under Assumption \ref{ass on lag fib}, we define 
\begin{equation}\label{tensor of rel inv}\Phi^{P_-}_{Y_-,D}\otimes\Phi^{P_+}_{Y_+,D}:
A_*(Z(\phi\boxplus \phi))\otimes A_{*}(\calA^{P_-}\times \calA^{P_+}) \to A_*\left(\Hilb^{P_-}(Y_-,D)\times \Hilb^{P_+}(Y_+,D) \right). \end{equation}
%Here although the notation is in the product form, the map is not necessarily the tensor product of two maps in general (unless under restriction to the image of
%$A_*(Z(\phi)\times Z(\phi))\to A_*(Z(\phi\boxplus \phi))$). 
Recall also that the base change of \eqref{tensor of rel inv} defines $\sqrt{r_{\bar{\Delta}}^!}$ in the above diagram.  

Now we are ready to state our degeneration formula in the cycle version. In below, we denote $[\Hilb^{P_r} (X_r)]^{\vir}$ to be the virtual class of $\Hilb^{P_r} (X_r)$ for smooth Calabi-Yau 
4-fold $X_r$ ($r\neq 0$) with choice of orientation given by the restriction of the orientation in Remark \ref{rmk on ori} to the fiber of $\pi$.
\begin{theorem}\label{thm on dege formula}
Let $X\to \bbA^1$ be given in Setting \ref{set of dege for} and $P_t\in K_{c,\leqslant 1}^\mathrm{{num}} (X_t)$ for all $t\in \mathbb{A}^1$.
Let $\Phi^{P_t}_{X/\bbA^1}$ be the family invariant  \eqref{equ of fam inv} and $\Phi^{P_-}_{Y_-,D}\otimes\Phi^{P_+}_{Y_+,D}$ be defined as \eqref{tensor of rel inv} under Assumption \ref{ass on lag fib}. Then  
$$i_r^!\left(\Phi^{P_t}_{X/\bbA^1}[\calC^{P_t}]\right)=[\Hilb^{P_r} (X_r)]^{\vir}, \quad \forall\,\,r\in \mathbb{A}^1\backslash \{0\},$$
$$i_{0}^!\left(\Phi^{P_t}_{X/\bbA^1}[\calC^{P_t}]\right)= n_*\,\sum_{(P_-, P_+)\in \Lambda^{P_0}_{spl}}i_{(P_-, P_+)}^!\left(\Phi^{P_t}_{X/\bbA^1}[\calC^{P_t}]\right), $$
where $n$ is the normalization map in \eqref{diag on cpr two maps0}. 
For each $(P_-, P_+)$, we have 
$$(\bar{\sigma})^*i_{(P_-, P_+)}^!\left(\Phi^{P_t}_{X/\bbA^1}[\calC^{P_t}]\right)=
 \sqrt{r_{\bar{\Delta}}^!}\,\left([W]\times [\calC_0^{\dagger, (P_-, P_+)}] \right), $$
where $\bar{\sigma}$ is the automorphism in Data \ref{data on auto}.
Moreover, we have  
$$i_{\bar{\Delta}*}(\bar{\sigma})^*i_{(P_-, P_+)}^!\left(\Phi^{P_t}_{X/\bbA^1}[\calC^{P_t}]\right)=
 \Phi^{P_-}_{Y_-,D}\otimes\Phi^{P_+}_{Y_+,D} \,\left(\bar{\Delta}_*[W]\times [\calC_0^{\dagger, (P_-, P_+)}] \right). $$
Moreover, in the presence of a torus action,  the above equalities hold in equivariant Chow groups. 
%Moreover the above equalities hold in equivariant Chow groups when there are torus actions. 
\end{theorem}
\begin{proof}
The first equality is by the base change property of square root virtual pullbacks \cite[Prop.~1.15,~Def.~A.3]{Park1}.  
Similar base change implies that  
\begin{align}\label{equ on ipp}i_{(P_-, P_+)}^!\left(\Phi^{P_t}_{X/\bbA^1}[\calC^{P_t}]\right) 
%\sum_{(P_-, P_+)\in \Lambda^{P_0}_{spl}}(\bar{\sigma})^*\sqrt{\pi_{\mathrm{node}}^!}\,[\calC_0^{\dagger, (P_-, P_+)}]  \\
&\overset{\eqref{equ of fam inv}}{=}  i_{(P_-, P_+)}^!\sqrt{\pi^!}\,[\calC^{P_t}]  \\  \nonumber
&\overset{\eqref{diag on derived enh1}}{=}  \sqrt{\pi_{\mathrm{node}}^!}\,i_{(P_-, P_+)}^![\calC^{P_t}] \\ \nonumber
&\overset{\mathrm{Lem}.\,\ref{lem on sect of line bdl}}{=}  \sqrt{\pi_{\mathrm{node}}^!}\,[\calC_0^{\dagger, (P_-, P_+)}] \\ \nonumber
&\overset{\eqref{diag on cpr two maps}}{=}(\id_{\calH}\times ev_\Delta)^!p^*_{\calH}\sqrt{\pi_{\mathrm{node}}^!}\,[\calC_0^{\dagger, (P_-, P_+)}]  \\ \nonumber
&\overset{\eqref{diag on cpr two maps}}{=}(\id_{\calH}\times ev_\Delta)^!\sqrt{(\pi_{\mathrm{node}}\times\id_W)^!}p^*_{W}\,[\calC_0^{\dagger, (P_-, P_+)}]  \\ \nonumber
&\overset{\mathrm{Prop}.\,\ref{prop on gluing form}}{=}\left(\bar{\sigma}^*\right)^{-1}\sqrt{r_{\bar{\Delta}}^!}\,\left([W]\times [\calC_0^{\dagger, (P_-, P_+)}] \right). 
\end{align}
Recall that  $\sqrt{r_{\bar{\Delta}}^!}$ is defined by the base change of \eqref{tensor of rel inv} via anti-diagonal embedding, therefore
%Applying $i_{\bar{\Delta}*}\bar{\sigma}^*$ to above, we obtain 
\begin{equation*}i_{\bar{\Delta}*}\bar{\sigma}^*i_{(P_-, P_+)}^!\left(\Phi^{P_t}_{X/\bbA^1}[\calC^{P_t}]\right) =\Phi^{P_-}_{Y_-,D}\otimes\Phi^{P_+}_{Y_+,D} \,\left(\bar{\Delta}_*[W]\times [\calC_0^{\dagger, (P_-, P_+)}] \right). \end{equation*}
Finally, we have 
\begin{align*}i_{0}^!\left(\Phi^{P_t}_{X/\bbA^1}[\calC^{P_t}]\right)
&\overset{\eqref{equ of fam inv}}{=}i_{0}^!\sqrt{\pi^!}[\calC^{P_t}]\overset{\eqref{diag on cpr two maps0}}{=}\sqrt{\pi_{\mathrm{node}}^{'!}}\,i_{0}^!\,[\calC^{P_t}] \\
&\overset{\eqref{diag on cpr two maps0}}{=}\sqrt{\pi_{\mathrm{node}}^{'!}}\,[\calC_0^{P_0}]  \overset{\eqref{diag on cpr two maps0}}{=}n_*\,\sqrt{\pi_{\mathrm{node}}^{!}}\,\sum_{(P_-, P_+)\in \Lambda^{P_0}_{spl}}[\calC_0^{\dagger, (P_-, P_+)}]  
\\
&\overset{\eqref{equ on ipp}}{=} n_*\,\sum_{(P_-, P_+)\in \Lambda^{P_0}_{spl}}i_{(P_-, P_+)}^!\left(\Phi^{P_t}_{X/\bbA^1}[\calC^{P_t}]\right). \qedhere
\end{align*}
\end{proof}
In the special case when $W$ in Assumption \ref{ass on lag fib} is a point, inclusions $i_{\bar{\Delta}}$, $\bar{\Delta}$ and automorphism $\bar{\sigma}$ in Theorem \ref{thm on dege formula} are identities, therefore: 
\begin{corollary} \label{cor on dege for on pot zero}
Let $X\to \bbA^1$ be given in Setting \ref{set of dege for} and $P_t\in K_{c,\leqslant 1}^\mathrm{{num}} (X_t)$ for all $t\in \mathbb{A}^1$.
Assume for any splitting datum $(P_-,P_+)$ of $P_0$, we have $\textbf{\emph{Hilb}}^{P_-|_D}(D)=\Spec \mathbb{C}$. Then we have 
$$i_r^!\left(\Phi^{P_t}_{X/\bbA^1}[\calC^{P_t}]\right)=[\Hilb^{P_r} (X_r)]^{\vir}, \quad \forall\,\,r\in \mathbb{A}^1\backslash \{0\},$$
$$i_0^!\left(\Phi^{P_t}_{X/\bbA^1}[\calC^{P_t}]\right)=n_*
\sum_{(P_-, P_+)\in \Lambda^{P_0}_{spl}} \Phi^{P_-}_{Y_-,D}\otimes\Phi^{P_+}_{Y_+,D} \,[\calC_0^{\dagger, (P_-, P_+)}]. $$
Moreover, in the presence of a torus action,  the above equalities hold in equivariant Chow groups. 
\end{corollary}
%\begin{remark}
When we consider (relative) Hilbert stacks of \textit{points} on $X/\mathbb{A}^1$ and $(Y_{\pm},D)$, the above assumption holds. In particular, using \eqref{equ of rel inv}, we define 
\begin{equation}\label{vir class of rel hilb pts}[\Hilb^{n_{\pm}}(Y_{\pm},D)]^{\vir}:=\Phi^{n_\pm}_{Y_{\pm},D}[\calA^{n_\pm}]\in A_{n_{\pm}}(\Hilb^{n_{\pm}}(Y_{\pm},D)). \end{equation} 
In the following section, we will use the above degeneration formula (and a variant of it) to calculate \textit{tautological integrals} on $[\Hilb^{P_r} (X_r)]^{\vir}$ ($r\neq 0$) when $X_r$ is $\mathbb{C}^4$ or local curve $\Tot_C (L_1 \oplus L_2 \oplus L_3)$. 
%\end{remark}

\section{Applications to Hilbert schemes of points  } \label{sect on hilb}

In this section, we study equivariant $\DT_4$ theory for Hilbert schemes of points on (log) Calabi-Yau 4-folds.  
As an application of our degeneration formula, we compute tautological integrals on virtual classes $[\Hilb^n(\mathbb{C}^4)]^{\vir}$ and 
prove the conjectural formula proposed in \cite{CK1}. 
%Combining with the vertex formalism \cite{CK1, NP}, this implies the conjecture in~\textit{loc}.~\textit{cit.} for any toric Calabi-Yau 4-fold.
We also compute zero dimensional  (relative) $\DT_4$ invariants for all local curves.

%{\color{red}more background...}

\subsection{Virtual classes of $\Hilb^n(\mathbb{C}^4)$}

Consider $\bbC^4$ with torus action by $(\bbC^*)^4$:  
$$t\cdot (x_1,x_2,x_3,x_4)=(t_1\cdot x_1,t_2\cdot x_2,t_3\cdot x_3,t_4\cdot x_4), $$
%$$t\cdot (x_1,x_2,x_3,x_4)=(t_1^{-1}\cdot x_1,t_2^{-1}\cdot x_2,t_3^{-1}\cdot x_3,t_4^{-1}\cdot x_4), $$
and also the 3-dimensional CY subtorus 
\begin{equation}\label{equ on cy subtor}T:=\{ (t_1, t_2, t_3, t_4) \mid t_1 t_2 t_3 t_4 =1 \} \subset (\bbC^*)^4. \end{equation}
The ring of equivariant parameters is
$$
H^*_T(\pt) = \bbC [s_1, s_2, s_3] \cong \frac{\bbC [s_1, s_2, s_3, s_4]}{(s_1 + s_2 + s_3 + s_4)}. 
$$
As the $T$-fixed locus of $\Hilb^n (\bbC^4)$ is proper \cite[Lem.~3.1]{CK1}, Hilbert scheme $\Hilb^n (\bbC^4)$ of $n$ points admits a $T$-equivariant virtual class \cite{OT}:
$$[\Hilb^n (\bbC^4)]_T^\vir\in A^T_n(\Hilb^n (\bbC^4)), $$ 
which depends on the choice of orientation \cite{CGJ}.
As $\Hilb^n (\bbC^4)$ is connected \cite[Prop.~2.3]{Fo}, reversing an orientation affects the virtual class by a minus sign. 

%There has been a vertex formalism developed for its calculations \cite{CK1, CK2, CKM1, NP}. 
To fix the \textit{orientation} of virtual classes, we fix a presentation of $\C^4$ (ref.~Theorem~\ref{ori of cy4 family}) as 
\begin{equation}\label{ori on hilbc4}\C^4=\Tot\left(\omega_{\C^3_{x_1,x_2,x_3}}\right), \end{equation}
and then use convention as specified by Remark \ref{rmk on ori}. Here $\C^3_{x_1,x_2,x_3}$ denotes the first three coordinate planes of $\C^4_{x_1,x_2,x_3,x_4}$.

\subsection{Tautological insertions}

Let $L\in \Pic^{(\C^*)^4}(\C^4)$ be a $(\C^*)^4$-equivariant line bundle on $\C^4$. The tautological rank $n$ bundle of $L$ is 
\begin{equation}\label{tau bdl}L^{[n]}:=\pi_{M*}(\pi_X^*L\otimes \oO_\mathcal{Z}), \end{equation}
where $\pi_M: \Hilb^n(\C^4)\times \C^4\to \Hilb^n(\C^4)$, $\pi_X: \Hilb^n(\C^4)\times \C^4\to \C^4$ are projections and $\mathcal{Z}\hookrightarrow \Hilb^n(\C^4)\times \C^4$
is the universal subscheme. 

As in \cite{CKM1}, we take an extra 1-dimensional torus $\bbC^*_m$ acting trivially on $\C^4$ (and also on $\Hilb^n(\C^4)$), and let 
$$
L_m:=\calO_{\C^4}\otimes e^m, \qquad L_m^{[n]}\cong \oO^{[n]}\otimes e^m
$$
denote the trivial line bundle on $\C^4$ with $\bbC^*_m$-equivariant weight $m$, and its tautological bundle.

\begin{definition}\label{defi of dim zero dt4}
Choose orientation on $\Hilb^n(\C^4)$ ($n\geqslant 1$) as specified by Eqn.~\eqref{ori on hilbc4} and Remark \ref{rmk on ori}. 
We define the generating function of $(T \times \bbC^*_m)$-equivariant  zero dimensional $\DT_4$ invariants \emph{with tautological insertions}:
$$
 Z_{} (\bbC^4) :=1+\sum_{n=1}^\infty q^n \int_{[\Hilb^n (\bbC^4)]_T^\vir} e_{T \times \bbC^*_m }(L_m^{[n]}), 
$$
which lies in $(\Frac H^*_{T} (\pt) \otimes H^*_{\bbC^*_m} (\pt) )  [\![q]\!] \cong \bbC (s_1, s_2, s_3 )[m] [\![ q ]\!]$, where $q$ is the degree-counting parameter. 
\end{definition}
\begin{remark}\label{rmk on dg par}
%There is no need to invert $m$ in above, 
As $\bbC^*_m$ acts trivially on $\bbC^4$, virtual normal bundles appearing in localization formulae of $\Hilb^n (\bbC^4)$ do not involve $m$, hence coefficients of  
above generating series are polynomials in $m$. 

%By virtual dimension counting,  
If we assign parameters with gradings as follows
$$
\deg s_i = \deg m = 1, \,\,\, 1\leq i\leq 3; \quad \deg q =0,
$$
the generating function $ Z(\bbC^4)$ will be homogeneous of degree $0$. 
\end{remark}

\subsection{Relative and rubber geometry}\label{Sec-rel-rub}

In the approach of \cite{MNOP2}, zero dimensional $\DT_3$ invariants are computed via the relative $\DT_3$ invariants of $( \bbP^1 \times \bbC^2,  \{\infty\} \times \bbC^2)$.
The naive 4-dimensional analogue $( \bbP^1 \times \bbC^3,  \{\infty\} \times \bbC^3)$ does not have a well-defined $\DT_4$ analogy, as it fails to be (log) $\mathrm{CY}_4$. Instead,  
using constructions in \S \ref{sect on rel invs} and their rubber analogue, we study $T$-equivariant versions of the following invariants (more basics about rubber theory are 
recalled in Appendix \ref{sec on rub}).

\begin{enumerate}

\vspace{1ex}

\item $\DT_4$ invariants of $(X, D_\infty)$, 
where 
$$
X = X_{-1, 0, 0} := \Tot (\calO_{\bbP^1} (-1) \oplus \calO_{\bbP^1} \oplus \calO_{\bbP^1}), \quad
D_\infty := \{\infty\} \times \bbC^3, 
$$ 
and $(\bbC^*)^4$ acts on $X$ with the following tangent weights at fixed points:
$$
\text{at } 0: \,\,\, -s_1, -s_2, -s_3, -s_4; \quad 
\text{at } \infty:  \,\,\,
s_1, -(s_2 + s_1), -s_3, -s_4. 
$$
%$$\text{at } 0: \,\,\, s_1, s_2, s_3, s_4; \quad \text{at } \infty:  \,\,\,-s_1, s_2 + s_1, s_3, s_4. $$
We take a 3-dimensional subtorus $T\subset (\bbC^*)^4$ as follows. 
Recall $X$ can be constructed as  
$ (\bbC^2 \backslash \{0\} ) \times_{\bbC^*} \bbC^3$, where $\bbC^*$ acts on $\bbC^3$ by weights $(-1, 0, 0)$. 
Then we take\,\footnote{This is the same as the choice of torus in Remark \ref{rmk on toru on pair}.}
 $$T = (\bbC^*)^2 \times_{\bbC^*} (\bbC^*)^2, $$ 
where the 2nd $(\bbC^*)^2 \subset (\bbC^*)^3$ is the CY torus on $\bbC^3$.
In other words, $T$ is a twisted product of the $\bbC^*$ on $\bbP^1$ and the CY torus on the fiber $\bbC^3$. 
Equivariant parameters for $T$ are given by the condition 
$$s_1 + s_2 + s_3 + s_4 = 0. $$

\vspace{1ex}

\item $\DT_4$ invariants of $(\Delta, D_0 \sqcup D_\infty)$, where
$$
\Delta = X_{0,0, 0} := \bbP^1 \times \bbC^3, \quad 
D_0 := \{0\} \times \bbC^3, 
\quad 
D_\infty := \{\infty\} \times \bbC^3, 
$$
and $(\bbC^*)^4$ acts on $\Delta$ with tangent weights $\mp s_1, -s_2, -s_3, -s_4$ at $0$ and $\infty$. 
As before, we take $$T =\bbC^* \times (\bbC^*)^2, $$ where the 2nd factor is the CY torus on $\C^3$. 
The equivariant parameters of CY torus $T$ are given by the equation $$s_2 + s_3 + s_4 = 0. $$

% (3) is the bubble component of (2). The bubble component of (1) is described in \eqref{Finfit}.

\vspace{1ex}

\item \emph{Rubber invariants} of $(\Delta, D_0 \sqcup D_\infty)$, where the $\bbC^*$-action on $\bbP^1$ is considered as automorphism in the relative Hilbert stack.
In this case, $$T = (\bbC^*)^2$$ is the CY torus on the fiber $\C^3$, whose equivariant parameters satisfy $$s_2 + s_3 + s_4 = 0. $$
\end{enumerate}
To fix the \textit{orientation} of virtual classes, we need to fix the presentation of the log Calabi-Yau pair (ref.~Theorem~\ref{relative ori of cy4 family}\footnote{For ideal sheaves of points, relative orientations are equivalent to absolute orientations, as the image of the map $r_{\calD\to \calY}$ in Definition \ref{def of rel or} is $\Spec\C\times \calA$.}) as 
\begin{equation}\label{equ on pre as log} \left(\Tot(\omega_{U}(S)),\Tot(\omega_S)\right), \end{equation}
and then use convention as specified by Remark \ref{rmk on ori2}. In the above three cases, we fix 
\begin{equation}\label{equ on pre as log cy4}
(1):  U= \Tot (\calO_{\bbP^1} (-1) \oplus \calO_{\bbP^1}),\,\, S= \{\infty\} \times \bbC^2_{x_2,x_3}; \,\,\, (2)\,\&\,(3): U= \bbP^1 \times \bbC^2_{x_2,x_3},\,\, S= \{0,\infty\} \times \bbC^2_{x_2,x_3},\end{equation}
where $\C^2_{x_2,x_3}$ denotes the first two coordinate planes of $\C^3_{x_2,x_3,x_4}$. 

As in Definition \ref{defi of dim zero dt4}, we define relative invariants in above geometries. 
\begin{definition}\label{defi of relat inv}
%We still consider a tautological insertion of the form $e_{T \times \bbC^*_m} (L^{[n]})$, with $L_m = \calO \otimes e^m$. 
Choose orientation as specified by Eqn.~\eqref{equ on pre as log cy4} and Remark \ref{rmk on ori2}. 
We define generating functions of those \textit{relative} $\DT_4$ \textit{invariants} with \textit{tautological insertions} by: 
%\yl{write down definition, explain what are tautological bundles in those cases}
\begin{enumerate}

\vspace{1ex}

\item $(X, D_\infty)$:
$$Z (X, D_\infty):=1+\sum_{n=1}^\infty q^n \int_{[\Hilb^n (X, D_\infty)]_T^\vir} e_{T \times \bbC^*_m }(L_m^{[n]}), $$  
where $[\Hilb^n (X, D_\infty)]_T^\vir$ is the $T$-equivariant virtual class \eqref{vir class of rel hilb pts}. 

\item $(\Delta, D_0 \sqcup D_\infty)$:
$$Z (\Delta, D_0 \sqcup D_\infty):=1+\sum_{n=1}^\infty q^n \int_{[\Hilb^n (\Delta, D_0 \sqcup D_\infty)]_T^\vir} e_{T \times \bbC^*_m }(L_m^{[n]}), $$
where $[\Hilb^n (\Delta, D_0 \sqcup D_\infty)]_T^\vir$ is the $T$-equivariant virtual class \eqref{vir class of rel hilb pts}.

\item \emph{Rubber theory} for $(\Delta, D_0 \sqcup D_\infty)$:
$$Z^\sim (\Delta, D_0 \sqcup D_\infty):=1+\sum_{n=1}^\infty q^n \int_{[\Hilb^{\sim, n} (\Delta ,D_0 \sqcup D_\infty )]_T^\vir} e_{T \times \bbC^*_m }(L_m^{[n]}), $$ 
where $[\Hilb^{\sim, n} (\Delta ,D_0 \sqcup D_\infty )]_T^\vir$ is the rubber $T$-equivariant virtual class (Definition \ref{def of rubber pair}).  
\end{enumerate}
%Here $L_m = \calO \otimes e^m$ is a trivial line bundle with $\C^*_m$ weight $m$ and $L_m^{[n]}=\oO^{[n]}\otimes e^m$ is its tautological bundle \eqref{tau bdl}.  
\end{definition}
We follow\cite{MNOP2} to compute these invariants. 
\begin{lemma} \label{Lemma-poles}
The $q$-coefficients of $Z (X, D_\infty)$ only have poles in monomials of $s_2, s_1 + s_2, s_3, s_1+s_2 + s_3$. 
\end{lemma}

\begin{proof}
There is a projective map 
$$
\pi : X = \Tot (\calO_{\bbP^1} (-1) ) \times \bbC^2 \to \bbC^2 \times \bbC^2, 
$$
where $\Tot (\calO_{\bbP^1} (-1) ) \to \bbC^2$ is the blowup of $\bbC^2$ at the origin. 
With our convention of $T$-action, $\pi$ is $T$-equivariant if $T$ acts on $\bbC^2 \times \bbC^2$ with tangent weights $-s_2, -(s_1+s_2), -s_3, -s_4 = (s_1 + s_2 + s_3)$. 
%\yl{sign for weights}

The rest of the argument is the same as \cite[Lem.~3]{MNOP2}.
The Hilbert--Chow morphism, together with the contraction of bubbles, yields a proper and $T$-equivariant map 
$$
\Hilb^n (X, D_\infty) \to \Sym^n X \to \Sym^n \bbC^4 \to \bigoplus_{i=1}^n \bbC^4, 
$$
where the last map is 
$$\big\{(x_i, y_i, z_i, w_i) \big\} \mapsto \left(\sum_i x_i, \sum_i y_i, \sum_i z_i, \sum_i w_i \right) \oplus \cdots \oplus \left(\sum_i x_i^n, \sum_i y_i^n, \sum_i z_i^n, \sum_i w_i^n\right). $$
The lemma then follows by pushing forward the virtual class to $\bigoplus_{i=1}^n \bbC^4$. 
\end{proof}

\begin{lemma}\label{lem-rel-loc}Let $\psi_0 = \psi_0 (\calA^{\sim})$ is the cotangent line bundle class introduced in Defintion \ref{Defn-psi}.
\begin{equation} \label{eqn-rel-loc}
 Z (X, D_\infty) =  Z (\bbC^4) \cdot W_{\infty}, 
\end{equation}
where
\begin{equation}\label{w func}W_{\infty} := 1 + \sum_{n=1}^\infty q^n \int_{[\Hilb^{\sim, n} (\Delta, D_0 \sqcup D_\infty)]_T^\vir} 
\frac{e_{T \times \bbC^*_m }(L_m^{[n]})}{s_1 - \psi_0 }. \end{equation}
\end{lemma}

\begin{proof}
This is standard  (cf.~\cite{GV}). 
Consider the subtorus
$$
\bbC^* = \{(t_1, t_1^{-1}, 1, 1) \} \subset T. 
$$
By $\bbC^*$-localization, we have
$$\displaystyle \Hilb^n (X, D_\infty)^{\bbC^*} = \coprod_{n_- + n_+ = n} \left( \Hilb^{n_-} (\bbC^4) \times \Hilb^{\sim, n_+} (\Delta, D_0 \sqcup D_\infty ) \right)^{\bbC^*}. 
$$
The result then follows from a comparison of obstruction theories, where the denominator $s_1 - \psi_0$ comes from splitting the node. 
\end{proof}

For $l\geq 0$, we introduce
\begin{align}\label{Finfit}
F_{\infty, l }:=& \sum_{n=1}^\infty q^n \int_{[\Hilb^{\sim, n} (\Delta, D_0 \sqcup D_\infty)]_T^\vir} e_{T \times \bbC^*_m }(L_m^{[n]}) \cdot \psi_0^l  \\ \nonumber 
& \in \Frac \left( \frac{\bbC [s_1 + s_2 , s_3, s_4] }{ (s_1 + s_2 + s_3 + s_4) } \right) \otimes H^*_{\bbC^*_m} (\pt) [\![q]\!] \cong \bbC (s_1 + s_2, s_3) [m] [\![q]\!] . 
\end{align}
Note that here for the rubber theory, the bubble components $\Delta$ come from the chart of $X$ containing $D_\infty$, where the torus on the fiber has equivariant weights 
$-(s_1 + s_2), -s_3, -s_4$. 
%$s_1 + s_2, s_3, s_4$. 

\begin{lemma} \label{Lemma-s1}
$\log   W_{\infty} = \dfrac{1}{s_1}  F_{\infty, 0}$. 
\end{lemma}
\begin{proof}
This is similar to \cite[Lem.~4]{MNOP2}. 
Consider the universal substack
$$\pi_{\calZ^\sim}: \calZ^{\sim,n} \to \Hilb^{\sim, n} (\Delta, D_0 \sqcup D_\infty), $$
which is flat and finite of degree $n$ over $\Hilb^{\sim, n} (\Delta, D_0 \sqcup D_\infty)$. 
We define $$[\calZ^{\sim, n}]^\vir := \pi_{\calZ^\sim}^* [\Hilb^{\sim, n} (\Delta, D_0 \sqcup D_\infty) ]_T^\vir. $$ 
For $l\geq 1$, the projection formula implies that 
\begin{align*}
q \frac{d}{dq} F_{\infty, l}&= \sum_{n=1}^\infty q^n \int_{ [\calZ^{\sim, n}]^\vir}\pi_{\calZ^\sim}^*e_{T \times \bbC^*_m }(L_m^{[n]}) \cdot \pi_{\calZ^\sim}^* \psi_0^l \\ 
&= \sum_{n=1}^\infty q^n \int_{ [\calZ^{\sim, n}]^\vir}\pi_{\calZ^\sim}^*e_{T \times \bbC^*_m }(L_m^{[n]}) \cdot \pi^* \psi_0^l \\ 
&= \sum_{n=1}^\infty q^n \int_{[\calZ^{\sim, n}]^\vir} \pi_{\calZ^\sim}^*e_{T \times \bbC^*_m }(L_m^{[n]}) \cdot \pi^* \psi_0^{l-1} \cdot \psi_0 (\calR^\sim), 
\end{align*}
where $\psi_0$ is the tautological psi-class on $\Hilb^{\sim, n} (\Delta, D_0 \sqcup D_\infty)$, pulled back from $\psi_0(\calA^\sim)$, and $\psi_0 (\calR^\sim)$ is the tautological psi-class on the universal family $\pi: \calR^\sim \to \calA^\sim$ (see Remark \ref{rk-2nd-moduli}). 
The last identity follows from  
$$
\psi_0 (\calR^{\sim, n}) \cap [\calZ^{\sim, n}]^\vir = \pi^* \psi_0 (\calA^{\sim, n}) \cap [\calZ^{\sim, n}]^\vir, 
$$
which holds in the zero-dimensional case, since base changed over $\calZ^\sim$, the universal section $\Sigma_1$ is always disjoint from the relative divisor $D_0$.

%Now we can apply Proposition \ref{prop decom rubber} (1) and proceed 
For any $l\geq 1$, we have 
\begin{align}\label{diff equ}
&\,\, q \frac{d}{dq} F_{\infty, l} \\ \nonumber 
=& \sum_{n=1}^\infty q^n \int_{ [\calZ^{\sim, n}]^\vir} \pi_{\calZ^\sim}^*e_{T \times \bbC^*_m }(L_m^{[n]}) \cdot  \pi^* \psi_0^{l-1} \cdot  \psi_0 (\calR^\sim) \\ \nonumber
=& \sum_{n=1}^\infty q^n \sum_{\begin{subarray}{c}n_- + n_+ = n \\ n_-,n_+>0   \end{subarray} } \int_{ [\Hilb^{\sim, n_-} (\Delta, D_0 \sqcup D_\infty) ]_T^\vir } e_{T \times \bbC^*_m }(L_m^{[n_-]}) \cdot \psi_0^{l-1} \cdot \int_{ [\calZ^{\sim, n_+} ]^\vir}\pi_{\calZ^\sim}^*e_{T \times \bbC^*_m }(L_m^{[n_+]}) \\ \nonumber
=& F_{\infty, l-1} \cdot q \frac{d}{dq} F_{\infty, 0}.
\end{align}
In the 2nd equality, we use the decomposition:
$$\calZ^{\sim, n} \times_{\calR^{\sim, n}} D(0 |1, \infty)^{n_-,n_+} \cong \Hilb^{\sim, n_-} (\Delta, D_0 \sqcup D_\infty) \times \calZ^{\sim, n_+}, $$ 
where a similar argument as in Corollary \ref{cor on dege for on pot zero} shows that the virtual classes coincide.
We also use the splitting (according to the above isomorphism):
$$L_m^{[n]} \cong L_m^{[n_-]} \boxplus L_m^{ [n_+]}, \quad \mathrm{for} \,\,\, L_m:= \calO \otimes e^m. $$  
Note that $\psi_0^{l-1}$ is distributed to the left integral as the marked point  $0$  is given from $\Hilb^{\sim, n_-}(\Delta, D_0 \sqcup D_\infty)$, 
and $n_-,n_+$ in above are positive by Remark \ref{rm on nonepty moduli}. 

Solving \eqref{diff equ}, we obtain 
$$F_{\infty, l} =\frac{F_{\infty, 0}^{l+1}}{(l+1)!}. $$
By expanding the RHS of \eqref{w func} and use the above equation, we obtain  
\begin{equation*}W_{\infty}=1 + \sum_{l=0}^\infty \frac{F_{\infty, l}}{s_1^{l+1}}=1 + \sum_{l=0}^\infty\frac{F_{\infty, 0}^{l+1}}{(l+1)!s_1^{l+1}}=\exp\left(\frac{F_{\infty, 0}}{s_1}\right).
  \qedhere \end{equation*}
\end{proof}

\subsection{Zero dimensional $\DT_4$ invariants of $\bbC^4$}

\begin{lemma} \label{Lemma divisible L}
For $n\geq 1$, $\displaystyle \int_{[\Hilb^n (\bbC^4)]_T^\vir} e_{T \times \bbC^*_m} (L_m^{[n]})$ is divisible by $m$. 
\end{lemma}

\begin{proof}
By the vertex formalism \cite[Prop.~3.17]{CK1}, 
$$\displaystyle \int_{[\Hilb^n (\bbC^4)]_T^\vir}e_{T \times \bbC^*_m} (L_m^{[n]})$$ can be written as a summation over solid partitions $\pi$ of size $n$. 
Let $\pi = \{\pi_{ijk}\}_{i, j, k\geq 1}$ be a solid partition (\cite[Def.~3.3]{CK1}), and $Z$ be the $T$-fixed point in $\Hilb^n (\bbC^4)$ corresponding to $\pi$. 
Then the contribution from the tautological insertion is 
$$e_{T \times \bbC^*_m} (L_m^{[n]}|_Z) =e_{T \times \bbC^*_m} (H^0 (\bbC^4, \oO_Z)\otimes e^m), $$ 
where
$$
H^0 (\bbC^4, \oO_Z)\otimes e^m= \sum_{i, j, k\geq 1} \sum_{l=1}^{\pi_{ijk} }  e^m t_1^{ i-1} t_2^{j-1} t_3^{k-1} t_4^{l-1}. 
$$
In particular, if $n\geq 1$, the solid partition $\pi$ must contain the box with $(i,j,k,l) = (1,1,1,1)$. 
Therefore, $H^0 (\bbC^4, \oO_Z)\otimes e^m=e^m+\cdots$, which contributes to a factor $m$ in $e_{T \times \bbC^*_m} (L_m^{[n]}|_Z)$.
\end{proof}

%\yl{depend on how to choose action weight $(1,1,1,1)$ or $(-1,-1,-1,-1)$}
\begin{theorem} \label{thm on ck conj}
With the canonical choice of orientation (in Definition \ref{defi of dim zero dt4}),  
$$
 Z (\bbC^4) = M(q)^{\int_{\C^4} m\,c^T_3(\C^4)} = M(q)^{\frac{-m (s_2 s_3 s_4 + s_1 s_3 s_4 + s_1 s_2 s_4 + s_1 s_2 s_3)}{s_1 s_2 s_3 s_4} },
$$
where 
$$M(q):=\prod_{n\geqslant 1}\frac{1}{(1-q^n)^n}$$
is MacMahon function and $\int_{\C^4}$ denotes equivariant push-forward to point. 

As a consquence, for any $L\in \Pic^{(\C^*)^4}(\C^4)$, we have 
$$
1+ \sum_{n=1}^\infty q^n \int_{[\Hilb^n (\bbC^4)]_T^\vir}e_T(L^{[n]}) = M(q)^{\int_{\C^4} c^T_3(\C^4)\cdot c^T_1(L)}, 
$$
~i.e.~\cite[Conj.~1.6]{CK1} holds on $\C^4$. 
\end{theorem}

\begin{proof}
%We first work with tautological insertion for $L = \calO\otimes e^m $. 
Identity (\ref{eqn-rel-loc}) implies that
$$
\log   Z (\bbC^4) = \log   Z (X, D_\infty) - \log   W_{\infty}. 
$$
By Lemma  \ref{Lemma-poles}, $\log Z (X, D_\infty)$ only has poles in monomials of $s_2, s_1 + s_2, s_3, s_1+s_2 + s_3$.
By Lemma \ref{Lemma-s1} and Eqn.~\eqref{Finfit}, 
$$\log W_{\infty}=\dfrac{1}{s_1}  F_{\infty, 0} $$ 
 has a simple pole at $s_1$ and coefficients of $F_{\infty, 0}$ are rational functions in $s_1+s_2, s_3$.
%only has poles  in $s_1+s_2$, $s_3$ and $s_1$ (the order of pole at $s_1$ is exactly one). 

Combining with Lemma \ref{Lemma divisible L}, we know the $q^n$-coefficients ($n\geq 1$) of $\log   Z (\bbC^4)$ must be of the form (after applying $s_4 = -s_1 - s_2 - s_3$):
%\begin{equation}  \label{form-1}\frac{m}{s_1} \cdot \frac{P_n(s_1, s_2, s_3,m)}{ s_2^{N_n} (s_1 + s_2 + s_3)^{M_n} Q_n(s_1+ s_2, s_3) },  \end{equation}
%where $P_n, Q_n$ are polynomials and $N_n,M_n\in \mathbb{N}$ are non-negative integers.
%First, we have the $q^n$-coefficients $(n\geq 1)$ of $\log Z(\C^4)$ must be of the form 
\begin{equation}  \label{form-1}
\frac{m}{s_1 s_2^a s_3^b (-s_1 - s_2 - s_3)^c} \cdot \frac{P_n (s_1, s_2, s_3,m)}{Q_n (s_1 + s_2, s_3)}, 
\end{equation}
where $a,b,c\in \mathbb{Z}$ are integers (possibly depending on $n$), $P_n$ and $Q_n$ are homogeneous polynomials, such that \eqref{form-1} is homogeneous of degree $0$. 
We may also assume that $P_n$ and $Q_n$ are not divisible by $s_1, s_2, s_3, s_1 + s_2 + s_3$. 

By symmetry, we have 
$$\log Z(\C^4)=F(s_1, s_2, s_3, s_4, q) |_{s_4 = -(s_1 + s_2 + s_3)}, $$ 
for some invariant function $F \in \C (s_1, s_2, s_3, s_4,q)^{S_4}$,
where $S_4$ is the symmetric group of $s_1, s_2, s_3, s_4$. 
In particular, $Z(\C^4)$ is symmetric in $s_1, s_2, s_3$, and hence the order of poles at $s_1, s_2, s_3$ are the same. Therefore $a = b = 1$. 
As symmetric rational functions are quotients of symmetric polynomials, we may further assume that $P_n$ and $Q_n$ are symmetric in $s_1, s_2, s_3$.

Since $Q_n (s_1 + s_2, s_3)$ is symmetric in $s_1, s_2, s_3$, we have  
$$
Q_n (s_1 + s_2, s_3) = Q_n (s_1 + s_3, s_2). 
$$
By setting $s_3 = 0$, we see that the 2-variable polynomial $Q_n (x,y)$ satisfies that 
$$Q_n (x,y)= Q_n (x+y, 0), $$ 
which implies that $Q_n (s_1+ s_2, s_3) = Q_n (s_1 + s_2 + s_3, 0)$. 
By homogeneity, we have $$Q_n (s_1 + s_2, s_3) \in \C \cdot (s_1 + s_2 + s_3)^{\mathbb{N}}. $$ 
However, $Q_n(s_1 + s_2, s_3)$ is not divisible by $s_1 + s_2 + s_3$ by assumption, so it must be a constant. 

By the $S_4$-invariance, the function $F$ is invariant under
$$
s_1 \leftrightarrow - s_1 - s_2 - s_3.  
$$
By investigating the order of $s_1$ and $-s_1 -s_2 -s_3$, we know $c=1$. After applying $s_4 = -s_1 - s_2 - s_3$,  \eqref{form-1} is of the form   
%So we arrive at the form 
%$$ \frac{m}{s_1 s_2 s_3 (-s_1 - s_2 - s_3)} \cdot P_n (s_1, s_2, s_3,m).  $$
%On the other hand, $Z (\bbC^4)$ is symmetric\footnote{More precisely, it means the following: there exists a rational function $F \in \bbC (s_1, s_2, s_3, s_4) [q]$, symmetric with %respect to $s_1,s_2,s_3,s_4$, whose specialization under $s_1 + s_2 + s_3 + s_4 = 0$ is $\log   Z (\bbC^4)$. Then we apply the remaining argument to $F$.} 
%in $s_k$'s, $1\leq k\leq 4$. 
%in $s_1,s_2,s_3,s_4$.
%Therefore (\ref{form-1}) must be of the form 
$$
\frac{m}{s_1 s_2 s_3 s_4}  \left(  f_1 (s_1, s_2, s_3, s_4) + m  f_2 (s_1, s_2, s_3, s_4) + m^2  f_3 (s_1, s_2, s_3, s_4)  +  m^3  f_4 (s_1, s_2, s_3, s_4)  \right)  , 
$$
where $f_i$'s are symmetric polynomials in $s_1,s_2,s_3,s_4$ (depending on $n$), homogeneous of degree $4-i$ (so that the expression is homogeneous of degree $0$ by the convention in Remark \ref{rmk on dg par}).
%In particular, $f_4 \in \bbC$. 

It follows that
$$
f_1 (s_1, s_2, s_3, s_4) = a_3 \, e_3 (s) + a_{2,1} \, e_2 (s)\, e_1 (s) + a_{1,1,1} \, e_1 (s)^3, 
$$
$$
 f_2 (s_1, s_2, s_3, s_4) = a_2 \, e_2 (s) + a_{1,1} \, e_1(s)^2, 
$$
$$
f_3 (s_1, s_2, s_3, s_4) = a_1 \, e_1 (s), \quad f_4 (s_1, s_2, s_3, s_4) = a_0, 
$$
for some $a_3, a_{2,1}, a_{1,1,1}, a_2, a_{1,1}, a_1, a_0 \in \bbC$ and 
 $$e_k (s) := \sum_{1\leq i_1 < \cdots < i_k \leq 4} s_{i_1} \cdots s_{i_k}$$ are the elementary symmetric polynomials. 

Under the CY condition $s_1 + s_2 + s_3 + s_4 = 0$, we have $e_1(s)=0$, thus we may take 
$$
a_{2,1} = a_{1,1,1} = a_{1,1} =  a_1 = 0.
$$  
Therefore  (\ref{form-1}) must be of the form
$$
\frac{1}{s_1 s_2 s_3 s_4}  \left(  a_3\,  e_3 (s) \, m +  a_2\, e_2 (s) \, m^2 +  a_0 \, m^4   \right) . 
$$
Note that $a_3, a_2, a_0$ actually also depend on $n$, which we omit to avoid cumbersome notations. 
Summing over $n$, we conclude that there exists some $g_3 (q) , g_2 (q), g_0 (q) \in q\cdot \bbC[\![q]\!]$ such that 
$$
\log   Z (\bbC^4) = \frac{m\, e_3 (s)\, g_3 (q) + m^2 e_2 (s) g_2 (q) + m^4 g_0 (q) }{s_1 s_2 s_3 s_4} . 
$$
To determine the formula, by taking $m=-s_4 =s_1 + s_2 + s_3$ (ref.~\cite[Lem.~3.14]{CK1}), 
we specialize to the case when $L =  \calO\otimes t_4^{-1}$,~i.e.~the line bundle associated to the \emph{smooth} $(\bbC^*)^4$-invariant divisor $D=\C^3\subset \C^4$ cut 
out by $x_4=0$.
%Now we may specialize to the case $L =  \calO\otimes t_1^{d_1} t_2^{d_2} t_3^{d_3} t_4^{d_4}$, by letting $m = d_1 s_1 + d_2 s_2 + d_3 s_3 + d_4 s_4$. 
%A further specialization to the case $(d_1, d_2, d_3, d_4) = (1, 0, 0, 0)$, i.e. $L = \calO(D)$ for a $(\bbC^*)^4$-invariant \emph{smooth} divisor is already 
This case has been calculated by 4-fold vertex in \cite[Thm.~1.7]{CK1}. 
To be precise, the invariant in~\textit{loc.~cit.}~is defined at the level of torus fixed locus (i.e.~solid partitions). It is related 
to our globally defined invariant by torus localization formula \cite{OT}, with the sign rule proven in \cite{KR}.
%To link it with our globally defined invariants, 
By Proposition \ref{prop on cpt ori} and the sign rule in \cite{KR}, we know at torus fixed zero dimensional subschemes which are scheme theoretically supported on $D$, we can take 
the 4-fold vertex (with the induced sign) to be the same as the 3-fold vertex given in \cite{MNOP1} up to a sign $(-1)^n$ depending on the number $n$ of points. For other torus fixed subschemes, their contribution to invariants are zero by \cite[Eqn.~(3.8)]{CK1}. 

Let $\bar e_k(s_1, s_2,  s_3)$ be the specialization:
$$
\bar e_k(s_1, s_2,  s_3) := e_k (s_1, s_2, s_3, - s_1 - s_2 - s_3),
$$
which is a homogeneous polynomial in $s_1, s_2, s_3$ of degree $k$. 
Then we have
\begin{equation} \label{eqn-log-Z}
\log   Z (\bbC^4) \Big|_{m = -s_4} = \frac{-s_4 \bar e_3 (s) g_3 (q) + s_4^2\, \bar e_2 (s) g_2 (q) + s_4^4\,g_0 (q) }{s_1 s_2 s_3 s_4} . 
\end{equation}
By the MNOP formula \cite{MNOP1}, (\ref{eqn-log-Z}) is equal to
$$
 \frac{\bar e_3 (s) }{s_1 s_2 s_3} \log M(q), \,\,\, \mathrm{where} \,\,\bar e_3(s)=-(s_1+s_2)(s_1+s_3)(s_2+s_3). 
$$
%where$$e_3(s)=-(s_1+s_2)(s_1+s_3)(s_2+s_3). $$
By expanding (\ref{eqn-log-Z}) in $s_4 = - s_1 - s_2 - s_3$ and comparing with the above formula, we conclude that
$$g_3 (q) = -\log M(q), \quad g_2 (q) = g_0 (q) = 0. $$ 
The theorem then follows from the identity 
\begin{equation*}s_2 s_3 s_4 + s_1 s_3 s_4 + s_1 s_2 s_4 + s_1 s_2 s_3=-(s_1+s_2)(s_1+s_3)(s_2+s_3), \end{equation*}
where $s_4 = - s_1 - s_2 - s_3$.
\end{proof}
%\begin{remark}\end{remark}
By a similar degeneration argument or directly
%for equivariant integrals of $1$ over $[\Hilb^n (\bbC^4)]_T^\vir$ or directly
applying a specialization argument (e.g.~\cite[App.~A]{CKM1}, \cite[\S 6.2]{CKM3}), we can obtain the following invariants without insertions, which confirms a conjecture of Nekrasov \cite[Conj.~B.1]{CK1}.
\begin{corollary}\label{cor on nek}
$\displaystyle 1+\sum_{n=1}^\infty q^n \int_{[\Hilb^n (\bbC^4)]_T^\vir} 1  = \exp \left[-\dfrac{(s_1 + s_2 )(s_1 + s_3) (s_2 + s_3)  }{s_1 s_2 s_3 (s_1 + s_2 + s_3) } q
 \right] $.
\end{corollary}
We can also extract relative and rubber invariants from a similar pole analysis. 
%by using \eqref{eqn-rel-loc}, Lemma \ref{Lemma-s1} and Theorem \ref{thm on ck conj}. 
\begin{corollary} \label{cor-(-1, 0, 0)}

We have
\begin{enumerate}

\item $\displaystyle  W_{\infty} =M(q)^{\frac{m }{s_1}}$, 

\item $\displaystyle Z (X, D_\infty) = M(q)^{- \frac{m (s_2 s_3 + s_3 s_4 + s_2 s_4 ) }{ s_2 s_3 s_4} }$. 

\end{enumerate}

\end{corollary}
\begin{proof}
The second  equation follows from the first equation, Eqn.~\eqref{eqn-rel-loc} and Theorem \ref{thm on ck conj}. 
We prove the first one. 
We have
$$
\log   Z (\bbC^4) = \log   Z (X, D_\infty) - \log   W_{\infty} =  \log   Z (X, D_\infty) - \frac{F_{\infty, 0} }{s_1}. 
$$
where $\log Z (X, D_\infty)$ only has poles in monomials of $s_2, s_1 + s_2, s_3, s_1+s_2 + s_3$ (Lemma~\ref{Lemma-poles}).
Therefore, $-F_{\infty, 0}$ is the residue of $\log Z(\bbC^4)$ at $s_1$, which is $-m \log M(q)$. 
\end{proof}
This corollary allows us to compute relative tautological invariants on 
$$(X=\Tot (\calO_{\bbP^1} (-1) ) \times \bbC^2,\,\, D_\infty= \{\infty\} \times \bbC^3)$$ 
for \textit{topologically nontrivial} line bundle (i.e.~not just~$L_m=\oO \otimes e^m$).
Consider a $(\bbC^*)^4 \times \bbC^*_m$-equivariant line bundle 
$$L[l]\in \Pic^{(\bbC^*)^4 \times \bbC^*_m}(X)$$ such that
$$
L[l] |_{X \backslash D_\infty} = \calO \otimes e^m, \quad L[l] |_{X \backslash D_0} = \calO \otimes e^m \cdot t_1^l. 
$$
In other words, $L[l]$ is an equivariant enhancement of the pullback of the line bundle $\calO_{\bbP^1} (-l)$ to $X$. 

As in Definition \ref{defi of relat inv}, we define 
$$Z_{L[l]} (X, D_\infty):=1+\sum_{n=1}^\infty q^n \int_{[\Hilb^n (X, D_\infty)]_T^\vir} e_{T \times \bbC^*_m }(L[l]^{[n]}). $$
\begin{corollary} \label{cor-tw-L}
$\displaystyle   Z_{L[l]} (X, D_\infty) =   Z_{} (X, D_\infty) \cdot M(q)^l$. 
\end{corollary}

\begin{proof}
Similar to \eqref{eqn-rel-loc}, we have 
\begin{equation*}
  Z_{L[l]} (X, D_\infty) =   Z (\bbC^4) \cdot   W_{\infty, t_1^l \cdot e^m} 
=   Z (\bbC^4) \cdot   W_{\infty} |_{m\mapsto m+ l s_1}, 
\end{equation*}
where $$W_{\infty,t_1^l \cdot e^m} := 1 + \sum_{n=1}^\infty q^n \int_{[\Hilb^{\sim, n} (\Delta, D_0 \sqcup D_\infty)]_T^\vir} 
\frac{e_{T \times \bbC^*_m }(\oO^{[n]}\otimes t_1^l \cdot e^m)}{s_1 - \psi_0} $$
is defined similarly as \eqref{w func}. Then the conclusion follows from Corollary \ref{cor-(-1, 0, 0)}\,(1) and \eqref{eqn-rel-loc}. 
\end{proof}

\subsection{Zero dimensional $\DT_4$ invariants of log CY local curves}

Let $C$ be a smooth connected projective curve of genus $g$, and $L_1$, $L_2$, $L_3$ be line bundles on $C$. 
Consider smooth quasi-projective 4-fold
\begin{equation}\label{equ of local curve}X := \Tot_C (L_1 \oplus L_2 \oplus L_3). \end{equation}
Let $p_1, \ldots, p_r \in C$ be $r$ distinct points and $D_i := \{p_i\} \times \bbC^3$ be the fibers over $p_i$ in $X$, which we will take as relative divisors.  
We impose the (log) CY condition:
$$
L_1 \otimes L_2 \otimes L_3 \cong \omega_C (p_1 + \cdots + p_r), 
$$
which implies that 
\begin{equation}\label{eqn antican div}
D:=\sqcup_{i=1}^r D_i \end{equation}
is an anti-canonical divisor of $X$ and 
\begin{equation} \label{eqn-deg-genus}
\deg L_1 + \deg L_2 + \deg L_3 = 2g-2 +r \geq -2.
\end{equation}
Let $(\bbC^*)^3$ be the torus acting on the fibers of $X\to C$, with CY subtorus 
$$T:=\{(t_2,t_3,t_4) \,|\, t_2t_3t_4=1\} \subset (\bbC^*)^3. $$
The ring of equivariant parameters is
$$
H^*_T(\pt) = \frac{\bbC [s_2, s_3,s_4]}{(s_2 + s_3 + s_4)}. 
$$
Let $L$ be a line bundle on $C$. 
By an abuse of notation, we denote its pullback to $X$ still by $L$. 
Let $\bbC^*_m$ be an extra 1-dimensional torus acting trivially on $X$ and Hilbert schemes on $X$.
Consider $\bbC^*_m$-equivariant line bundle on $X$:
\begin{equation}\label{equ on lm}L_m:=L\otimes e^m, \end{equation} 
with $\C^*_m$-weight to be $m$.
%with the equivariant character $e^m$. 
The $T$-fixed locus of $\Hilb^n (X, \sqcup_{i=1}^r D_i)$  is proper (Remark~\ref{rmk on proper}), whose 
virtual class can be defined as in previous sections.

To fix the \textit{orientation} of virtual classes, we need to fix the presentation of the log Calabi-Yau pair (ref.~Theorem~\ref{relative ori of cy4 family})
%\footnote{For ideal sheaves of points, relative orientations are equivalent to absolute orientations.}
as 
$$(X,\sqcup_{i=1}^r D_i)=(\Tot(\omega_{U}(S)),\Tot(\omega_S)), $$
and then use convention as specified by Remark \ref{rmk on ori2}. Here we take 
\begin{equation}\label{equ on pre as log cy42} 
U=\Tot_C\left(L_1\oplus L_2\right), \quad S= \sqcup_{i=1}^r\{p_i\}\times \C^2_{x_2,x_3}.
\end{equation}

%U= \bbP^1 \times \bbC^2_{x_2,x_3},\,\, S= \{\infty\} \times \bbC^2_{x_2,x_3}, \quad (2),\,(3): U= \bbP^1 \times \bbC^2_{x_2,x_3},\,\, S= \{0,\infty\} \times \bbC^2_{x_2,x_3}
%where $\C^2_{x_2,x_3}$ denotes the first two coordinate planes of $\C^3_{x_2,x_3,x_4}$. 

\begin{definition}\label{defi of rel inv of local curve}
Choose orientation as specified by Eqn.~\eqref{equ on pre as log cy42} and Remark \ref{rmk on ori2}. 
The $(T \times \bbC^*_m)$-\textit{equivariant relative} $\DT_4$ \textit{generating function} of $\left(X, \sqcup_{i=1}^r D_i\right)$ \textit{with tautological insertions} is
$$
 Z (X, \sqcup_{i=1}^r D_i\,; L_m) :=1+\sum_{n=1}^\infty q^n \int_{[ \Hilb^n (X, \sqcup_{i=1}^r D_i) ]^\vir_T } e_{T \times \bbC^*_m} (L_m^{[n]})  \in   \frac{\bbC (s_2, s_3, s_4 )}{(s_2 + s_3 + s_4)}[m] [\![ q ]\!].
$$
\end{definition}
Our approach to compute $  Z (X, \sqcup_{i=1}^r D_i\,; L_m)$ is an analogue to the application of degeneration formula in the 3-fold case \cite{BP1, BP2, OP}. Consider a simple degeneration of the curve $C$ into a union $$C_0 = C_- \cup_p C_+,$$ 
where the line bundles $L_i$ ($1\leq i\leq 3$) and $L$ degenerate to $L_i^\pm$ and $L^\pm$ on $C_\pm$. 
The topological data are required to satisfy
\begin{equation}\label{equ on topo data of dege}
g (C_-) + g(C_+) = g, \quad \deg L_i^- + \deg L_i^+ = \deg L_i, \quad \deg L^- + \deg L^+ = \deg L.  
\end{equation}
We denote
$$
Y_\pm := \Tot_{C_\pm} (L_1^\pm \oplus L_2^\pm \oplus L_3^\pm ). 
$$
Let $p_1, \ldots, p_{r_-}$ be the points that degenerate into $C_-$, and $p_{r_- +1}, \ldots, p_{r_- + r_+}$ be those degenerating into $C_+$, where $r_- + r_+ = r$. 
$Y_\pm$ then carries the corresponding divisors $D_i$'s. 
The node $p$ also introduces an extra divisor $D_p^\pm$ on $Y_\pm$. 
We require that Equ.~\eqref{eqn-deg-genus} also holds for $Y_\pm$. 

The degeneration formula, i.e. Corollary \ref{cor on dege for on pot zero}, implies that
$$
 Z (X, \sqcup_{i=1}^r D_i\,; L_m) = Z(Y_-, \sqcup_{i=1}^{r_-} D_i \sqcup D_p^- \,; L^-_m) \cdot Z(Y_+, \sqcup_{i=r_- +1}^{r} D_i \sqcup D_p^+ \,; L^+_m ),
$$
where the orientation in the RHS is chosen similarly as Definition \ref{defi of rel inv of local curve}.
%Moreover, if $g\geq 1$, we may also consider a degeneration of $C$ into a curve with a self node $p$.
%Let $\tilde C$ be its partial normalization at $p$. 
%We have $\tilde g := g(\tilde C) = g -1$, and there are two preimages $\tilde p'$, $\tilde p''$ of $p$ in $\tilde C$. 
%Denote by $\tilde X$

\begin{definition}\label{defi of log tan bdl}
For a smooth divisor $D \subset X$, let $\Omega_X[D]$ denote\footnote{Here we follow
the notation as \cite{MNOP2}. In many situations, it is denoted by $T_X(-\log D)$ (e.g.~\cite[\S 8.2.2]{Voi}).} the locally free sheaf of 1-forms on $X$ with \textit{logarithmic poles}
along $D$ and 
$T_X[-D]$ be its dual sheaf of \textit{tangent vectors with logarithmic zeros} along $D$. 
\end{definition}
\begin{example}
For the divisor $D$ in $\bbC^4$ defined by $x_1 = 0$, $T_{\bbC^4}[-D]$ is generated by tangent vectors 
$x_1 \partial_{x_1}$ and $\partial_{x_i}$ ($2\leq i\leq 4$). 
As $(\bbC^*)^4$-representations, we have 
$$T_{\bbC^4} = t_1^{-1} + t_2^{-1} + t_3^{-1} + t_4^{-1}, \quad T_{\bbC^4}[-D] =1 + t_2^{-1} + t_3^{-1} + t_4^{-1}. $$ 
In particular, $c_3 (T_{\bbC^4}[-D]) = -s_2 s_3 s_4$.
\end{example}

\begin{lemma} \label{Lemma-RHS-glue}
Let $D:= \sqcup_{i=1}^r D_i$. 
Then  
$$
\int_X c_1^{T \times \bbC^*_m} (L_m) \cdot c_3^T (T_X[-D]) = \deg L -m \left( \frac{1}{s_2} + \frac{1}{s_3} + \frac{1}{s_4} \right) (2-2g -r).
$$
\end{lemma}

\begin{proof}
%Let $\pi: X \to C$ be the projection. 
Recall that $L$ is the pullback of a line bundle from $C$. 
%and $T_X (-D) = \pi^* T_C (- \sum_{i=1}^r p_i)$.
By $T$-localization, we have
\begin{equation}\label{equ on localiz com}
\int_X c_1^{T \times \bbC^*_m} (L_m) \cdot c_3^T (T_X [-D]) = \int_C \frac{
c_1^{T \times \bbC^*_m} (L_m) \cdot c_3^T (T_X[- D] |_C )
}{
c_3^T (L_1 \oplus L_2 \oplus L_3)
}.
\end{equation}
Then the result follows from  the short exact sequence 
$$0 \to L_1 \oplus L_2 \oplus L_3  \to T_X[-D] |_C \to T_C \left(-\sum_{i=1}^r p_i \right) \to 0,$$ 
and a direct calculation. 
\iffalse
which implies that
\begin{align*}
c_3^T (T_X[- D] |_C )=\prod_{i=1}^3\left(c_1^T (L_i)\right) + \left(\prod_{1\leqslant i< j\leqslant 3} c_1^T (L_i)   c_1^T (L_j) \right)  c_1^T \left( T_C \left( -\sum_{i=1}^r p_i \right) \right). 
\end{align*}
On the other hand, for $1\leq i\leq 3$, we have: 
$$
c_1^{T \times \bbC^*_m} (L_m) = m + c_1 (L), \quad c_1^T (L_i) = -s_{i+1} + c_1 (L_i), $$ 
$$c_1^T \left( T_C \left(-\sum_{i=1}^r p_i \right) \right) = c_1 \left( T_C \left(-\sum_{i=1}^r p_i \right) \right)=(2-2g-r)\,[\pt]. 
$$
Combining with \eqref{equ on localiz com}, we obtain our conclusion. 
\fi
%We conclude that $$ is equal to
\end{proof}

Now we state our main theorem on zero-dimensional $\DT_4$ invariants for any log CY local curve.
\begin{theorem} \label{thm on local curve}
Let $X$ be a local curve \eqref{equ of local curve}, $\sqcup_{i=1}^r D_i$ \eqref{eqn antican div} be its anti-canonical divisor and 
$L_m$ be the line bundle \eqref{equ on lm}.
Then
$$
Z (X, \sqcup_{i=1}^r D_i\,; L_m) = M(q)^{\int_X c_1^{T \times \bbC^*_m} (L_m) \cdot c_3^T (T_X[-D])}, 
$$
where $M(q):=\prod_{n\geqslant 1}\frac{1}{(1-q^n)^n}$
is the MacMahon function, and $T_X[-D]$ is given in Definition \ref{defi of log tan bdl}. 
\end{theorem}

\begin{proof}
The moduli space of complex structures on a curve $C$ and Picard groups of line bundles on $C$ with fixed degrees are connected. 
The equivariant Euler class of the tautological bundle of $L_m$ depends only on the degree of $L$ by the Grothendieck-Riemann-Roch theorem. 
With orientation chosen canonically in Definition \ref{defi of rel inv of local curve}, 
our generating function is deformation invariant and hence depends only on the data: $g$, $\deg L_i$, $\deg L$ and $r$, which we denote as 
$$
 Z (g; l_1, l_2, l_3; l):=Z (X, \sqcup_{i=1}^r D_i\,; L_m), 
$$
where $l_i := \deg L_i$, $l:= \deg L$, and $r$ is determined by Eqn.~\eqref{eqn-deg-genus}.
The degeneration formula implies
\begin{equation}\label{equ on dg fo}
Z (g; l_1, l_2, l_3; l) = Z (g_-; l_1^-, l_2^-, l_3^-; l^-) \cdot Z (g_+; l_1^+, l_2^+, l_3^+; l^+),
\end{equation}
where relations between all topological data are given in \eqref{equ on topo data of dege}. 
Therefore invariants for $g\geq 1$ can be reduced to products of invariants on $g =1$. 
As for a $g=1$ curve $C$, we can consider a degeneration of $C$ into $C_- \cup_{\{p, p'\}}  C_+$, where $C_\pm \cong \bbP^1$, hence invariants can then be further reduced to $g=0$ case. 

On the other hand, Lemma \ref{Lemma-RHS-glue} shows that the RHS of the theorem also satisfies a similar gluing formula. 
Hence, it suffices to prove the theorem for $g=0$.

The formula \eqref{equ on dg fo} for $(0; 0, 0, 0; 0)$ is 
$$
Z (0; 0, 0, 0; 0) = Z (0; 0, 0, 0; 0)^2, 
$$
which implies that $Z (0; 0, 0, 0; 0) = 1$, and hence
$$
Z (0; l_1, l_2, l_3; l) \cdot Z (0; - l_1, -l_2, -l_3; -l ) = 1. 
$$
Combining with a further decomposition of $l_i$'s, $l$ using \eqref{equ on dg fo}, it suffices to prove the theorem for
the following four cases:
$$
(g; l_1, l_2, l_3; l)   =  (0; -1, 0, 0; 0),\,\,\, (0; 0, -1, 0; 0), \,\,\, (0; 0, 0,-1; 0), \,\,\, (0; -1, 0, 0; 1).
$$
In the first case, the theorem reduces to Corollary \ref{cor-(-1, 0, 0)} and Lemma \ref{Lemma-RHS-glue}. 
The second and third cases should follow from the first case by symmetry, but one should be careful about the choice of orientation:
by the orientation convention \eqref{equ on pre as log cy42} and the invariance of Corollary~\ref{cor-(-1, 0, 0)}\,(2) under variable change: $s_2\leftrightarrow s_3$, we know the second case follows from the first case.
As for the third case, we can redo \S \ref{Sec-rel-rub} by using the orientation fixed by presentating the log Calabi-Yau pairs in Definition \ref{defi of relat inv} 
as \eqref{equ on pre as log} with 
%in each individual case by: 
\begin{equation*}
(1): U= \bbP^1 \times \bbC^2_{s_3,s_4},\,\, S= \{\infty\} \times \bbC^2_{s_3,s_4}; \,\,\, (2)\,\&\,(3): U= \bbP^1 \times \bbC^2_{s_3,s_4},\,\, S= \{0,\infty\} \times \bbC^2_{s_3,s_4},\end{equation*}
in each individual case (compared with the previous choice given in \eqref{equ on pre as log cy4}).
Since the formula of Theorem \ref{thm on ck conj} (therefore also Corollary~\ref{cor-(-1, 0, 0)}\,(2)) is invariant under the permutation $(s_1,s_2,s_3,s_4)\to (s_1,s_3,s_4,s_2)$, this case can also be reduced to the first case.

The last case $(0; -1, 0, 0; 1)$ follows from Corollary~\ref{cor-tw-L} and Lemma~\ref{Lemma-RHS-glue}.
The theorem is therefore proved.
%We can check the case $(0; -1, 0, 0; 0)$ (for the tangent weights, see the geometry for $(X_{-1, 0, 0}, D_\infty)$ in Section \ref{Sec-rel-rub} (1)):
%\begin{eqnarray*}
%\int_X c_1^{T \times \bbC^*_m} (L) \cdot c_3^T (T_X (-D)) &=& m \left( \frac{- s_1 s_2 s_3 - s_1 s_2 s_4 - s_1 s_3 s_4 - s_2 s_3 s_4 }{s_1 s_2 s_3 s_4}  + \frac{ (s_1 + s_2 ) s_3 s_4 }{s_1 (s_1 + s_2) s_3 s_4 } \right) \\
%&=& - m \frac{ s_2 s_3 + s_2 s_4 + s_3 s_4  }{ s_2 s_3 s_4} , 
%\end{eqnarray*}
%which matches with the result of Corollary \ref{cor-(-1, 0, 0)}. 
%For the case $(0; -1, 0, 0; 1)$, it suffices to take $L = \calO (q) \otimes e^m$ for a point $q \in \bbP^1$ away from $\infty$. 
%We then have 
%$$
%\int_X c_1^{T \times \bbC^*_m} (L) \cdot c_3^T (T_X (-D)) = \int_X c_1^{T \times \bbC^*_m} (\calO \otimes e^m) \cdot c_3^T (T_X (-D) ) +1. 
%$$ 
%The theorem then follows from Corollary \ref{cor-tw-L}. 
\end{proof}

\appendix 

\section{Rubber moduli stacks}\label{sec on rub}

A very useful variant of the construction in \S \ref{Sec-exp-pair} and \S \ref{sect on exp dege} is the \emph{rubber moduli stacks} (ref.~\cite[\S 4.8,~\S 4.9]{OP}).
We briefly describe their construction here, focusing on moduli of zero-dimensional subschemes. 
With a view toward applications in \S \ref{sect on hilb}, we consider only the following case
\begin{itemize}
%\item $(Y, D)$ is of the form $Y = \Tot(\omega_{U}(S))$, $D=\Tot(\omega_S)$, where $(U, S)$ is a smooth pair and $\dim_\bbC U = 3$.
\item  $D=\Tot(\omega_S)$ for a smooth surface $S$, with a torus $T$-action on $D$ which preserves the Calabi-Yau volume form on $D$. 
\end{itemize}

\subsection{Rubber expanded pairs}

Let $\Delta:= D \times \bbP^1$, which\footnote{This comes from $\Delta = \bbP_D (\calO_D \oplus N_{D/Y})$ for a smooth pair $(Y, D)$ as in \S \ref{Sec-exp-pair}, with $N_{D/Y} \cong \calO_D$.} comes with two canonical divisors $D_0 = D\times \{0\}$ and $D_\infty := D \times \{\infty\}$.  
Similar to Proposition~\ref{prop:st_family}, for each $k\in\bbN$, there is a standard family $\Delta[k] \to \bbA^k$, whose central fiber is 
$$\Delta[k]_0 = \Delta \cup_D\underbrace{\Delta \cup_D \cdots \cup_D \Delta}_{k \text{-times}}. $$
The standard families for different $k$'s fit together. 
The family $\Delta[k] \to \bbA^k$ is $(\bbC^*)^{k+1}$-equivariant. Together with the discrete symmetries 
 containing permutations of the coordinates, we obtain a smooth equivalence relation $\sim$. 
The difference between this setting and that in \S \ref{Sec-exp-pair} is that here we also \textit{turn on} the $\C^*$-action on the first $\Delta$ in $\Delta[k]_0$.

We are primarily  interested in Hilbert stacks of points, where topological data are of the form
$$
P = n [\calO_{\pt}]  \in  K_c^\num (D), 
$$
with $\pt \in D$ a closed point. 
In numerical $K$-theory, the class $[\calO_{\pt}]$ is independent of the choice of $\pt\in D$ by the connectness of $D$. 
Transversality condition then forces that $P_- |_D = 0$ and $\Hilb^{P_- |_D} (D) = \Spec \C$. 
In what follows, we denote the topological datum by $n$ instead of $n [\calO_{\pt}]$. 
%and the Hilbert stack by 
%$$\Hilb^{\sim, n} (\Delta, D_0 \sqcup D_\infty). $$

Analogous to Definitions \ref{def of exp pair} \& \ref{defin on rel hilb stack}, we have
\begin{definition}\label{def of rubber pair}
Let $Y$, $D$, $\Delta$ be as above, and $n \geq 0$ be an integer.  
\begin{itemize}
\item The stack of \emph{rubber expanded pairs} is defined as 
$$
\calA^\sim := \underrightarrow\lim\, [\bbA^k / \sim]. 
$$ 
It is a smooth Artin stack of dimension $-1$.
As before, we denote by $\calA^{\sim, n}$ the stack of rubber expanded pairs with total continuous weight $P=n [\calO_{\pt}] $. 

\item The \emph{rubber universal family} is defined as $$\calR^\sim := \underrightarrow\lim\, [\Delta[k] / \sim]. $$
Similarly, let $\calR^{\sim, n}$ be the universal family over $\calA^{\sim,  n}$. 

\item Consider closed subschemes on the stack of rubber expanded pairs (and with total continuous weight $n$), whose dimensions relative to the base are $0$. 
Denote the corresponding \textit{rubber relative Hilbert stacks} (and its weighted version) by $$\Hilb^\sim (\Delta,  D_0 \sqcup D_\infty), \quad (\Hilb^{\sim, n} (\Delta, D_0 \sqcup D_\infty)). $$ 

\item All the constructions in \S \ref{sect on rel invs} carry out straightforwardly for rubber relative Hilbert stacks. 
In particular, we have 
 a rubber analogue to Theorem \ref{thm on relative invs} for the rubber relative Hilbert stack of zero-dimensional subschemes.
The map
$$
r_{\calD\to \calR^\sim} : \Hilb^{\sim, n} (\Delta, D_0 \sqcup D_\infty) \to  \calA^{\sim, n}
$$
admits a canonical ($T$-equivariant) isotropic symmetric obstruction theory in sense of Definition \ref{def on sym ob}. 
Hence, with orientation specified by Eqn.~\eqref{equ on pre as log cy4} and Remark \ref{rmk on ori}, there is a ($T$-equivariant) square root virtual pullback
$$
\Phi^{\sim, n}_{(\Delta ,D_0 \sqcup D_\infty )} := \sqrt{ r_{\calD\to \calR^\sim}^!} :  A_{*}(\calA^{\sim, n}) \to A_*^T(\Hilb^{\sim, n} (\Delta ,D_0 \sqcup D_\infty )). 
$$
The \textit{rubber} $T$-\textit{equivariant virtual class} is 
%\footnote{Although called by the same name, $[\calZ^{\sim, n}]^\vir$ does not come from an obstruction theory on $\calZ^{\sim, n}$.}:
$$
[\Hilb^{\sim, n} (\Delta ,D_0 \sqcup D_\infty )]_T^\vir := \Phi^{\sim, n}_{(\Delta ,D_0 \sqcup D_\infty )} [\calA^{\sim, n}]. $$
\end{itemize}
\end{definition}
 
\begin{remark}\label{rm on nonepty moduli}
As the $\bbC^*$-action is turned on,  
% in order to have finite automorphisms 
we need $n\geqslant 1$ to have nonempty Hilbert stacks.
\end{remark}

As observed by Graber and Vakil \cite{GV} (see also~\cite[Prop.~4.2]{ACFW}), the stack $\calA^\sim$ can be identified with a moduli stack of curves. 
Let $\fM_{0,2}^{\rm ss}$ be the Artin stack parameterizing semistable nodal curves of genus $0$ with two markings, and $\calC_{0,2}$ be its universal curve.  

\begin{prop} \label{prop-fM02}
There is an isomorphism $\calA^\sim \cong \fM^{\rm ss}_{0,2}$.
Moreover, if $\Delta = \bbP^1 \times D$,~i.e.~$N_{D/Y} \cong \calO_D$, then the isomorphism above induces an isomorphism $\calR^\sim \cong \calC_{0,2} \times D$. 
\end{prop}
Let $s_0: \calD_0 \cong  \calA^\sim \times D \hookrightarrow \calR^\sim\cong \calC_{0,2} \times D$ be the inclusion of the universal relative divisor. 
The universal cotangent line bundle at $\calD_0$ is defined as the line bundle $\calL_0 := s_0^* \omega_{\calR^\sim / \calA^\sim}$ over $\calA^\sim \times D  $.
When $\Delta = D \times \bbP^1$, we have $$\calL_0 = s_0^* \proj_{\calC}^* \omega_{\calC_{0,2} / \fM_{0,2}^{ss}} = \proj_{\calA}^* \calL_0 (\fM_{0,2}^{ss}),$$ 
where $\proj_{\calC}: \calC_{0,2} \times D\to \calC_{0,2}$, $\proj_{\calA}: \calA^\sim \times D\to \calA^\sim$ are projections and  
$\calL_0 (\fM_{0,2}^{ss})$ is the tautological cotangent line bundle of $\fM_{0,2}^{ss}\cong\calA^\sim$ at marked point $0$.

\begin{definition} \label{Defn-psi}
Assume $\Delta = D \times\bbP^1$. The pullback of the  psi-class $\psi_0 = c_1 (\calL_0 (\fM_{0,2}^{ss}))$  of $\fM^{\rm ss}_{0,2}$ under the isomorphism $\calA^\sim \cong \fM^{\rm ss}_{0,2}$ is called the psi-class $\psi_0 (\calA^\sim)$ at $\calD_0$. 
\end{definition}

\subsection{Moduli interpretations of $\calR^\sim$}\label{sect on mod inte}
 By definition, for a scheme $S$,  a family of rubber expanded pairs is a map $S\to\calA^\sim$. Let 
 $\calR^\sim_S \to S$ be the pullback of the universal family. 
 A map $S \to \calR^\sim$  is then equivalent to a pair $(\calR^\sim_S, \Sigma_{1, S})$, where $\calR^\sim_S \to S$ is a family  of rubber expanded pairs, and  $\Sigma_{1, S}$ is 
\emph{a section}  $S \to \calR^\sim_S$,  whose image in $\calR^\sim_S$ is also denoted by $\Sigma_{1, S}$. 
Let  $D_{0, S}$, $D_{\infty, S}$ be the relative divisors in $\calR^\sim_S$.  
Note that $\Sigma_{1, S}$ is allowed to intersect the nodal and relative divisors. 

In particular, over a geometric point $s\in S$, 
we have $\calR^\sim_s \cong \Delta [k_s]_0$ for some $k_s \geq 0$,
where 
$$
\Delta[k_s]_0 = \Delta \cup_D\underbrace{\Delta \cup_D \cdots \cup_D \Delta}_{k_s \text{-times}}. 
$$
The relative divisors in the rubber expanded pair are denoted by
$D_{0, s}$, $D_{\infty, s}$, and $\Sigma_{1, s}$ is a point in $\Delta [k_s]_0$. 
We denote by $\Delta_{1,s}$ the first component $\Delta$ in $\Delta[k_s]_0$, which contains $D_{0,s}$, and denote by $D_{1,s}$ the nodal divisor where $\Delta_{1,s}$ meets other irreducible components.

\begin{definition}
Let $D(0 |1, \infty)$ be the divisor in $\calR^\sim$, whose set of $S$-points consists of pairs $(\calR^\sim_S, \Sigma_{1, S})$ where $\Sigma_{1, s}$ does not lie in $\Delta_{1,s} \backslash D_{1,s}$ for any geometric point $s$ in $S$.
\end{definition}

From now on till the end of this section, we  work under the additional assumption that the bundle $N_{D/Y}$ is trivial. 

\begin{remark} \label{rk-2nd-moduli}
The stack $\calR^\sim$ and the divisor $D(0|1, \infty)$ has another moduli interpretation, parallel to Proposition~\ref{prop-fM02}, which we now explain. 
%The definition of $D(0|1, \infty)$ is more natural from this point of view. 
%By the analogy with moduli of curves, $\calR^\sim$ has \textit{another} moduli interpretation, i.e. parameterizing families $\widetilde\calR_S \to S$ of rubber expanded pairs, with $D_{0, S}$, %$D_{\infty, S}$ and a third section $\widetilde \Sigma_{1, S}$ which \emph{avoids all nodal and relative divisors}. 
By a standard construction,  analogue to the case of moduli of curves, $\calR^\sim$ is isomorphic to a moduli stack, whose $S$-points consists of pairs $(\widetilde\calR_S \to S, \widetilde \Sigma_{1, S})$, where  $\widetilde\calR_S \to S$ is a family  of rubber expanded pairs, with relative divisors $D_{0, S}$ and $D_{\infty, S}$,  and  $\widetilde \Sigma_{1, S}: S\to \widetilde\calR_S$ is a section which \emph{avoids} all nodal and relative divisors. Under this isomorphism,
$D(0| 1, \infty)$ is then the \textit{divisor} in $\calR^\sim$, where $\widetilde \Sigma_{1, S}$ lies in a different irreducible component from $D_{0, S}$. 
From this point of view, $\calR^\sim$ can be identified with an open substack\footnote{This embedding is implicitly used in the rubber calculus in e.g. \cite[\S 4.3]{MNOP2} \cite[\S 4.8]{OP}. For a reference, see \cite[Cor.~2.8]{BS1}.} in $\fM_{0,3}^{ss}  \times D$, with markings $0$, $1$, $\infty$, where curves are chains of rational curves;
the divisor $D(0|1, \infty)$ can then be identified with the closed substack where marking $0$ is required to lie on a different irreducible component from $1$ and $\infty$.
We also have a \textit{cotangent line bundle class} $$\psi_0 (\calR^\sim)\in H^2(\calR^\sim), $$ which can be identified with the pullback of $\psi_0$ on $\fM_{0,3}^{ss}$.
%the pullback of $\psi_0 (\calA^\sim)$ in Definition \ref{Defn-psi} and $\psi_0 (\calR^\sim)$ differs by $D(0,1|infty)$, but we have Prop.~\ref{prop decom rubber}. 
\end{remark}

A generic point in $D(0 |1, \infty)$ classifies an object $\Delta \cup_D \Delta$, where the point $\Sigma_{1,s}$ lies generically on the second $\Delta$ which contains $D_\infty$. %\gufang{The divisor $D(0 |1, \infty)$ is reducible. Right? By definition, $D_1$ is a component of it. Do you mean the generic point on the other component?}\zz{No, it is irreducible. It is like in $\fM_{0,3}$, we have a divisor $D(0|1,\infty)$, which parameterizes curves where $0$ is in a different component from $1$ and $\infty$. By a \emph{generic point}, it is similar as saying ``a generic point in $\fM_{0,3}$ consists of a smooth rational curve with 3 markings", and ``a generic point in $D(0 |1,\infty) \subset \fM_{0,3}$ consists of a curve with two irreducible components $C_1 \cup C_2$, with $0 \in C_1$, $1, \infty \in C_2$". By ``generic" I mean there are also more degenerate curves, which are non-generic points in the moduli of curves.}\gufang{I got your point. Let's discuss more this Friday.}
Given a splitting datum $n_- + n_+ = n$, we denote by $D(0 |1, \infty)^{n_-, n_+}$ the divisor in $\calR^{\sim, n}$, for a generic point of whom, the first $\Delta$ (i.e. containing $D_0$) carries weight $n_-$ and the second $\Delta$ (i.e. containing $\Sigma_{1,s}$ and $D_\infty$) carries weight $n_+$. 
%For $D(0 |1, \infty)^{n_-, n_+}$ to be nonempty, we need $n_-, n_+\geqslant 1$.

\section{Comparison of orientations on $\Hilb^n(\C^4)$}
In this section, we consider Hilbert schemes $\Hilb^n(\C^4)$ of points on $\C^4$. We compare the orientation specified by Remark \ref{rmk on ori}  and Eqn.~\eqref{ori on hilbc4} with the orientation chosen in the work of Kool-Rennemo \cite{KR}.

Let $U=\C^3$ and $X=\Tot(\omega_U)=\C^4$ with the projection 
\begin{equation}\label{equ on pro map}\pi: X=\Tot(\omega_U)\to U, \quad (x_1,x_2,x_3,x_4)\mapsto (x_1,x_2,x_3). \end{equation}
Let $\bullet=X$ or $U$, and $M_n(\bullet)$ be the moduli stack of zero dimensional sheaves of length $n$ on $\bullet$  with closed embedding 
$$M_n(\bullet)\hookrightarrow NM_n(\bullet):=\left[\End(\mathbb{C}^n)^{\times \dim_{\C}(\bullet)}/\GL(n)\right], $$
the image of which is characterised by the commuting relations in endormorphisms of $\mathbb{C}^n$:
\begin{equation}\label{equ on com rel}x_ix_j=x_jx_i, \quad \forall\,\, 1\leqslant i,j\leqslant  \dim_{\C}(\bullet). \end{equation}
Recall the embedding of Hilbert schemes into non-commutative Hilbert schemes (Example \ref{fl ex}):
\begin{align*}\Hilb^n(\bullet)\hookrightarrow \mathrm{NHilb}^n(\bullet)&:=\left(\End(\mathbb{C}^n)^{\times \dim_{\C}(\bullet)}\times \mathbb{C}^n\right)/\!\!/\GL(n), 
%&=\left(\Hom(\mathbb{C}^n,\mathbb{C}^n)^{\times \dim_{\C}(\bullet)}\times \mathbb{C}^n\right)^s/\GL(n)
\end{align*}
the image of which
is  characterised by the commuting relations \eqref{equ on com rel}.
We have the forgetful map   
$$ f: \mathrm{NHilb}^n(X) \to NM_n(X), \quad (x_1,\ldots, x_{4},v)\mapsto (x_1,\ldots, x_{4}), $$
whose restriction to $\Hilb^n(X)$ gives $f: \Hilb^n(X)\to M_n(X)$. 
Let us consider the projection 
$$\pi_\dagger: NM_n(X)\to NM_n(U), \quad (x_1,x_2, x_3, x_{4})\mapsto (x_1,x_2, x_{3})$$
and its restriction $\pi_\dagger: M_n(X)\to M_n(U)$ to $M_n(X)$. 
There is a commutative diagram
\begin{equation}\label{diag on hilb and tor moduli}
\xymatrix{
\Hilb^n(X) \ar[d]^{  } \ar[r]^{f}   & M_n(X)  \ar[r]^{\pi_\dagger}  \ar[d]^{\iota} & M_n(U) \ar[d]^{  }  \\
\mathrm{NHilb}^n(X) \ar[r]^{f} & NM_n(X)  \ar[r]^{\pi_\dagger}   &  NM_n(U), 
}\end{equation}
where vertical maps are the closed embeddings mentioned above.

For $\bullet=X$ or $U$, let $\mathbb{F}_{\bullet}\to M_n(\bullet)\times \bullet$ be the universal object and $\pi_{M_{\bullet}}: M_n(\bullet)\times \bullet\to M_n(\bullet)$ be the projection. By \eqref{equ on spectral cons}, \eqref{equ useful proof1.2} and a base change, we have a map 
\begin{equation}\label{equ on map from mxto mu}\pi_{M_X*}\dR\calH om(\mathbb{F}_X,\mathbb{F}_X)[1]\to \pi_{\dagger}^*\pi_{M_U*}\dR\calH om(\mathbb{F}_U,\mathbb{F}_U)[1], \end{equation}
which is the restriction of the tangent map of the derived enhancement of $\pi_\dagger: M_n(X)\to M_n(U)$ to the classical truncation. 
We give a description of this map using data on ambient stacks $NM_n(\bullet)$. 
\begin{lemma}
There exists a map of $\GL(n)$-equivariant complexes on $\End(\mathbb{C}^n)^{4}$:
\begin{equation}\label{equ cpx on cpr ori}
{\footnotesize
\xymatrix{
\fg \fl_n\otimes \oO   \ar[r]^{ }  \ar[d]^{=} & \fg \fl_n\otimes \C^4\otimes \oO   \ar[r]^{  }  \ar[d]^{ }  & \fg \fl_n\otimes \bigwedge^2\C^4\otimes \oO  \ar[r]^{ }  \ar[d]^{  }  & \fg \fl_n\otimes \bigwedge^3\C^4 \otimes \oO \ar[r]^{ }  \ar[d]^{ }  & \fg \fl_n \otimes \bigwedge^4\C^4 \otimes \oO    \ar[d]_{ } \\
 \fg \fl_n\otimes \oO  \ar[r]^{ }  & \fg \fl_n\otimes \C^3\otimes \oO  \ar[r]^{  } & \fg \fl_n\otimes \bigwedge^2\C^3\otimes \oO  \ar[r]^{ } & \fg \fl_n\otimes \bigwedge^3\C^3 \otimes \oO   \ar[r]^{ } & 0, } } 
\end{equation}
whose descent to $NM_n(X)=\left[\End(\mathbb{C}^n)^{4}/\GL(n)\right]$ defines a map of complexes of vector bundles. Its restriction to $M_n(X)$ is a resolution of the map \eqref{equ on map from mxto mu}
in the derived category. 

When restricted to $M_n(X)$, the complex $\left(\fg \fl_n\otimes \bigwedge^*\C^4\otimes \oO\right)$ has a natural pairing given by 
$$\bigwedge^*\C^4\otimes \bigwedge^*\C^4\to \bigwedge^4\C^4=\C, $$
which coincides with the Grothendieck-Serre duality pairing on $\pi_{M_X*}\dR\calH om(\mathbb{F}_X,\mathbb{F}_X)[1]$. 
\end{lemma}
\begin{proof}
The description of the complex $\pi_{M_U*}\dR\calH om(\mathbb{F}_U,\mathbb{F}_U)[1]$ in \eqref{equ on map from mxto mu} using the complex 
$(\fg \fl_n\otimes \bigwedge^*\C^3\otimes \oO)$ on $NM_n(U)$ is given by \cite[\S 3.2]{RS}, which uses a quasi-free dg resolution $Q_3$ of $\C[x_1,x_2,x_3]$.
One can do the $\C^4$ case by the quasi-free dg resolution $Q_4$ of $\C[x_1,x_2,x_3,x_4]$ given in \cite[pp.~55 \& Ex.~6.3.1]{Lam}.  
There is an explicit map between dg-algebra $Q_3\to Q_4$ (see e.g.~\cite[Ex.~3.4]{AS}), which induces a map of derived stacks of zero dimensional sheaves on $X$ and $U$: 
$$\pi_\dagger: \textbf{M}_n(X)\to \textbf{M}_n(U). $$
Its tangent map (described using $Q_4$ and $Q_3$) defines \eqref{equ cpx on cpr ori}. 
The statement on identification of pairings is because both are derived from the Calabi-Yau 4-algebra structure on $\C[x_1,x_2,x_3,x_4]$. 
%from the fact tCalabi-Yau structure of dg-categories. 
\end{proof}
One can construct  orientations on $\Hilb^n(X)$ in the sense of Definition \ref{ori on even cy}, by pulling back orientations on $M_n(X)$ (see \eqref{identify two det}).
On $M_n(X)$, orientations can be constructed using \textit{isotropic quotient complexes} $V^\bullet$ of $E^\bullet=\pi_{M_X*}\dR\calH om(\mathbb{F}_X,\mathbb{F}_X)[1]$. Here recall an isotropic quotient complex is a map $E^\bullet\stackrel{d}{\to} V^\bullet$
such that 
$$(V^{\bullet})^{\vee}[-2]  \xrightarrow{d^\vee[-2]} (E^\bullet)^\vee[-2] \cong E^\bullet   \xrightarrow{d} V^\bullet$$
is a distinguished triangle. The \textit{orientation} used in this paper (see \eqref{equ on map from mxto mu}, and Theorem \ref{ori of cy4 family}) is when 
$$V^\bullet=\pi_{\dagger}^*\pi_{M_U*}\dR\calH om(\mathbb{F}_U,\mathbb{F}_U)[1], $$ 
with an explicit description given by the bottom line of \eqref{equ cpx on cpr ori}. 

The \textit{orientation} used in Kool-Rennemo \cite{KR} is given by choosing a maximal isotropic quotient
$$ \fg \fl_n\otimes \bigwedge^2\C^4\otimes \oO\to \fg \fl_n\otimes \bigwedge^2\C^3\otimes \oO, $$
and formally adding automorphism and tangent part of $M_n(X)$,~i.e. considering
\begin{equation}\label{equ on kr ori}
{\footnotesize
\xymatrix{
 \fg \fl_n\otimes \oO   \ar[r]^{ }   & \fg \fl_n\otimes \C^4\otimes \oO   \ar[r]^{  }     & \fg \fl_n\otimes \bigwedge^2\C^4\otimes \oO  \ar[r]^{ }  \ar[d]^{  }  & 
 \fg \fl_n\otimes \bigwedge^3\C^4 \otimes \oO \ar[r]^{ } \ar@{=}[d]   & \fg \fl_n \otimes \bigwedge^4\C^4 \otimes \oO \ar@{=}[d]  \\
  &    & \fg \fl_n\otimes \bigwedge^2\C^3\otimes \oO \ar[r]^{  }   &  \fg \fl_n\otimes \bigwedge^3\C^4 \otimes \oO 
 \ar[r]^{  }   &  \fg \fl_n \otimes \bigwedge^4\C^4 \otimes \oO. } }
 \end{equation}
Note that although vertical maps in the above diagram is not a map of complexes, but one can use it to 
induce an orientation. 

\begin{prop}\label{prop on cpt ori}
The  two choices of orientations above coincide. 
\end{prop}
\begin{proof}
Recall that for a quadratic vector bundle $E=V\oplus V^\vee$ (with $V=A\oplus B$) on a connected scheme (where the pairing on $E$ is given by the natural pairing between $V$ and $V^\vee$),  if we choose its orientations induced by isotropic subbundles $A\oplus B$, $A\oplus B^\vee$, 
the orientations are the same if $\rk B$ is even and are opposite otherwise (e.g.~\cite[Eqns.~(7),~(8)]{OT}). 

Similarly we can compare the difference of orientations
determined by 
\eqref{equ cpx on cpr ori} and \eqref{equ on kr ori}.
Their difference happens on terms $\fg \fl_n\otimes \C\otimes \oO$ v.s. $\fg \fl_n\otimes \bigwedge^4 \C^4 \otimes \oO$,
and $\fg \fl_n\otimes \C^3\otimes \oO$ v.s. $\fg \fl_n\otimes \bigwedge^2\C^3\otimes \oO$ (where we use 
$\bigwedge^3\C^4=\bigwedge^3\C^3\oplus \bigwedge^2\C^3$). Since the rank of $(\fg \fl_n\otimes \C\otimes \oO)\oplus (\fg \fl_n\otimes \C^3\otimes \oO)$ is even, these two orientations coincide. 
\end{proof}

\section{Proof of Propositions~\ref{iso cond 2} \& \ref{iso cond 1}}\label{app on proof}

\begin{proof}[Proof of 
Proposition~\ref{iso cond 1}] 

The proof is by adapting Park's proof of deformation invariance of $\DT_4$ invariants \cite{Park2}. 
The infinitesimal criterion of isotropic condition in~\textit{loc.\,cit.} yields, the isotropic condition claimed in Proposition~\ref{iso cond 1} is equivalent to 
 the following. For any local Artinian scheme $\calB:=\Spec B$ with a map $\calB\to \calC^{P_t}$, 
the base-change of the
 obstruction theory \eqref{obs theory relW3} to the homotopy pullback  
\begin{equation*}\begin{xymatrix}{
\textbf{Hilb}^{P_t}(\calX_{\calB}/\calB)  \ar[r]^{ }   & \calB 
}\end{xymatrix}\end{equation*}
is isotropic. Here $\calX_{\calB}:=\calX^{P_t}\times_{\calC^{P_t}} \calB$ is the base change of $\calX^{P_t}\to \calC^{P_t}$ to $\calB$. 
Let $\Omega_{\textbf{Hilb}^{P_t}(\calX_{\calB}/\calB)/\calB}\in HN^{-4}(\textbf{Hilb}^{P_t}(\calX_{\calB}/\calB)/\calB)(2)$ be the pullback of the 
shifted symplectic form in Theorem \ref{thm on sft symp str}. 
By the relative local Darboux theorem \cite[Thm.~B]{Park2}, the isotropic condition follows if the image 
$$PC(\Omega_{\textbf{Hilb}^{P_t}(\calX_{\calB}/\calB)/\calB})\in HP^{-4}(\textbf{Hilb}^{P_t}(\calX_{\calB}/\calB)/\calB)(2)$$ 
of $\Omega_{\textbf{Hilb}^{P_t}(\calX_{\calB}/\calB)/\calB}$ in the periodic cyclic homology is zero. 
We use shorthands 
$$\textbf{M}:=\textbf{Hilb}^{P_t}(\calX_{\calB}/\calB), \quad M:=\Hilb^{P_t}(\calX_{\calB}/\calB). $$
Since $B$ is local Artinian, denote $\calB_{red}=\{b\}$. Let  $[I_Z]\in \textbf{M}$ be a closed point.
Necessarily, $[I_Z]$ lies in the fiber $\textbf{M}_b$ of $\textbf{M}$ over $b\in \calB$.
Note that the classical truncation $M_b$ of $\textbf{M}_b$ is a scheme. Without loss of generality, we assume its associated analytic space $M_b(\C)$ is connected. 
We have fiber diagrams
\begin{equation*}\begin{xymatrix}{
\calX_{b} \ar[r]^{ }  \ar[d]^{ } \ar@{}[dr]|{\Box} & \calX_{\calB}\times_{\calB}\textbf{M}  \ar[r]^{ }  \ar[d]^{ } \ar@{}[dr]|{\Box} & \calX_{\calB} \ar[d]^{}  \\
\{I_Z\} \ar[r]^{ } & \textbf{M} \ar[r]^{ } &  \calB.
}\end{xymatrix}\end{equation*}
By the functoriality of integration map, we have a commutative diagram 
\begin{equation}\label{comm diag on hp}\begin{xymatrix}{
HN_{\calZ_\calB}^0\left((\calX_{\calB}\times_{\calB}\textbf{M})/\calB\right)(2)   \ar[r]^{ }  \ar[d]^{\int_{\calX_{\calB}/\calB}\mathrm{vol}_{\calX_{\calB}/\calB}\wedge(-)}  & HP_{\calZ_\calB}^0\left((\calX_{\calB}\times_{\calB}\textbf{M})/\calB\right)(2) \ar[r]^{ } 
\ar[d]^{\int_{\calX_{\calB}/\calB}\mathrm{vol}_{\calX_{\calB}/\calB}\wedge(-)} &HP_{Z}^0\left(\calX_{b}/\calB\right)(2)  \ar[d]^{\int_{\calX_{b}}\mathrm{vol}_{\calX_{b}}\wedge(-)}  \\
HN_{}^{-4}\left(\textbf{M}/\calB\right)(2)   \ar[r]^{PC} &   HP_{}^{-4}\left(\textbf{M}/\calB\right)(2)  \ar[r]_{\cong}^{|_{\{I_Z\}}} & HP_{}^{-4}\left(\{I_Z\}/\calB\right)(2),
}\end{xymatrix}\end{equation} 
where $\calZ_\calB\subset \calX_{\calB}\times_{\calB}\textbf{M}$ is the universal substack,
and $HP_{\calZ_\calB}$ denotes the cohomology of periodic cyclic complex with support on $\calZ_\calB$, defined similarly
as \eqref{equ on nck}.

We claim that in the lower-right horizontal map above, the restriction map to $\{I_Z\}$,  is an isomorphism. Recall the following two facts.
\begin{itemize}
%\item $(Y, D)$ is of the form $Y = \Tot(\omega_{U}(S))$, $D=\Tot(\omega_S)$, where $(U, S)$ is a smooth pair and $\dim_\bbC U = 3$.
\item 
Nil-invariance of $HP^*$ shown by  Goodwillie \cite{Goo}: for any derived stack $\textbf{M}$ over $\calB$, the restriction map 
 to the reduced part of the classical truncation:
 \begin{equation}\label{rest map of hp}HP^*(\textbf{M}/\calB)(*) \to HP^*((t_0(\textbf{M}))_{red}/\calB)(*) \end{equation}
 is an isomorphism.
 
\item The isomorphism between $HP^*$ and the singular cohomology \cite[Thm.~2.2]{Emm} \& \cite[pp.~89,~Thm.~IV.~1.1]{Har}: for any $\C$-scheme $M_b$, we have a canonical isomorphism
\begin{equation}\label{sevel coh iso}HP^k(M_b)(p)\cong H^{k+2p}_{\mathrm{sing}}(M_b(\C)), \,\,\, \forall\,\, k,p\in \mathbb{Z}, \end{equation}
where $M_b(\C)$ is the associated analytic space of $M_b$.
\end{itemize}
Combining \eqref{rest map of hp} and \eqref{sevel coh iso}, we have 
\begin{align*}
&\quad \,\, HP^{-4}(\textbf{M}/\calB)(2)\cong HP^{0}(\textbf{M}/\calB)(0) \cong HP^{0}(M_b/\calB)(0)\cong HP^{0}((M_b\times \calB)/\calB)(0)  \\ 
&\cong  HP^{0}(M_b)(0)\otimes HP^{0}(\calB/\calB)(0) \cong HP^{0}(M_b)(0)\otimes B \cong H_{\mathrm{sing}}^{0}(M_b(\C))(0)\otimes B \cong B,
\end{align*}
where the first isomorphism uses the periodicity of $HP^*$ (which follows directly from definition), the fourth isomorphism 
uses the fact that $(M_b\times \calB)\to \calB$ factors as $M_b\times \calB\to \mathrm{Spec}\,\C\times \calB\cong \calB$, and the last one uses connectedness assumption on $M_b(\C)$.
Since all isomorphisms above are functorial in $\textbf{M}$, the  restriction map is an isomorphism:
\begin{equation}\label{equ on res iss}|_{\{I_Z\}}: HP^{-4}(\textbf{M}/\calB)(2)\stackrel{\cong}{\to} HP^{-4}(\{I_Z\}/\calB)(2). \end{equation}
Recall that $\Omega_{\textbf{Hilb}^{P_t}(\calX_{\calB}/\calB)/\calB}$  is obtained by applying the integration map $\int_{\calX_{\calB}/\calB}\mathrm{vol}_{\calX_{\calB}/\calB}\wedge(-)$  to $\widetilde{\ch_2}(I_{\calZ_\calB})$, where $I_{\calZ_\calB}$ is the  universal complex on $\calX_{\calB}\times_{\calB}\textbf{M}$, and $\widetilde{\ch_2}(I_{\calZ_\calB})$ is the compactly supported lift of ${\ch_2}(I_{\calZ_\calB})$ from the proof of Theorem \ref{thm on sft symp str}. The  restriction of 
$I_{\calZ_\calB}$ to $\calX_{b}\times_{}\{I_Z\}$ is $I_Z$. 
By the commutative diagram \eqref{comm diag on hp}, we have 
$$PC(\Omega_{\textbf{Hilb}^{P_t}(\calX_{\calB}/\calB)/\calB})|_{\{I_Z\}}=\int_{\calX_{b}}\mathrm{vol}_{\calX_{b}}\wedge\widetilde{\ch_2}(I_Z)|_{\calX_{b}/\calB}\in HP^{-4}(\{I_Z\}/\calB)(2)
. $$
Thanks to the isomorphism $|_{\{I_Z\}}$, to conclude the proof, 
 it is enough to show the vanishing of $\widetilde{\ch_2}(I_Z)$.
 %By Remark \ref{rmk on nonempty mod}, we know $$\widetilde{\ch_2}(I_Z)=i^*$$
 
Notice that the isomorphism \cite[Thm.~2.2]{Emm} is constructed on the chain-level. Taken into account  \cite[pp.~92,~Cor.~IV.~2.2]{Har}, we have an isomorphism
\[HP^k_Z(\calX_b)(p)\cong H_{\mathrm{sing}}^{k+2p}(\calX_b,\calX_b\setminus Z),  \,\,\, \forall\,\, k,p\in \mathbb{Z}. \]
Since $\dim_\bbC Z\leq 1$, we have $H_{\mathrm{sing}}^{4}(\calX_b,\calX_b\setminus Z)=0$.
Therefore 
\begin{equation*} 0=\widetilde{\ch_2}(I_Z)\in HP^{0}_Z(\calX_b)(2)=0. \qedhere  \end{equation*} 
\end{proof}

\begin{proof}[Proof of Proposition~\ref{iso cond 2}]
The obstruction theory \eqref{obs theory relW2} is obtained by the base-change of 
\begin{equation}\label{eqn:derived_obs}\bbE_{\Hilb^{P}(Y,D)/(W\times \calA^P)} \to \bbL_{\Hilb^{P}(Y,D)/(W\times \calA^P)} \end{equation}
along the map $Z(\phi)\hookrightarrow W$. Here \eqref{eqn:derived_obs} comes from the derived enhancement \eqref{equ on cop lag fib}. 
Similar to the proof of Theorem~\ref{iso cond 1}, we may replace $\calA^P$ by $\calB=\Spec B$ where $B$ is a local Artinian $\C$-algebra. 
For simplicity, we use the same notation for a derived $\calA^P$-stack (e.g.~$\calD, \calY$, $\textbf{Hilb}^{P}(Y,D)$ and universal substacks $\calZ_{Y,D}$, $\calZ$, $\calZ_{D}$) to
denote its base-change to $\calB$.  

By the relative Darboux theorem \cite[Thm.~B]{Park2},  it suffices to show  the composition of 
$$\Omega_{\textbf{Hilb}^{P}(Y,D)/(W\times \calB)}:\C\to NC_{}\left(\textbf{Hilb}^{P}(Y,D)/(\calB\times W)\right)[-4](2) $$ 
and the canonical map 
\[NC_{}\left(\textbf{Hilb}^{P}(Y,D)/(\calB\times W)\right)[-4](2)\to PC_{}\left(\textbf{Hilb}^{P}(Y,D)/(\calB\times W)\right)[-4](2)\]
 is null-homotopic.  
%$\HP^{-4}(\textbf{Hilb}^{P_t}(\calX_{\calB}/\calB)/\calB)(2)$ is zero. 

For this purpose, we describe $\Omega_{\textbf{Hilb}^{P}(Y,D)/(W\times \calB)}$. 
We introduce in this proof the following shorthands: 
$$M_{\calY}:=\textbf{Hilb}^{P}(Y,D), \,\,\, M_{D}:=\textbf{Hilb}^{P|_D}(D), \,\,\, M_{\calD}:=M_{D}\times \calB. $$
We have the following  commutative diagram of complexes:  
$$
%\xymatrix@C=10pt{ 
\footnotesize{
\xymatrix{
 &\C \ar[dl]_{\ev_{Y,D}^*\Omega_{\textbf{RPerf}}} \ar[dr]^{\ev_D^*\Omega_{\textbf{RPerf}}} & 
\\
NC_{\calZ_{Y,D}}\left((\calY\times_{\calB}M_{\calY})/\calB\right)(2) \ar[d]_{ }  \ar[r]^{i_{\calD\to \calY}^*}   & NC_{\calZ}\left((\calD\times_{\calB}M_{\calY})/\calB\right)(2)     \ar[d]_{ }  & NC_{\calZ_D}\left((\calD\times_{\calB}M_{\calD})/\calB\right)(2) \ar[l]_{r_{\calD\to \calY}^*} \ar[d]^{ } \\
NC_{\calZ_{Y,D}}\left((\calY\times_{\calB}M_{\calY})/(\calB\times W)\right) (2)\ar[r]^{i_{\calD\to \calY}^*}  \ar[d]_{ }  & NC_{\calZ}\left((\calD\times_{\calB}M_{\calY})/(\calB\times W)\right) (2)\ar[d]_{ }  & 
NC_{\calZ_{D}}\left((\calD\times_{\calB}M_{\calD})/(\calB\times W)\right) (2)\ar[l]_{r_{\calD\to \calY}^*}  
\ar[d]_{\int_{\calD/\calB}\mathrm{vol}_{\calD/\calB}\wedge(-)} \\
0  \ar[r]_{ }  & NC_{ }\left(M_{\calY}/(\calB\times W)\right)[-3](2)& \ar[l]_{r_{\calD\to \calY}^*} NC_{}\left(M_{\calD}/(\calB\times W)\right)[-3](2).  
}}
$$
The map 
\begin{equation}\label{equ on ncD}\C\to NC_{}\left(M_{\calY}/(\calB\times W)\right)[-3](2) \end{equation} has two null-homotopies. The shifted symplectic structure 
$$\Omega_{\textbf{Hilb}^{P}(Y,D)/(W\times \calB)}: \C\to NC_{}\left(M_{\calY}/(\calB\times W)\right)[-4](2)$$ 
is the loop formed by the concatenation of these two null-homotopies.  

Firstly recall the null-homotopy coming from the  Lagrangian structure Theorem \ref{thm on lag} on the map
$$\textbf{Hilb}^{P}(Y,D) \stackrel{r_{\calD\to \calY}}{\to} \textbf{Hilb}^{P|_D}(D)\times \calB. $$ 
The lower-left square in the diagram above fits into diagram \eqref{diag on nullhom}. In particular, the map $\bbC\to  NC_{ }\left(M_{\calY}/(\calB\times W)\right)[-3](2)$ factors through a contractible space $0$, and hence has a null-homotopy in this contractible space, which then projects to a null-homotopy in $NC_{ }\left(M_{\calY}/(\calB\times W)\right)[-3](2)$. Denote this null-homotopy by $\gamma_1$.

The other null-homotopy, denoted by $\gamma_2$, comes from the Lagrangian fibration structure (Assumption~\ref{ass on lag fib}) on the map 
\[\bCrit^{}(\phi)\to W.\]
We introduce a third null-homotopy of 
$\C\to PC_{}\left(M_{\calY}/(\calB\times W)\right)[-3](2)$, coming from the null-homotopy of a map \[PC(\Omega_{\bCrit^{}(\phi)}):\C\to PC_{}(\bCrit^{}(\phi))[-3](2), \]
as in \cite[Prop~3.27]{CZ}. By the same calculation on local Darboux chart as in \textit{loc}.\,\textit{cit.}, the loop formed by $\gamma_0$ and $\gamma_2$  in $PC_{ }\left(M_{\calY}/(\calB\times W)\right)[-3](2)$ is trivial after base-change along the map $Z(\phi)\hookrightarrow W$. 
Here and below, we denote the image of $\gamma_2$ under $PC: NC(-)(2)\to PC(-)(2)$ by the same notation.  Similar for $\gamma_1.$

To conclude the argument, it suffices to show that $\gamma_0$ and $\gamma_1$ form a trivial loop. 
Note that both  $\gamma_0$ and $\gamma_1$ come from null-homotopies in $PC_{ }\left(M_{\calY}/\calB\right)[-3](2)$, which are still denoted by $\gamma_0$ and $\gamma_1$ by an abuse of notation, 
so we are reduced to 
show the loop they form is trivial. 
For this purpose, we use a similar argument as in  the proof of Proposition~\ref{iso cond 1}. Let $b$ be the closed point in $\calB$ and $\calY_b$ the fiber of $\calY$ at the point $b$. Let $p=\{I_{Z_Y}\}$ be a closed point of $M_\calY$ sitting over $b$, where $Z_Y\hookrightarrow\calY_b$ is a closed subscheme. By an abuse of notations, we also denote $p$ to be its image 
under the composition $M_\calY \xrightarrow{r_{\calD\to\calY}} M_\calD \to M_D$, where the second map is the projection.
The $p$ in the image is represented by $I_{Z_D}$ with $Z_D=Z_Y\cap D$. 

Then we have the following commutative diagram
$$
\scalebox{.75}{
\xymatrix@C=10pt{
&&\bbC\ar[dl]\ar[drr]&\\
&PC_{\calZ_{Y,D}}(\calY\times_\calB M_\calY/\calB)(2)\ar[r]\ar[dl]\ar@{-->}@<-2ex>[dd]&PC_{\calZ}(\calD\times_\calB M_\calY/\calB)(2)\ar@{-->}@<-2ex>[dd]\ar[dl]&PC_{\calZ_D}(\calD\times_\calB M_\calD/\calB)(2)\ar@{-->}@<-2ex>[dd]\ar[l]&PC_{Z_D}(D\times M_D)(2) \ar@{-->}@<-2ex>[dd]\ar[l]^{\quad \cong}\ar[dl]\\
PC_{Z_{Y}}(\calY_b\times \{p\}/\calB)(2) \ar[r]\ar[dd]&PC_{Z}(D\times \{p\}/\calB)(2)\ar[dd]&& PC_{Z_D}(D\times \{p\})(2) \ar[dd]\ar[ll]&\\
&0 \quad \ar[r]\ar[dl]&PC(M_\calY/\calB)[-3](2)) \ar[dl]& PC(M_\calD/\calB)[-3](2) \ar[l] & PC(M_D)[-3](2) \ar[dl]\ar[l]_{\quad \cong}\\
0\ar[r]&PC(\{p\}/\calB)[-3](2)&&PC(\{p\})[-3](2)\ar[ll]
}}$$
where all the maps from the back side to the front side are restriction maps. 

By the same argument used in \eqref{equ on res iss}, the restriction from $M_\calY$ to $\{p\}$:
$$PC(M_\calY/\calB)[-3](2) \to PC(\{p\}/\calB)[-3](2)$$ 
induces an isomorphism on $\pi_1$. To show the loop formed by  $\gamma_0$ and $\gamma_1$ is trivial in $PC(M_\calY/\calB)[-3](2)$, it suffices to show the image of this loop in $PC(\{p\}/\calB)[-3](2)$ is trivial. 

%Let $I_{Z_Y}$ the sheaf on $\calY_b$ represented by $p\in M_\calY$, and $I_{Z_D}$ on $D$ represented by the image of $p$ in $M_\calD$ \yl{should be $M_D$} (i.e.~the $%restriction of $I_{Z_Y}$ to $D$).
In the diagram above, the composition 
$\bbC\to PC_{Z_D}(D\times \{p\})(2)$  is given by $\widetilde{\ch_2}(I_{Z_D})$, and $\bbC\to PC_{Z_{Y}}(\calY_b\times \{p\}/\calB)(2)$
 is given by $\widetilde{\ch_2}(I_{Z_Y})$.
Similar to 
Prop~\ref{iso cond 1}, $\widetilde{\ch_2}(I_{Z_Y})=0$. Taking its image in $PC(\{p\}/\calB)[-3](2)$, we obtain yet another null-homotopy of $\int_D\mathrm{vol}_D\wedge\widetilde{\ch_2}(I_{\calZ})|_{\{p\}}$ when restricted to $\{p\}$. Notice that on the one hand, this null-homotopy factors through the contractible space 0 at the bottom left corner 
%\yl{this factorization has nothing to do with how $\widetilde{\ch_2}(I_{Z_Y})$ is hom to zero}  
and hence is homotopic in $PC(\{p\}/\calB)[-3](2)$ to the image of $\gamma_1$. On the other hand, restricting to $D$ gives $\widetilde{\ch_2}(I_{Z_D})=0$ and hence  a null-homotopy of the restriction of $PC(\Omega_{M_D})$ in $PC(\{p\})[-3](2)$. 
We claim that this null-homotopy is homotopic to the restriction of $\gamma_0$ in $PC(\{p\})[-3](2)$. 
Taking image  in $PC(\{p\}/\calB)[-3](2)$ of the homotopy from the claim, we concludes that $\gamma_0$ and $\gamma_1$ are homotopic. 

It remains to show the claim. 
Using the presentation of $PC(M_D)[-3](2)\to PC(\{p\})[-3](2)$  given by \cite[Ex.~2.8]{BBJ}, $\gamma_0$ goes to $\phi|_{\{p\}}$ which is 0 by the footnote of Assumption~\ref{ass on lag fib}. 
Similarly, recall that the null-homotopy $\widetilde{\ch_2}(I_{Z_D})\sim 0$ is constructed using the equivalence $PC_{Z_D}(D)(2)\simeq DR_{Z_D}(D)[4]$. Hence, we do calculations in  $DR_{Z_D}(D)$, where we have the map 
\[
 DR_{Z_D}(D)[4]\xrightarrow{\int_D \mathrm{vol}_D\wedge(-)}DR(\{p\})[1]\simeq
PC(\{p\})[-3](2).
\]
Any de Rham differential form $\beta$ on $D$ with $d_{dR}\beta=\widetilde{\ch_2}(I_{Z_D})$, has to be a sum of $(2,1)$ and $(1,2)$-forms and hence goes to 0 under $\int_D  \mathrm{vol}_D\wedge(-)$. 
\end{proof}
\providecommand{\bysame}{\leavevmode\hbox to3em{\hrulefill}\thinspace}
\providecommand{\MR}{\relax\ifhmode\unskip\space\fi MR }
\providecommand{\MRhref}[2]{%
 \href{http://www.ams.org/mathscinet-getitem?mr=#1}{#2}}
\providecommand{\href}[2]{#2}

\end{document}